\documentclass[11pt,reqno,a4paper]{amsart}
\usepackage{amsmath,amsthm,verbatim,amscd,amssymb,enumitem,hyperref}
\usepackage{exscale,color}
\usepackage[utf8]{inputenc}

\renewcommand{\epsilon}{\varepsilon}
\newcommand{\N}{\mathbb{N}}
\newcommand{\Z}{\mathbb{Z}}

\newcommand{\R}{\mathbb{R}}
\newcommand{\C}{\mathbb{C}}
\renewcommand{\P}{\mathbb{P}}
\newcommand{\Tor}{\mathbb{T}}

\renewcommand{\leq}{\leqslant}
\renewcommand{\geq}{\geqslant}

\newtheoremstyle{fancy}{}{}{\itshape}{}{\textbf\bgroup}{.\egroup}{ }{}
\newtheoremstyle{fancy2}{}{}{\rm}{}{\textbf\bgroup}{.\egroup}{ }{}

\newcounter{mtheorem}
\newtheorem{mtheorem}[mtheorem]{Theorem}

\setlist{leftmargin=*}

\numberwithin{equation}{section}

\numberwithin{equation}{section}

\theoremstyle{fancy}
\newtheorem{theorem}[equation]{Theorem}
\newtheorem{lemma}[equation]{Lemma}
\newtheorem{prop}[equation]{Proposition}
\newtheorem{corollary}[equation]{Corollary}
\newtheorem{obs}[equation]{Observation}

\theoremstyle{fancy2}
\newtheorem{definition}[equation]{Definition}
\newtheorem{ex}[equation]{Example}

\theoremstyle{remark}
\newtheorem{remark}[equation]{Remark}
\newtheorem*{remark*}{Remark}

\usepackage{xspace}
\usepackage{slashed}

\DeclareMathOperator{\im}{\textup{Im}}

\newcommand{\sorth}[1]{\textup{SO}(#1)}
\newcommand{\sunitary}[1]{\textup{SU$(#1)$}}
\newcommand{\sunitaryn}{\textup{SU$(n)$}}

\newcommand{\quat}[1]{\textup{Sp$(#1)$}}
\newcommand{\Hol}{\operatorname{Hol}}
\newcommand{\Sph}{\mathbb{S}}

\newcommand{\norm}[1]{\left\|#1\right\|}

\newcommand{\tsfrac}[2]{\textstyle\frac{#1}{#2}}
\newcommand{\half}{{\tsfrac{1}{2}}}

\newcommand{\bbz}{\mathbb{Z}}
\newcommand{\bbr}{\mathbb{R}}
\newcommand{\bbc}{\mathbb{C}}

\newcommand{\bbrp}{\mathbb{R}^{+}}

\newcommand{\ie}{\emph{i.e.} }
\newcommand{\eg}{\emph{e.g.} }
\newcommand{\Eg}{\emph{E.g.} }
\newcommand{\cf}{\emph{cf.} }

\newcommand{\aalg}{\mathfrak{a}}

\newcommand{\holda}[1]{C^{#1, \alpha}}
\newcommand{\holdad}[1]{C^{#1, \alpha}_{\delta}}
\newcommand{\harm}{\mathcal{H}}

\newcommand{\calr}{\mathcal{R}}
\newcommand{\calo}{\mathcal{O}}
\newcommand{\cala}{\mathcal{A}}

\newcommand{\anglen}{\theta}
\newcommand{\cyl}{\infty}
\newcommand{\dirac}{\slashed{\partial}}

\newcommand{\gen}[1]{\langle#1\rangle}

\newcommand{\contra}[1]{\textstyle\frac{\partial}{\partial #1}}
\newcommand{\pd}[2]{\frac{\partial #1}{\partial #2}}

\DeclareMathOperator{\iso}{Iso}
\newcommand{\bd}{\rm bd}
\newcommand{\acyl}{{\rm ACyl}\xspace}

\newcommand{\db}{\bar\partial}
\newcommand{\oM}{\hspace*{0pt} \mkern 4mu\overline{\mkern-4mu M\mkern-1mu}\mkern 1mu}
\newcommand{\oD}{\hspace*{0pt} \mkern 4mu\overline{\mkern-4mu D\mkern-1mu}\mkern 1mu}
\newcommand{\tM}{\tilde M}
\newcommand{\tX}{\tilde X}

\newcommand{\tomega}{\tilde{\omega}}

\begin{document}

\title{Asymptotically cylindrical Calabi-Yau manifolds}
\author[M.~Haskins]{Mark~Haskins}
\address{Department of Mathematics, Imperial College London, London SW7 2AZ, United Kingdom}
\email{m.haskins@imperial.ac.uk}
\author[H.-J.~Hein]{Hans-Joachim Hein}
\address{Department of Mathematics, University of Maryland, College Park, MD 20742--4015, USA}
\email{hein@umd.edu}
\author[J.~Nordstr\"om]{Johannes Nordstr\"om}
\address{Department of Mathematical Sciences, University of Bath, Bath BA2 7AY, United Kingdom}
\email{j.nordstrom@bath.ac.uk}
\date{\today}

\begin{abstract}
Let $M$ be a complete Ricci-flat K\"ahler manifold with one end and assume that this end converges at an exponential rate to $[0,\infty) \times X$
for some compact connected Ricci-flat manifold $X$.
We begin by proving general structure theorems for $M$; 
in particular we show that there is no loss of generality in assuming that $M$ is 
simply-connected and irreducible with $\Hol(M) = {\rm SU}(n)$, where 
$n$ is the complex dimension of $M$. If $n > 2$ we then show that there exists a projective orbifold $\oM$ and a divisor $\oD \in |{-K_{\oM}}|$ with torsion normal bundle such that $M$ is biholomorphic to $\oM \setminus \oD$, thereby
settling a long-standing question of Yau in the asymptotically cylindrical setting. 
We give examples where $\oM$ is not smooth: the existence of such examples appears not to have been noticed previously.
Conversely, for any such pair $(\oM, \oD)$ we give a short and self-contained proof of the existence and uniqueness of 
exponentially asymptotically cylindrical Calabi-Yau metrics on $\oM \setminus \oD$. 
\end{abstract}
\maketitle

\thispagestyle{empty}

\section{Introduction}\label{s:intro}

\subsection*{Background and overview} In one of their foundational papers on complete Ricci-flat
K\"ahler metrics \cite[Cor 5.1]{tianyau90} Tian and Yau proved the existence of such metrics with linear volume growth 
on smooth noncompact quasi-projective varieties of the form $M=\oM \setminus \oD$,
where $\oM$ is a smooth projective variety that fibres over a Riemann surface with generic fibre $\oD$ a connected smooth and reduced
anticanonical divisor.
In fact, the estimates of \cite{tianyau90} imply that the end of $M$ is bi-Lipschitz equivalent to one half of a metric cylinder $M_\infty = \R \times X$
where $X= \Sph^{1} \times \oD$ and $\oD$ is endowed with a Ricci-flat K\"ahler metric that exists because
$c_{1}(\oD) =0$ by adjunction \cite{yau78}.

The current paper has two principal goals:
\begin{enumerate}
\item
\label{goal1}
To give a short and self-contained proof of a generalised and refined
version of the Tian-Yau theorem; 
as one consequence of this generalisation we obtain asymptotically cylindrical Ricci-flat
K\"ahler metrics whose cross-section $X$ no longer takes the split form $\Sph^{1} \times \oD$; one
of our refinements is to establish the exponential convergence of $M$ to $[0,\infty) \times X$.
\item
\label{goal2}
To show that every complete Ricci-flat K\"ahler manifold of complex dimension $n > 2$ that is exponentially asymptotic
to a half-cylinder $[0,\infty) \times X$ 
arises from our generalisation of the Tian-Yau construction in \ref{goal1}.
\end{enumerate}

The exponential convergence in \ref{goal1} is important because 
it is used in an essential way in the so-called \emph{twisted connected sum} construction of 
compact Riemannian $7$-manifolds with holonomy group $G_{2}$ 
\cite{chnp1,chnp2,kovalev03}, 
first suggested by Donaldson and then pioneered 
by Kovalev in \cite{kovalev03}.
At present no complete proof of the existence of exponentially asymptotically cylindrical Ricci-flat K\"ahler metrics exists in the literature; \cf Section \ref{s:existence}. 
Moreover, the original existence proof with bi-Lipschitz control due to Tian and Yau \cite{tianyau90} is difficult and very general; we will show that the 
asymptotically cylindrical case allows for a short and direct treatment, bypassing most of the technicalities of \cite{tianyau90}.

\ref{goal2} fits naturally into the broader framework of complex analytic compactifications of 
complete Ricci-flat K\"ahler manifolds---a topic Yau 
raised in his 1978 ICM Address \cite[p. 246, 2nd question]{yau:ICM78}.
Indeed, under the assumption of finite topology
all currently known constructions of such manifolds yield examples that are 
complex analytically compactifiable in Yau's sense.
In other settings 
some compactification results have been proven 
by studying the section ring of the (anti-)canonical bundle
---in \cite{mok-zhong} for ${\rm Ric} < 0$ with finite volume  and
 in  \cite{mok} for ${\rm Ric} > 0$ with Euclidean volume growth---but we are not aware of any such results in the Ricci-flat case
even under additional hypotheses.

In this paper we develop a new approach to constructing compactifications by
exploiting detailed asymptotics for the metric at infinity. To state the basic idea,
let $M$ be a complete Ricci-flat K\"ahler manifold with one end that converges at an exponential rate to one half of a metric cylinder $M_{\cyl} = \R \times X$. We begin by proving that after passing to a finite cover and splitting off compact factors we can assume that $M$ is simply-connected of holonomy SU$(n)$ with $n = \dim_\C M$. If $n > 2$, we will then prove that $M_\cyl$ has a finite cover that splits as a K\"ahler product $\R \times \Sph^{1}\times D$,
where $D$ is compact Ricci-flat  K\"ahler. The cylinder $M_{\infty}$ now admits a natural orbifold
compactification, so we can try to use the fact that $M$ is asymptotic to $M_{\infty}$ 
to build an orbifold compactification of $M$. This is indeed possible but requires significant technical work: see Section \ref{sec:cptfy}.

\subsection*{Basic terminology} Before proceeding to a more detailed description of the main results and the organisation of the paper, we begin with a few basic definitions and remarks.

\begin{definition}\label{d:ACyl}
A complete Riemannian manifold $(M, g)$ is called \emph{asymptotically cylindrical}
(\emph{ACyl}) if there exist a bounded domain $U \subset M$, a closed 
(not necessarily connected) Riemannian manifold $(X, h)$, and a diffeomorphism $\Phi: [0,\infty) \times X \to M \setminus U$ such that $|\nabla^k(\Phi^*g - g_\infty)| = O(e^{-\delta t})$ with respect to the product metric $g_\infty \equiv dt^2 + h$ for some $\delta > 0$ and all $k \in \N_0$. Here $t$ denotes projection onto the $[0,\infty)$ factor; we often extend the function $t \circ \Phi^{-1}$ by zero and refer to this extension as a \emph{cylindrical coordinate function} on $M$.
We call the connected components of $M_{\infty}\equiv \R \times X$ endowed with the
product metric $g_\infty$ the \emph{asymptotic cylinders} (or sometimes the
\emph{cylindrical ends}), $(X,h)$ the \emph{cross-section}, and $\Phi$ the \emph{ACyl diffeomorphism} or 
\emph{ACyl map} of the ACyl manifold $(M,g)$.
\end{definition}

We will often suppress the map $\Phi$ in our notation, or tacitly replace it by $\Phi \circ [(t,x) \mapsto (t+t_0, x)]$ for some large constant $t_0$.
Also, it will be irrelevant whether we measure norms of tensors on $M \setminus U$
with respect to $g$ or $g_{\infty}$. Finally, we remark that exponential asymptotics are a priori more natural than 
polynomial or even weaker ones because solutions to linear elliptic equations on cylinders tend to behave exponentially. The Calabi-Yau condition is not linear, but we obtain a consistent theory within the exponential setting; see also
the Concluding Remarks at the end of this section.

\begin{remark}\label{r:basic}
We will mainly be interested in ACyl manifolds that are \emph{Ricci-flat}. In this case:
\label{R:acyl:rf}
\begin{enumerate}
\item
\label{it:split}
$M$ has only a single end except when it is isometric to a product cylinder. This is an immediate consequence of the Cheeger-Gromoll splitting theorem \cite[Thm 2]{cheeger71}, and holds even if we assume only ${\rm Ric} \geq 0$.
From now on in this remark, assume $M$ is not a product cylinder.
\item
\label{it:toruscover}
The end $M_{\infty}$ is a Ricci-flat cylinder, so the
cross-section $X$ is compact connected and Ricci-flat.
We recall a basic structure result: there exists a finite Riemannian covering \mbox{$\Tor \times X' \rightarrow X$} where $\Tor$ is a flat torus with $\dim \Tor \ge b^{1}(X)$ and $X'$ is compact simply-connected  and Ricci-flat
\mbox{\cite[Thm 4.5]{fischer}}. 
This is deduced from a more general theorem for ${\rm Ric} \geq 0$
\cite[Thm~3]{cheeger71},  but uses the inequality ${\rm Ric} \leq 0$ in an
essential way to ascertain that all Killing fields are parallel.
\end{enumerate}
\end{remark}

We also need to recall some terminology related to holonomy groups. 
We say that $(M,g)$ is \emph{locally irreducible} if the representation of the
restricted holonomy group $\Hol_0(M)$ on the tangent space of any point of $M$ is irreducible; by de Rham's theorem
this is equivalent to $M$ being locally irreducible in the sense of isometric product decompositions.
We call $(M^{2n},g)$ \emph{Calabi-Yau} if $\Hol(M) \subseteq \sunitaryn$ and \emph{hyper-K\"ahler} if $n$ is even and $\Hol(M) \subseteq \quat{\tfrac{n}{2}} \subset \sunitaryn$. The Calabi-Yau condition implies that $M$ is Ricci-flat Kähler. Conversely, if $M$ is Ricci-flat Kähler then $\Hol(M) \subseteq {\rm U}(n)$ and 
$\Hol_0(M) \subseteq \sunitaryn$, so if $M$ is simply-connected
then it is Calabi-Yau, and if additionally $M$ is irreducible then---by Berger's classification---either $\Hol(M) = \sunitaryn$, or $n$ is even and $\Hol(M) = \quat{\tfrac{n}{2}}$.

A final point of notation: $\Sph^k$ will denote a round $k$-sphere and $\Tor^k$ a flat $k$-torus (not necessarily a product of $k$ circles). Thus $\Sph^1 = \Tor^1$ is a circle but we do not specify its radius. However, we always identify $\Sph^1 = \R/2\pi\Z$ topologically and denote the resulting angular coordinate on $\Sph^1$ by $\theta$.

\subsection*{Killing the fundamental group} Our first main result gives an ACyl analogue of the structure theorem
for compact Ricci-flat manifolds of Remark \ref{R:acyl:rf}\ref{it:toruscover}. 
This again follows from
 a structure result for (ACyl) manifolds with nonnegative Ricci curvature: Theorem \ref{thm:acyl_nonneg}.

\begin{mtheorem}
\label{thm:acyl_structure}
Every Ricci-flat {\rm ACyl} manifold has a finite normal
covering space that splits as the isometric product of a flat torus and
a simply-connected Ricci-flat {\rm ACyl} manifold.
\end{mtheorem}

In particular, if $M$ is ACyl Ricci-flat {K\"ahler}, then $M$ has a finite normal covering space $\tilde{M}$ such that $\tilde{M} = M' \times N$, where $M'$ is simply-connected irreducible ACyl, $N$ is compact, and both $M'$ and $N$ are K\"ahler except in the trivial case where $M' = \R$. Thus, for almost all purposes we can assume without any loss that the full holonomy of $M$ is either SU$(n)$ or Sp$(\frac{n}{2})$ (some care
must be taken \eg in establishing projectivity 
of complex analytic compactifications in Theorem \ref{t:cptfy} because of the potential presence of  non-projective compact factors in the splitting above.)

\subsection*{Holonomy and the asymptotic cylinder} We will assume from now on that our Ricci-flat ACyl manifold $M$ is K\"ahler of complex dimension $n$.
Our next main result---Theorem \ref{T:acyl:holo}, to be proved in Section \ref{s:acyl_holonomy}---shows that $\R \times X$ being the asymptotic cylinder of a Ricci-flat K\"ahler manifold imposes strong additional restrictions on $X$ beyond $\R \times X$ being Ricci-flat K\"ahler; see \ref{T:acyl:holo}\ref{cylstruc}. In particular, $b^1(X) = 1$ if $n > 2$. This is consistent with \ref{T:acyl:holo}\ref{nohk} because ${\rm Hol}(M) = {\rm Sp}(\frac{n}{2})$ implies that $b^1(X) \geq 3$. However, we will prove Theorem \ref{T:acyl:holo} by treating the two cases ${\rm Hol}(M) = {\rm Sp}(\frac{n}{2})$ and $\Hol(M)$ $=$ $\sunitaryn$ in parallel, using the same type of argument to derive restrictions on $X$ in both cases.

\begin{mtheorem}
\label{T:acyl:holo}
Let $M$ be simply-connected irreducible {\rm ACyl} Calabi-Yau with $n = \dim_\C M > 2$.
\begin{enumerate}
\item\label{nohk}
$M$ is not hyper-K\"ahler, or in other words $\Hol(M) = \sunitaryn$.
\item\label{cylstruc}
There exists a compact Calabi-Yau manifold $D$ with a K\"ahler isometry $\iota$ of finite order $m$ such that the cross-section $X$ of $M$ can be written as $X = (\Sph^1 \times D)/\langle \iota\rangle$,
where $\iota$ acts on the product via $\iota(\theta,x) = (\theta + \frac{2\pi}{m}, \iota(x))$. Moreover, $\iota$ preserves the holomorphic volume form on $D$ but no other holomorphic forms of positive degree. In particular, $b^1(X) = 1$.
\end{enumerate}
\end{mtheorem}

The case $n = 2$ is exceptional in several respects---the main reason being that  SU$(2)$ = Sp$(1)$, so that Calabi-Yau and hyper-K\"ahler coincide in complex dimension $2$---and we will not say very much about it here. ACyl examples do exist but their asymptotic cylinders need not be finite quotients of a product $\R \times \Sph^1 \times D$; see Remark \ref{r:4dimrem} for some more details in this direction.

For another immediate clarification, let us point out that the order $m$ of the K\"ahler isometry $\iota$ of \ref{T:acyl:holo}\ref{cylstruc} really can be greater than $1$ even though $\pi_1(M) = 0$; see Examples \ref{ex:end} and \ref{ex:dihedral}, both of which are $3$-dimensional. This possibility seems not to have been observed previously. In particular, such
examples do not fit within the remit of the known constructions \cite{kovalev03,kovalev-lee08} based on \cite{tianyau90}.

\begin{remark}\label{r:endrestrictions}
We now take a closer look at the restrictions on $M_\cyl$ imposed by \ref{T:acyl:holo}\ref{cylstruc}.
\vspace{-0.1em}
\begin{enumerate}
\item
\label{it:3d}
If $n = 3$ then $D$ could be $\Tor^4$ or $K3$, but not a finite quotient of either; in Examples \ref{ex:end} and \ref{ex:dihedral} we show that both occur (with $m > 1$). 
In both cases there are strong a priori restrictions on the possible values of $m$:
if $D=\Tor^{4}$ then $m\in \{2, 3, 4, 6\}$ by \cite[Lemma 3.3]{fujiki}, 
while if $D=K3$ then $m \le 8$ (and the number of fixed points of $\iota$ depends only on $m$) by \cite[\S 0.1]{mukai:finite:k3} or \cite{nikulin80}.
\item\label{it:end_restr_2}
If $m = 1$, then $h^{p,0}(D) = 1$ for $p \in \{0, n-1\}$ but $h^{p,0}(D) = 0$ otherwise.  
Thus, if $n = 3$ then $D = K3$.  
Also if $\pi_1(D) = 0$ then $\Hol(D) = {\rm SU}(n-1)$; 
in general $D$ could be locally reducible though: 
$D = (K3 \times K3)/\Z_2$ is not ruled out if $\Z_2$ acts anti-symplectically on each factor, \ie
as a holomorphic involution of $K3$ that changes the sign of the holomorphic volume form.
\end{enumerate}
\end{remark}

Theorem \ref{T:acyl:holo}\ref{cylstruc} is important for the compactification problem in view of the following
\begin{quotation}
\emph{Compactification ansatz}: A complex product cylinder
$\bbr \times \Sph^1 \times D \cong \bbc^* \times D$ can be compactified as 
$\bbc \times D$. If $D$ has a holomorphic volume form $\Omega_D$, then 
$(dt + id\theta) \wedge \Omega_D$ extends to a meromorphic volume form with a
simple pole along $\{0\} \times D$. 
\end{quotation}
Thus \ref{T:acyl:holo}\ref{cylstruc} implies that $M_{\infty}$ is biholomorphic to
the complement of $(0 \times D)/\bbz_m$ in $(\bbc \times D)/\bbz_m$. It is therefore natural to allow for \emph{orbifold} compactifications: if $n$ is odd and if $D$ has no holomorphic forms except in degrees $0$ and $n-1$, then the holomorphic Lefschetz formula tells us that $\iota$
acting on $D$ must have fixed points,
so the compactification of $M_{\infty}$ is definitely not smooth if $m > 1$.

If $M$ is an \emph{arbitrary} ACyl K\"ahler manifold, then the orbits of the parallel vector field $J\partial_t$ on $M_\infty$ have no 
reason to split off as isometric $\Sph^1$-factors in any finite cover, so the compactification ansatz above may not apply.
This does not mean that $M_\cyl$ is not holomorphically compactifiable, but the construction of a compactification could then be much more complicated; \cf Remark
\ref{r:4dimrem}.

\subsection*{A compactification theorem} In Section \ref{sec:cptfy} we will prove that any ACyl K\"ahler manifold $M$ that satisfies the conclusion of Theorem \ref{T:acyl:holo}\ref{cylstruc} has an orbifold holomorphic compactification $\oM$ modelled
on the holomorphic compactification of $M_\infty$ discussed above. Somewhat surprisingly, this is \emph{not} an immediate consequence of the ACyl asymptotics and indeed requires significant technical work; \cf the introduction to Section \ref{s:hol_cptf}. Further technical work shows that $\oM$ is K\"ahler, and if $M$ is Calabi-Yau then $\oM$ is projective. 
Thus, our results are most comprehensive if $M$ satisfies the assumptions of Theorem \ref{T:acyl:holo}; for simplicity we give the statement only in this case.

\begin{mtheorem}\label{t:cptfy}
Let $M$ be simply-connected irreducible {\rm ACyl} Calabi-Yau
of complex dimension $> 2$. Let $X$, $D$,  $\iota \in {\rm Isom}({D})$, and $m$ be as in Theorem \textup{\ref{T:acyl:holo}\ref{cylstruc}} and define $\oD = D/\langle \iota\rangle$. Then with respect to either of the two parallel complex structures on $M$ we have: 
\begin{enumerate}
\item There exists a projective orbifold $\oM$ with $h^{p,0}(\oM) = 0$ for all $p > 0$ and vanishing plurigenera such that $\oD \in |{-K_{\oM}}|$ is an orbifold divisor and $M$ is biholomorphic to $\oM \setminus \oD$. The orbifold normal bundle to $\oD$ in $\oM$ is biholomorphic to $(\C \times D)/\langle \iota\rangle$ as an orbifold line bundle. Thus, if
$m = 1$ then $\oM$ is smooth and  the normal bundle of $\oD$ is holomorphically trivial.

\item The {\rm ACyl} K\"ahler form is
cohomologous to the restriction to $M$ of a K\"ahler form on $\oM$.

\item\label{it:fibre} If $b^1({D}) = 0$ then the linear system $|m\oD|$ is a pencil on $\oM$, defining a fibration $\oM \to \mathbb{P}^1$ with $\oD$ as an $m$-fold fibre. In particular this holds for $m = 1$ since $b^1(X) = 1$ by Theorem \textup{\ref{T:acyl:holo}\ref{cylstruc}}.
\end{enumerate}
\end{mtheorem}

Before discussing the statement of Theorem \ref{t:cptfy} in more detail, let us indicate the basic strategy of the proof when $m = 1$.
Given a smooth divisor $\oD$ in a complex manifold $\oM$ whose normal bundle is trivial as a smooth complex line bundle, there exist exponential maps sending the fibres of the normal bundle to holomorphic disks in $\oM$. In proving Theorem \ref{t:cptfy}, we first  construct a ``punctured version'' of such an exponential map purely within $M$.
By studying $\bar{\partial}$-equations along the resulting punctured holomorphic disks in $M$, we will then be able to prove that the complex structure of $M$ is sufficiently regular at infinity to admit a holomorphic compactification $\oM$.

\begin{ex}
\label{ex:end}
To further illustrate the $m > 1$ case of Theorem \ref{t:cptfy}, we describe a simply-connected 
irreducible ACyl Calabi-Yau $3$-fold where $D$ is a torus and $m = 2$. This space is closely related to a Kummer construction due to Joyce; see \cite[7.3.3(iv)]{jnthesis}.

Let $E$ be an elliptic curve and let
$\oM_0 = (\P^1 \times E \times E)/\gen{\alpha,\beta}$,
where $\alpha$ and $\beta$ act on $\P^1$ as the 
commuting holomorphic involutions $z \mapsto \frac{1}{z}$ and $z \mapsto -\frac{1}{z}$,
and on $E \times E$ as $(-1,1)$ and $(1,-1)$. Let $\oM$ be the blow-up of
$\oM_0$ at the fixed sets of $\alpha$ and $\beta$ (these have complex
codimension~2). The fixed points of $\iota = \alpha\beta$ become orbifold
singularities in $\oM$ contained in the image $\oD \cong (E {\times} E)/\{\pm 1\}$
of $\{0, \infty\} \times E \times E$. Since $\{0, \infty\}$ is an anticanonical
divisor on $\P^1$ and the blow-up is crepant, $\oD$ is an anticanonical orbifold
divisor on $\oM$ (``two cylindrical ends folded into one'').

We can deduce from Theorem \ref{t:geom-exi} that $M = \oM \setminus \oD$ admits ACyl Calabi-Yau metrics. 
However, we can also think of $M$ as a blow-up of the flat orbifold
\[
M_0 = (\bbr \times \Sph^1 \times E \times E)/\gen{\alpha,\beta}
\] and obtain
ACyl Calabi-Yau metrics by a generalised Kummer construction \cite[7.3.3(iv)]{jnthesis}.
Because $\gen{\alpha,\beta}$ is generated by elements with fixed points,
the argument of \cite[\S 12.1.1]{joyce00} can be used to prove that
$\pi_1(\bbr \times \Sph^1 \times E \times E) \to \pi_1(M_0)$ is surjective,
and that $M_0$ and $M$ are simply-connected.
This model for $M$ also makes it easy to see that
the cross-section $X$ is the quotient of
$\Sph^1 \times E \times E$ by the fixed-point free involution
$(\theta, x, y) \mapsto (\theta+\pi, -x, -y)$; in particular, $b^1(X) = 1$ in accordance with Theorem \ref{T:acyl:holo}\ref{cylstruc}
 since the only $\Z_2$-invariant parallel $1$-form upstairs is $d\theta$.
\end{ex}

\begin{remark}\label{r:fibration}
We now make some basic comments about the fibration in Theorem \ref{t:cptfy}\ref{it:fibre}.
\begin{enumerate}
\item
No compact complex manifold with finite fundamental group can fibre over
a Riemann surface with non-zero genus, since then the lift of the fibering map
to the universal cover would be a non-constant holomorphic function from a
compact complex manifold to $\bbc$.
\item 
We can compare the conclusions of Theorems \ref{T:acyl:holo}\ref{nohk} and \ref{t:cptfy}\ref{it:fibre} with the following observation due to Matsushita \cite[Lemma 1(2)]{matsushita}: if $M$ is a \emph{compact} K\"ahler manifold of holonomy ${\rm Sp}(\frac{n}{2})$, $n = \dim_\C M$, and if  $f: M \to B$ is a surjective holomorphic map onto a K\"ahler manifold $B$ of complex dimension $0 < b < n$, then $b = \frac{n}{2}$. 
(In this situation, a much more difficult result due to Hwang \cite{hwang} then asserts that $B$ is projective space if both $M$ and $B$ are algebraic; these algebraicity hypotheses have very recently been removed by Greb and Lehn \cite{greb13}.)
\item\label{it:nonfibration} We do
not know whether or not $|m\oD|$ still defines a fibration of $\oM$ over $\P^1$
if $b^1({D}) > 0$ (hence necessarily $m > 1$). In this direction, observe that composing the projection $\P^1 \times E \times E \to \P^1$ in Example \ref{ex:end} with a degree $4$ map
$\P^1 \to \P^1$ invariant under $\gen{\alpha,\beta}$ yields a fibration 
$\oM \to \P^1$ corresponding to $|2\oD|$.
Now $M$ admits nontrivial ACyl Calabi-Yau
deformations with the same cylindrical end as $M$; it is not clear to us
whether or not these are still fibred by $|2\oD|$.
\end{enumerate}
\end{remark}

\begin{remark}\label{r:4dimrem}
The compactification question for $n = 2$ is more subtle. To begin with, we have $X = \Tor^3$ since
$\Hol(\R \times X) \not\subseteq \sunitary2$ if $X$ is a proper quotient of
$\Tor^3$ (but all orientable proper quotients of $\Tor^3$ do arise as
cross-sections of \emph{locally} hyper-K\"ahler ACyl $4$-manifolds with
nontrivial $\pi_1$ \cite[Thm 0.2]{BM}). By \cite[Thm 1.10]{hein}, $X$ need not be an isometric product $\Sph^1 \times \Tor^2$, and by extending the construction of \cite{hein} one can show that \emph{every} flat
torus $\Tor^3$ occurs as a cross-section. Thus, for a generic choice of
hyper-K\"ahler metric or parallel complex structure $J$, the orbits of
$J\partial_t$ do not split off as isometric $\Sph^1$-factors in any finite
cover of $X$, and our compactification ansatz does not apply.

It is nevertheless possible to compactify $M_\infty$ holomorphically, strongly suggesting that $M$ itself can be compactified so that $\oM$ is 
$\P^2$ blown up in $9$ {general} points, $\oD$ is the proper transform of the \emph{unique} cubic passing through these points, and $|\oD|$ is trivial. By contrast, the 
construction in \cite{hein} is based on \emph{pencils} of cubics in $\mathbb{P}^2$. We plan to discuss the details of this picture elsewhere.
\end{remark}

\subsection*{Existence and uniqueness of ACyl Calabi-Yau metrics} 
Our final main result both extends the Tian-Yau existence theorem for Ricci-flat K\"ahler metrics of linear volume growth \cite[Cor 5.1]{tianyau90} to a natural level of generality and establishes exponential asymptotics for these metrics. We also have a basic uniqueness result in this context (Theorem \ref{t:uniqueness}).

\begin{mtheorem}\label{t:geom-exi}
Let $\oM$ be a compact K\"ahler orbifold of complex dimension $n \geq 2$. Let $\oD \in |{-}K_{\oM}|$ be an effective orbifold divisor satisfying the following two conditions:
\begin{enumerate}
\item The complement $M = \oM\setminus\oD$ is a smooth manifold. 
\item The orbifold normal bundle of $\oD$ is biholomorphic to
$(\C \times D)/\langle \iota \rangle$ as an orbifold line bundle, where $D$ is a connected compact
complex manifold and $\iota$ is a complex automorphism of $D$ of order $m < \infty$ acting on the product via $\iota(w,x) = (\exp(\frac{2\pi i}{m})w, \iota(x))$.
\end{enumerate}
Let $\Omega$ be a meromorphic $n$-form on $\oM$ with a simple pole along $\oD$. For every orbifold K{\"a}hler class $\mathfrak{k}$ on $\oM$ there exists an {\rm ACyl} Calabi-Yau
metric $\omega$ on $M$ such that $\omega \in \mathfrak{k}|_M$ and
$\omega^n = i^{n^2}\Omega \wedge \bar{\Omega}$.
\end{mtheorem}

\begin{remark}
\label{r:geom-exi:orbi}
We can describe the ACyl geometry of $(M,\omega)$ more precisely.
\begin{enumerate}
\item The cross-section of $(M,\omega)$ is isometric to 
$(\Sph^1 \times D)/\langle \iota \rangle$. Here $D$ is equipped with the unique $\iota$-invariant Ricci-flat K\"ahler metric representing the pullback of $\mathfrak{k}|_{\oD}$, where we observe that $\oD$ has trivial canonical bundle by adjunction so that the Calabi-Yau theorem \cite{yau78} applies. The
 length of the $\Sph^{1}$-factor is determined by the choice of a meromorphic volume form $\Omega$, which is unique only up to a scalar factor (and is independent of the choice of a K\"ahler class $\mathfrak{k}$).

\item The ACyl map $\Phi : \bbrp \times (\Sph^1 \times D)/\langle \iota\rangle \to M$
is obtained by composing a suitable exponential map, $\exp$, on the normal bundle of $\oD$ with the complex exponential function $\R^+ \times \Sph^1 \to \C^*$. The precise construction of $\exp$ is somewhat involved and relies on Appendix \ref{app:trivial_nb}.
\end{enumerate}
\end{remark}

\begin{remark}
The original Tian-Yau construction \cite{tianyau90} concerns the special case of Theorem \ref{t:geom-exi} where
$\oM$ is a projective manifold fibred by the linear system $|\oD|$. This is not general enough to cover all
possible pairs 
$(\oM,\oD)$ arising from Theorem \ref{t:cptfy}.
If $m = 1$, then $\oM$ is necessarily smooth and fibred by $|\oD|$ by \ref{t:cptfy}\ref{it:fibre}, but even in this case our proof makes no use of the
fibration and our result is more precise: Tian-Yau make no statement about which K\"ahler classes on $M$ contain complete
Ricci-flat metrics, nor do they prove that these metrics converge to cylinders at infinity.
\end{remark}

Projective \emph{manifolds} $\oM$ satisfying the hypotheses of Theorem \ref{t:geom-exi}
were first constructed by Kovalev \cite{kovalev03} as blow-ups of Fano $3$-folds; this construction 
yields around one hundred families of ACyl Calabi-Yau $3$-folds with split cross-section $\Sph^1 \times D$.
In \cite{chnp1} so-called \emph{weak} Fano manifolds are used instead; 
the weak Fano construction yields hundreds of thousands of families of split ACyl Calabi-Yau $3$-folds. 

Kovalev-Lee \cite{kovalev-lee08} describe a different class of manifolds $\oM$
satisfying the hypotheses of Theorem~\ref{t:geom-exi} based on $K3$ surfaces with anti-symplectic involutions.
This leads to around 70 further families of 
split ACyl Calabi-Yau $3$-folds. 
By modifying the construction of \cite{kovalev-lee08}, we can find admissible \emph{orbifolds} $\oM$ with $m > 1$, as follows. (The
cross-section of the resulting non-split ACyl Calabi-Yau $3$-fold will be the mapping torus 
of a finite order symplectic automorphism of $K3$.)

\begin{ex}
\label{ex:dihedral}
Let $D$ be a $K3$ surface with a group $G = \gen{\iota, \tau}$ of holomorphic automorphisms where $\iota$ is symplectic of order $m$ and $\tau$ is an anti-symplectic
involution with non-empty fixed set such that $\tau\iota\tau = \iota^{-1}$; in particular, $G$ is isomorphic to the dihedral group with $2m$ elements.

Let $\iota$ act on $\P^1$ by
$z \mapsto e^{2\pi i/m}z$, and $\tau$ by $z \mapsto \frac{1}{z}$.
Let $\oM_0 = (\P^1 \times D)/G$ and let $\oM$ be the blow-up of $\oM_{0}$ at the fixed sets of
the reflections $\tau\gen{\iota} \subset G$ (which are disjoint).
$\oM$ has orbifold singularities from the
fixed points of the rotations $\gen{\iota}$, which all lie in the image
$\oD = D/\Z_m$ of $\{0, \infty\} \times D$. 

By Theorem \ref{t:geom-exi}, $M = \oM \setminus \oD$ admits
ACyl Calabi-Yau metrics with cross-section $X = (\Sph^1 \times D)/\Z_m$.
Moreover, we can construct a fibration $\oM \to \P^1$ with $\oD$ as an $m$-fold fibre as
in Example \ref{ex:end}, though in this case the existence of the fibration is
also guaranteed by Theorem \ref{t:cptfy}\ref{it:fibre} since $b^{1}(D)=0$.

Here we choose not to pursue a systematic study of such examples  and instead content ourselves with 
exhibiting a few concrete ones. 
As in Remark \ref{r:endrestrictions}\ref{it:3d} we have the a priori bound $m\le 8$. 
\mbox{\cite[\S 3]{koz05}} describes a $K3$ surface with an automorphism group
$A_6 \rtimes \bbz_4$ containing $G$ of the
required kind for $2 \leq m \leq 6$; see also \cite[\S 7]{frantzenthesis}.
For $m = 2, 3,4$ one can also use Kummer surface constructions.
\end{ex}

To round off our discussion we state a uniqueness theorem. Given some facts from ACyl Hodge 
theory, the proof is fairly straightforward. See also \cite[Thm 1.9]{hein} and the surrounding discussion.

\begin{mtheorem}\label{t:uniqueness}
Let $M$ be an open complex manifold with only one end and let $\omega_1,\omega_2$ be {\rm ACyl} K\"ahler metrics on $M$ such that $\omega_1-\omega_2$ is exponentially decaying with respect to either $\omega_1$ or $\omega_2$. If $\omega_1,\omega_2$ represent the same class in $H^2(M)$ and have the same volume form, then $\omega_1 = \omega_2$.
\end{mtheorem}

Our main reason for including this result is that it allows us to see that Theorems \ref{t:cptfy} and \ref{t:geom-exi} are
inverse to each other---at least in the simply-connected $n > 2$ case. Indeed, if we start with an ACyl Calabi-Yau $n$-fold $M$ with metric
$\omega$, apply Theorem \ref{t:cptfy} to
compactify it to $\oM$, and apply Theorem \ref{t:geom-exi} to $\oM$ to
construct another ACyl Calabi-Yau metric $\omega'$ on $M$ in the same Kähler
class as $\omega$, then $\omega - \omega'$
will be exponentially decaying and so Theorem \ref{t:uniqueness} implies that
$\omega = \omega'$.

\subsection*{Concluding remarks}

We have now come full circle in our theory if the complex dimension is at least $3$: there exists
a natural generalisation and refinement of the Tian-Yau construction of K\"ahler Ricci-flat metrics of linear volume growth, 
and we have proved that this construction exhausts all possible examples of exponentially asymptotically
cylindrical Calabi-Yau manifolds 
that are simply-connected and irreducible. In this section we wish to point out a few open questions.

At a rather basic level we do not currently know whether ACyl Calabi-Yau $n$-folds with non-split 
cross-section $(\Sph^1 \times D)/\langle \iota \rangle$, ${\rm ord}(\iota) = m > 1$, are scarce or plentiful. All the examples we know of are fibred over $\C$, though we have been unable to prove the existence of such a fibration in general and unlike in \cite{tianyau90} our constructions do not rely on it.
There exist formal obstructions to fibering over $\C$ (see \ref{r:fibering_obstructions}), and we suspect that the existence of a fibration is not stable under deformations.

Even in the split case ($m = 1$) it remains to classify the possible projective manifolds $\oM$
satisfying the hypotheses of Theorem \ref{t:geom-exi}.  In three dimensions the
vast majority of known examples \cite{chnp1, kovalev03} (but not all \cite{kovalev-lee08}) arise by blowing up the base loci of smooth
anticanonical pencils in smooth weak Fano $3$-folds. The weak Fano construction produces
a very large but provably finite number of deformation families of split ACyl Calabi-Yau
3-folds. Is it possible to prove that there exist only finitely many 
deformation families of split ACyl Calabi-Yau $3$-folds?

Another (metric) question that remains is whether there exist asymptotically cylindrical Calabi-Yau manifolds with \emph{slower} than exponential convergence. However, applying the methods of Cheeger-Tian \cite{cheeger:tian} should rule this out---if the Gromov-Hausdorff distance of a complete Calabi-Yau manifold to a cylinder goes to zero at infinity, then the convergence should automatically be exponential in $C^\infty$ because 
the cross-section of the cylinder is always integrable as an Einstein manifold.

For a potentially more interesting analytic question, recall that complete Riemannian manifolds of nonnegative Ricci curvature always have at least linear volume growth. The case of \emph{precisely} linear volume growth would therefore seem to be somewhat rigid; but
examples due to Sormani show that numerous pathologies can occur \cite{sor}. Does the Calabi-Yau condition impose further restrictions? Is a complete Calabi-Yau of linear volume growth necessarily Gromov-Hausdorff asymptotic to 
$\R \times X$ for some geodesic metric space $X$? If so, then could $X$ be 
non-compact or singular? 

Finally, we would like to mention some closely related papers that have appeared since this paper was first posted to the arXiv. Li \cite{ChiLi} proved a compactification theorem for asymptotically \emph{conical} complex manifolds similar to Theorem \ref{thm:cptfy} and gave some interesting applications. Li's result was used in \cite{AC3} to prove an asymptotically conical analogue of Theorem \ref{t:cptfy} and a number of uniqueness theorems for asymptotically conical Calabi-Yau manifolds. In a different direction, \cite{CMR} establishes a complete picture of the \emph{deformation and moduli theory} of ACyl Calabi-Yau manifolds.

\subsection*{Acknowledgments} 
MH would like to thank the EPSRC for their continuing support of his research
under Leadership Fellowship EP/G007241/1, which also provided postdoctoral
support for HJH and JN. JN also thanks the ERC for postdoctoral support under
Grant 247331.
The authors would like to thank Martijn Kool and Richard Thomas for help with the proof of Theorem \ref{t:cptfy}\ref{it:fibre}, and Tommaso Pacini for useful comments on a previous draft of this paper.

\section{Basic properties of ACyl Calabi-Yau manifolds}

This section discusses the basic analysis, geometry, and topology of ACyl
Calabi-Yau manifolds. In particular, it provides the technical tools necessary for
the rest of the paper. The results stated in Theorems
\ref{thm:acyl_structure} and \ref{T:acyl:holo} will be proved as we go along: 
see Corollary \ref{t:acyl_top} for \ref{thm:acyl_structure} and
\S\ref{s:acyl_holonomy} for \ref{T:acyl:holo}.

\subsection{Linear analysis and Hodge theory on ACyl manifolds}
\label{s:acyl_analysis}

We review some analytic facts for elliptic operators on manifolds with
cylindrical ends from Lockhart-McOwen \cite{lockhart85}, with applications to
the scalar and Hodge Laplacians and the Dirac operator on ACyl manifolds.

Suppose that $M = U \cup ([0,\infty) \times X)$ topologically for a bounded domain $U \subset M$ and a compact (but not necessarily connected) manifold $X$. A differential operator $\cala : \Gamma(E) \to \Gamma(F)$ on sections
of tensor bundles on $M$ is called \emph{asymptotically translation-invariant} if
there is a translation-invariant operator $\cala_\cyl$ on sections of the
corresponding bundles on $\R_t \times X$ such that the difference between
the coefficients of $\cala$ and $\cala_\cyl$ goes to zero in $C^\infty$ uniformly as $t \to \infty$. 
Now even if $\cala$ is elliptic, then since $M$ is noncompact we cannot expect $\mathcal{A}$ to induce a Fredholm operator on ordinary H\"older or Sobolev spaces. To fix this, it is helpful to introduce Hölder norms with exponential weights.

\begin{definition}
Extend $t$ smoothly to the whole of $M$. For $u \in C^\infty_0(E)$ define
\begin{equation}
\norm{u}_{C^{k,\alpha}_\delta(E)} \equiv \|e^{\delta t} u\|_{\holda{k}(E)},
\end{equation}
and let $C^{k,\alpha}_\delta(E)$ denote the associated Banach space completion of
$C^\infty_0(E)$. Thus, $\holdad{k}$ sections are exponentially decaying for $\delta > 0$, and at worst exponentially growing for $\delta < 0$. We will occasionally use the notation $C^\infty_\delta(E) \equiv \bigcap C^{k,\alpha}_\delta(E)$. 
\end{definition}

We now assume that $\cala$ is elliptic, \ie that the principal symbol
of $\mathcal{A}$ is an isomorphism in every cotangent direction. Then $\delta$
is called a \emph{critical weight} if there exists a non-zero solution of
\begin{equation}
\label{eq:poly}
\cala_\cyl(e^{i\lambda t}u) = 0,
\end{equation}
where $\im \lambda = \delta$ and $u$ is a section of $E \to \R \times X$ that
is polynomial in $t$. The set of critical weights is a discrete subset of
$\bbr$. We then have the following basic result {\cite[Thm 6.2]{lockhart85}}:

\begin{prop}
\label{thm:lockhart}
Let $\cala : \Gamma(E) \to \Gamma(F)$ be asymptotically
translation-invariant elliptic of order $r$.
If $\delta$ is not a critical weight then
the induced linear map $\cala : \holdad{k+r}(E) \to \holdad{k}(F)$ is Fredholm.
\end{prop}

We mention some ingredients of the proof---partly because the result is stated for Sobolev rather than H\"older spaces in
\cite{lockhart85}, and partly because we will need Remark \ref{r:right_inverse} repeatedly in Section \ref{sec:cptfy}.
The first step is to invert $\cala$ along the cylindrical end.

\begin{prop}
\label{p:mazya}
If $\delta$ is not critical then there exists $\calr : \holdad{k}(F) \to \holdad{k+r}(E)$ linear and bounded such that
$\cala \circ \calr = {\rm id}$ on the complement of a bounded subset of $M$.
\end{prop}

\begin{proof}
Maz'ya-Plamenevski\u{\i} \cite[Thm 5.1]{mazya78} show that
$\cala_\infty : \holdad{k+r}(E) \to \holdad{k}(F)$
is an isomorphism by using the Fourier transform. The condition on $\delta$ ensures
that if $v \in \Gamma(F)$ is translation-invariant and $\im \lambda = \delta$,
then $\mathcal{A}_\infty(e^{i \lambda t}u) = e^{i \lambda t}v$ has a unique
translation-invariant solution $u \in \Gamma(E)$.

Let $t_0 \gg 1$ and let $\rho: \bbrp \to \bbr$ be a cut-off function that is $0$ for $t < t_0-1$
and $1$ for $t > t_0$. Set $\cala' \equiv (1-\rho)\cala_\cyl + \rho \cala$
on $X \times \bbr$. Then $\cala'$ is close to $\cala_\cyl$ in operator norm, so it has an
inverse $\calr' : \holdad{k}(E) \to \holdad{k+r}(E)$.
If we define $\calr(u) \equiv \calr'(\rho u)$ on $M$, then $\cala(\calr(u)) = u$ for $t > t_0$.
\end{proof}

\begin{remark}\label{r:right_inverse}
What is proved here is that $\mathcal{A}$ has a right inverse defined on $C^{k,\alpha}_\delta(F)$ over $[t_0,\infty) \times X$ provided that $t_0$ is large enough depending on $k,\alpha,\delta$ and on the rate of convergence of $\mathcal{A}$ to $\mathcal{A}_\cyl$. Since right inverses are not unique, it is not immediately clear whether or not the one constructed here is independent of $k,\alpha$, \ie compatible with the obvious inclusions $C^{\ell,\beta}_\delta \subseteq C^{k,\alpha}_\delta$ for $\ell \geq k$ and $\beta \geq \alpha$.
But this is clear from the proof, provided that the same cut-off function $\rho$ is used.
\end{remark}

Now let $\psi \in C^\infty_0(M)$ be a cut-off function which is equal to $1$ for
$t < t_0$. Proposition \ref{thm:lockhart} can be deduced from Proposition \ref{p:mazya} together
with local Schauder theory and the fact that multiplication
by $\psi$ and the commutator $[\mathcal{A},\psi]$ define compact maps
$\holdad{k+r}(E) \to \holdad{k}(E)$; see \cite[\S 2]{lockhart85}.

In \cite[Thm 6.2]{lockhart85}, Lockhart-McOwen also provide a formula to compute the change in
index of $\mathcal{A}$ as $\delta$ passes a critical weight, by counting the
number of solutions of \eqref{eq:poly}. In \cite[Thm 7.4]{lockhart85},
this is used to compute the indices of formally self-adjoint
operators for $|\delta| \ll 1$. One application is

\begin{prop}\label{p:invert_laplacian}
If $X$ is connected and $\delta > 0$ is smaller than the square root of the first eigenvalue of the scalar Laplacian on $X$, then the scalar Laplacian on $M$ maps $C^{k+2,\alpha}_\delta(M)$ isomorphically onto the subspace $C^{k,\alpha}_\delta(M)_0$ of functions of mean value zero.
\end{prop}

\begin{proof}
Integration by parts shows that the kernel of $\Delta : C^{k+2,\alpha}_\delta(M) \to C^{k,\alpha}_\delta(M)$ is trivial, and that functions in the image have mean value zero. But the index of $\Delta$ on these spaces is $-1$.
\end{proof}

The proof of the index formula uses asymptotic expansions for the elements
in the kernel of $\cala$. If we assume that $\cala$ is asymptotic to $\cala_\cyl$ at
an exponential (rather than just uniform) rate, these can be described more
simply. This often makes it possible to imitate Hodge theoretic arguments on compact manifolds that are based on integration by parts and Weitzenb\"ock formulas.

For example, if $M$ is ACyl in the sense of Definition \ref{d:ACyl}, then every bounded harmonic form $\alpha$ on $M$ has an asymptotic limit
$\alpha_\cyl$, which is itself a harmonic form on $M_\cyl$, such that
$\alpha - \alpha_\cyl \in C^{k,\alpha}_\delta$ on $M_\infty$ for all $k,\alpha$ and some $\delta > 0$.
The bounded harmonic forms with $\alpha_\infty = 0$ are precisely the $L^2$-integrable
ones. We denote the space of all bounded harmonic $k$-forms by $\harm^k_{\bd}(M)$.

\begin{prop}
\label{p:acyl:hodge}
Let $M$ be an \acyl Riemannian manifold.
\begin{enumerate}
\item
\label{it:harm:abs:dr}
The natural map $\harm^{k}_{\bd}(M) \to H^k(M)$ to the de Rham cohomology
of $M$ is surjective.
\item
\label{it:harm:dr:1end}
If $M$ has a single end then $\harm^{1}_{\bd}(M) \to H^1(M)$ is an isomorphism. 
\item
\label{it:acyl:bochner}
If $M$ has nonnegative Ricci curvature then any bounded harmonic $1$-form on $M$ is parallel.
\item
\label{it:killing}
If $M$ has nonpositive Ricci curvature then any Killing vector field on $M$ is
parallel.
\end{enumerate}
\end{prop}
\begin{proof}
For (i), see Melrose \cite[Thm 6.18]{melrose94}. For (ii), see
\cite[Cor 5.13]{jn1}. (iii) is proved by the Bochner method. For (iv), first note that
 every Killing field of $M$ converges exponentially to a Killing field of $M_\infty$ \cite[Prop 6.22]{jn1}. Thus, the Bochner method applies again.
\end{proof}

Another application, which will be very significant for us, is to the Dirac operator of an ACyl spin manifold $M$.
Let $\harm^S_\cyl$ be the space of translation-invariant solutions of the
Dirac equation $\dirac s = 0$ on~$M_\infty$, and let $\harm^S_{\bd}$ and
$\harm^S_{L^2}$ denote the bounded and $L^2$ solutions on $M$. In analogy with harmonic forms, every element of $\harm^S_{\bd}$ is asymptotic at an exponential rate to an
element of $\harm^S_\cyl$.

\begin{prop}\label{prop:spcount}
Let $M$ be an \acyl spin manifold.
\begin{enumerate}
\item
$\dim (\harm^S_{\rm bd}/\harm^S_{L^2}) = \half \dim \harm^S_\cyl$.
\item If $M$ has nonnegative scalar curvature, then every element of
$\harm^S_{\rm bd}$ is parallel.
\end{enumerate}
\end{prop}

\begin{proof}
(i) is essentially an instance of (3.25) in Atiyah-Patodi-Singer
\cite{atiyah75}. It can also be deduced from the previously mentioned index
formula \cite[Thm 7.4]{lockhart85}; see \cite[\S 2.3.5]{jnthesis} for
details.
(ii) follows from the Lichnerowicz formula and integration by parts.
\end{proof}

\begin{remark}
Proposition \ref{prop:spcount}(i) has a rather simple intuitive meaning. Let $\mathcal{A}$ be an asymptotically translation-invariant elliptic differential operator. Given any subexponentially growing solution to $\mathcal{A}_\infty(u_\infty) = 0$ on $\R \times X$, we can try to find a solution to $\mathcal{A}(u) = 0$ on $M$ with
asymptotic limit $u_\infty$. Obstructions arise by taking the $L^2$ inner product of the equation $\mathcal{A}(u) = 0$ with subexponentially growing elements of ${\rm ker}(\mathcal{A}^*)$ and integrating by parts. Thus, 
if $\mathcal{A}= \mathcal{A}^*$, then we expect that exactly half of all possible solutions $u_\infty$ can be extended in this way. For instance, if $\mathcal{A}$ is the Laplacian on scalars and if $X$ is connected, then clearly the constant functions on $\R \times X$ extend harmonically to $M$ but $t$ does not because otherwise $0 = \int_M \Delta u = \lim_{T \to \infty} \int_{X} \frac{\partial u}{\partial t}(T,x)\,dx = {\rm Vol}(X)$. 
\end{remark}

The strength of Proposition \ref{prop:spcount} is well-illustrated by the following
``positive mass theorem'', which is an immediate consequence by \cite{wang89} (but will not be used in the rest of this paper).

\begin{corollary}
Let $M$ be an {\rm ACyl} spin manifold of nonnegative scalar curvature. If the end $M_\infty$ is Ricci-flat of special holonomy, then so is $M$.
\end{corollary}

\subsection{Structure of Ricci-flat ACyl manifolds}
The goal here is to extend the structure theorem for compact Ricci-flat manifolds of
Remark \ref{r:basic}\ref{it:toruscover} to the ACyl setting, proving Theorem \ref{thm:acyl_structure}. 
As in the compact case, this will be a relatively easy consequence of a more general result 
(Theorem \ref{thm:acyl_nonneg}) for manifolds with ${\rm Ric} \geq 0$. At the end of this section, we also collect some closely related remarks that will not be used in rest of this paper, but are useful in \cite[\S 2]{chnp1} and
\cite[\S 3]{chnp2}. All coverings in this section will be Riemannian, and all deck transformations are isometries.

The theory in the compact case rests on a subtle observation due to Cheeger-Gromoll in the proof of \cite[Thm 3]{cheeger71}. The following proposition states a slight extension of their idea that we require for our ACyl structure theorem. We give the proof for convenience.

\begin{prop} 
\label{prop:cpt_nonneg}
A complete Riemannian manifold $Z$ with ${\rm Ric} \geq 0$ admits a cocompact
isometric group action if and only if $Z$ splits as the isometric product of
$\R^k$ and some compact manifold.
In this case, every cocompact and discrete subgroup $\Gamma \subset \iso(Z)$
contains a normal subgroup $\Psi$ of finite index such that $[\Psi,\Psi]$ is
finite and $\Psi/[\Psi,\Psi]$ is a free abelian group of rank $k$.
\end{prop}

\begin{proof}
By the splitting theorem, $Z = \R^k \times Z'$, where $Z'$ contains no lines, and we must show that $Z'$ is necessarily compact. Notice that ${\rm Iso}(Z) = {\rm Iso}(\R^k) \times {\rm Iso}(Z')$ because $Z'$ is line-free. Since ${\rm Iso}(Z)$ acts cocompactly on $Z$,
there exists a compact set $F' \subset Z'$ whose translates under ${\rm Iso}(Z')$ 
cover $Z'$. If $Z'$ itself was noncompact, then there would exist a nontrivial ray $\gamma: [0,\infty) \to Z'$. For each $n \in \N$ there exists
$g_n \in {\rm Iso}(Z')$ with $g_n(\gamma(n)) \in F'$. We can assume that $g_n(\gamma(n))$ has a limit as $n \to \infty$ because $F'$ is compact. But then the shifted rays $\gamma_n: [-n,\infty) \to Z'$ defined by $\gamma_n(t) = g_n(\gamma(t + n))$ subconverge to a line locally uniformly in $t$, which contradicts the definition of $Z'$. 

Let $\Gamma'$ be the kernel of the projection of $\Gamma$ to ${\rm Iso}(\R^k)$. Then $\Gamma'$ is a discrete subgroup of ${\rm Iso}(Z')$, hence finite. On the other hand, the \emph{image} $\Gamma''$ of the projection of $\Gamma$ to ${\rm Iso}(\R^k)$ acts cocompactly
on $\R^k$, and is discrete because ${\rm Iso}(Z')$ is compact and $\Gamma$ is discrete. Thus $\Gamma''$ is a Bieberbach group. In other words, we have an exact sequence $1 \to \Gamma' \to \Gamma \to \Gamma'' \to 1$ with $\Gamma'$ finite, and a split exact sequence $1 \to \Z^k \to \Gamma'' \to \Gamma''' \to 1$ with $\Gamma'''$ a finite subgroup of O$(k)$ acting on $\Z^k$ in the standard fashion. The preimage $\Psi$ of $\Z^k$ under $\Gamma \to \Gamma''$ is then normal of finite index in $\Gamma$. Also, we have an exact sequence
$1 \to \Psi' \to \Psi \to \Z^k \to 1$, so that $[\Psi,\Psi] \subset \Psi' \subset \Gamma'$ must be finite.
\end{proof}

\begin{remark}\label{r:grouptheory}
Given a finitely generated group $\Gamma$ with a finite index normal subgroup $\Psi$ such that $[\Psi,\Psi]$ is
finite, the rank $k < \infty$ of the abelian group $\Psi/[\Psi,\Psi]$ only depends on $\Gamma$; in fact, $k$ is equal to the volume growth exponent of the Cayley graph of $\Gamma$. 
\end{remark}

By applying Proposition \ref{prop:cpt_nonneg} to various normal covers of the cross-section of an ACyl manifold
and bringing in some ACyl Hodge theory from Section \ref{s:acyl_analysis}, we will prove the following key

\begin{theorem}
\label{thm:acyl_nonneg}
Let $M$ be {\rm ACyl} with ${\rm Ric} \geq 0$ and a single end. Then either $M$
is a $\Z_2$-quotient of a cylinder, or its universal cover is isometric
to $\bbr^k \times {M}'$, where ${M}'$ is {\rm ACyl} with a single end.
\end{theorem}

\begin{remark}
We will see in the proof that $k \geq b^1(M)$, but the inequality can be strict; this already happens in the compact case if 
$M$ is any compact flat $k$-manifold other than $\Tor^k$. However, $k$ equals $b^1$ of a certain finite normal cover of $M$ whose fundamental group has finite derived group.
\end{remark}

The structure theorem for Ricci-flat ACyl manifolds (Theorem \ref{thm:acyl_structure}) follows from this.

\begin{corollary}
\label{t:acyl_top}
Every Ricci-flat {\rm ACyl} manifold has a finite normal cover 
that splits isometrically as the product of a flat torus and a simply-connected Ricci-flat {\rm ACyl} manifold.
\end{corollary}

\begin{proof}
If $M$ is a cylinder or a $\Z_2$-quotient of one, then the claim follows from Remark \ref{r:basic}\ref{it:toruscover} applied to the cross-section. If not, then Theorem \ref{thm:acyl_nonneg} shows that the universal
cover $\tM$ of $M$ splits as an isometric product $\R^k \times M'$, where $M'$ is
ACyl with a single end. Thus, $\iso(\tM) = \iso(\R^k) \times \iso(M')$.
As $M'$ has a single end, the orbits of $\iso(M')$ are bounded, which
implies that $\iso(M')$ is compact. Therefore the projection of $\pi_1(M)$ to
$\iso(\bbr^k)$ is discrete, hence a Bieberbach group, so its projection to 
$\sorth{k} = \iso(\bbr^k)/\R^k$ is finite. Since 
$M'$ is simply-connected Ricci-flat, Proposition
\ref{p:acyl:hodge}\ref{it:killing} tells us that $\iso(M')$ is discrete, hence finite. The kernel
$\Gamma$ of the projection $\pi_1(M) \to \sorth{k} \times \iso(M')$ is therefore a finite
index normal subgroup of $\pi_1(M)$ whose image in $\iso(\R^k)$
acts on $\R^k$ as a full rank lattice of translations.
Thus $(\R^k/\Gamma) \times M'$ is a cover of the required form. 
\end{proof}

\begin{ex}
To appreciate the role that the Ricci-flat condition plays in this proof, it is helpful to consider the following (compact) example \cite[p. 440]{cg1}. Let $M$ be the mapping torus of a rotation of $\Sph^2$ by an irrational angle. Then $M$ is diffeomorphic to
$\Sph^1 \times \Sph^2$, ${\rm Ric}\, M \geq 0$, but no finite cover of $M$ 
splits
isometrically as $\Sph^1 \times \Sph^2$. The proof of Corollary \ref{t:acyl_top}
fails at the point where one uses that the isometry group of $M'$ is finite: the kernel of $\pi_1(M) \to \sorth k \times \iso(M')$ is trivial here.
\end{ex}

We preface the proof of Theorem \ref{thm:acyl_nonneg} with a simple lemma that
will be applied twice.

\begin{lemma}\label{l:top}
Let $Y$ be a connected manifold
and $i: X \to Y$ the inclusion of a connected open set. Let $G$ be a subgroup
of $\pi_1(Y)$ and $p:\widetilde{Y} \to Y$ the covering space with
characteristic group $G$. Then the number of connected components of
$p^{-1}(X)$ is equal to the index of $\langle G, i_*(\pi_1(X))\rangle$ in
$\pi_1(Y)$, and each such connected component is a covering of $X$ with
characteristic group $i_*^{-1}(G) \subset \pi_1(X)$.
\end{lemma}

The first application deserves separate mention since it will itself be applied repeatedly.

\begin{lemma}
\label{lem:cyl_cover}
If $M$ is {\rm ACyl} with ${\rm Ric} \geq 0$ and a single end, then either $\pi_1(M_\cyl) \to \pi_1(M)$ is onto and every finite cover of $M$
has a single end, or else the image has index $2$ and
$M = M_\cyl/\Z_2$.
\end{lemma}

\begin{proof}
If $\pi_1(M_\cyl) \to \pi_1(M)$ is not surjective, consider the cover
$\tM \to M$ with characteristic group equal to the image. By Lemma \ref{l:top},
$\tM$ has at least two cylindrical ends on which the covering map is a
diffeomorphism onto $M_\infty$. Thus, by the splitting theorem,
$\tM = M_\infty$, and $M = M_\infty/\Z_2$.
\end{proof}

\begin{proof}[Proof of Theorem \ref{thm:acyl_nonneg}]
Write $M_\infty = \R \times X$ for the end of $M$. By Lemma
\ref{lem:cyl_cover}, we can assume that $\pi_1(M_\cyl) \to \pi_1(M)$ is
surjective. By Proposition
\ref{prop:cpt_nonneg} applied to the universal cover of $X$, 
$\pi_1(M_\infty)$ contains a finite index normal subgroup whose
derived group is {finite}.
Since $\pi_1(M_\infty)$ surjects onto $\pi_1(M)$, the image
$\Psi$ of this subgroup in $\pi_1(M)$ is still normal of finite index and has finite derived group.
Replacing $M$ by its finite normal cover with characteristic group $\Psi$, which is still ACyl with a single end, we can thus assume without loss that $\pi_1(M)$ itself has finite derived group. 

Let $k \in \N_0$ denote the rank of the abelianisation of $\pi_1(M)$.
Then in particular $b^1(M) = k$, and so Proposition
\ref{p:acyl:hodge}\ref{it:harm:dr:1end}-\ref{it:acyl:bochner} tells us that
$k$ is also the number of parallel vector fields on $M$. Thus, by de Rham's
theorem, the universal cover $\tM$ splits as an isometric product $\tM = \R^k
\times M'$, where $M'$ is complete and simply-connected. A priori $M'$ could
split off further line factors, but our goal is to show that this does not
happen and moreover that $M'$ is ACyl with a single end.

The parallel vector fields on $M$ form a $k$-dimensional abelian Lie algebra
$\aalg$ of Killing fields on $M$.
Sending each element of $\aalg$ to its asymptotic
limit under the inverse ACyl map $\Phi^{-1}$ of Definition \ref{d:ACyl},
we obtain an isomorphism $\phi: \aalg_\infty \to \aalg$ with an abelian
Lie algebra $\aalg_\infty$ of parallel Killing fields on
$M_\infty = \R \times X$. The elements of $\aalg_\infty$ have no $\partial_t$-components (or in other words, are tangent to $X$) since otherwise
$\iso(M)$ would have unbounded orbits, which is not possible since $M$ has only one end \cite[Lemma 3.6]{kovalev05}.
Notice also that $\Phi$ is asymptotically $\phi$-equivariant: we have
\begin{equation}\label{e:approx_equiv}
{\rm dist}_M ( \Phi(t,\exp(a)x), \, \exp(\phi(a))\Phi(t,x) )
\leq C|a|e^{-\delta t}
\end{equation}
for all $a \in \mathfrak{a}_\infty$, simply by how $\phi$ was defined.

Elements of $\aalg$ pull back to parallel Killing fields on $\tM$. By
construction, the Lie algebra $\tilde\aalg$ of all such pull-backs consists of the
parallel vector fields tangent to the $\bbr^k$ factor in $\tM = \bbr^k \times M'$. 
We can assume that the domain $U$ of Definition \ref{d:ACyl} is $\mathfrak{a}$-invariant. Put $E \equiv M \setminus U$ and let $\tilde{E}$ be the
preimage of $E$ under the covering map $\tM \to M$. By $\tilde{\mathfrak{a}}$-invariance, we have $\tilde{E} = \bbr^k \times E'$ with
$E' \subset M'$.

Lemma \ref{l:top} tells us that $\tilde{E}$ is a connected normal covering space of $E$ with characteristic group $\ker \left(\pi_1(M_\infty) \to \pi_1(M) \right)$ and deck group $\pi_1(M)$. There certainly exists a connected normal covering space $\tilde{X} \to X$ such that there exists a diffeomorphism $\tilde{\Phi}: [0,\infty) \times \tX \to \tilde{E}$ covering $\Phi$. Let $\tilde{\aalg}_\infty$ be the pull-back of $\aalg_\cyl$ to $\tX$. Then $\tilde{\aalg}_\cyl$ is an abelian Lie algebra of parallel Killing fields on $\tX$, $\phi$ induces an isomorphism $\tilde{\phi}: \tilde{\aalg}_\cyl \to \tilde{\aalg}$, 
and (\ref{e:approx_equiv}) implies that
\begin{equation}\label{e:approx_equiv_lifted}
{\rm dist}_{\tM}(\tilde{\Phi}(t,\exp(\tilde{a})\tilde{x}), \,
\exp(\tilde{\phi}(\tilde{a}))\tilde{\Phi}(t,\tilde{x}))
\leq C|\tilde{a}|e^{-\delta t}
\end{equation}
for all $\tilde{a} \in \tilde{\mathfrak{a}}_\infty$; to prove (\ref{e:approx_equiv_lifted}), fix $N \gg 1$ depending only on $\tilde{a}$ such that, for every $\tilde{y} \in \tilde{X}$, $\exp(\frac{\tilde{a}}{N})\tilde{y}$ is closer to $\tilde{y}$ than any deck group translate of $\tilde{y}$, and then apply (\ref{e:approx_equiv}) $N$ times.

We now wish to use these preparations to argue that $\tX = \R^k \times X'$ with $X'$ \emph{compact}, and that $\tilde{\Phi}$ induces an ACyl diffeomorphism $\Phi': [0,\infty) \times X' \to E'$ in the sense of Definition \ref{d:ACyl}. The key point of this argument is the following: $\pi_1(M)$ acts isometrically on $\tX$ with compact quotient $X$. Thus, Proposition \ref{prop:cpt_nonneg} tells us that 
 $\tilde{X} = \R^\ell \times X'$ with $X'$ {compact} for some $\ell \in \N_0$, and that $\pi_1(M)$ has a finite index normal subgroup with finite derived group whose abelianisation has rank $\ell$. But recall that we arranged for $\pi_1(M)$ itself to have finite derived group; thus, $\ell = k$ by Remark \ref{r:grouptheory}.

Now the basic idea for splitting off $\Phi'$ from $\tilde{\Phi}$ is as follows. 
Since $\tilde{\Phi}$ is an almost isometry, it sends lines to almost lines. But the lines in $\tM$ are $\tilde{\aalg}$-orbits and $\tilde{\Phi}$ is almost equivariant, so the lines in $\tX$ are $\tilde{\aalg}_\cyl$-orbits (approximately---hence precisely) even though a priori we only knew that $\tilde{\aalg}_\cyl$ consisted of parallel vector fields and $X'$ might have parallel vector fields too. Using the
approximate isometry and equivariance properties of $\tilde{\Phi}$ again, it quickly follows that $\tilde{\Phi}$ acts as an almost isometry on the $\R^k$ factor and as an ACyl diffeomorphism on the $[0,\infty) \times X'$ factor.

In fact we will argue slightly differently. If $\tilde{a} \in \tilde{\aalg}_\cyl$ had a nontrivial $X'$-component, then the curves $\gamma_t(s) \equiv (t,\exp(s\tilde{a})\tilde{x})$ would not be lines, \ie there exist $s_0 > 0$ and $\theta < 1$ independent of $t$ such that the distance between $\gamma_t(0)$ and $\gamma_t(s_0)$ is $\theta s_0$. But $\tilde{\aalg}$ is tangent to the $\R^k$ factor in $\tilde{E}$, so (\ref{e:approx_equiv_lifted}) shows that $\tilde{\Phi} \circ \gamma_t: [0,s_0] \to \tilde{E}$ remains $O(s_0 e^{-\delta t})$ close to a line segment of length $s_0$. This means that if $\sigma$ is any other curve in $\tX$ connecting $\gamma_t(0)$ and $\gamma_t(s_0)$, then $\tilde{\Phi} \circ \sigma$ has length at least $s_0 - O(s_0e^{-\delta t})$. Now $\tilde\Phi^* g_{\tilde M} =  dt^2 + g_{\tX} + O(e^{-\delta t})$, so the length of $\sigma$ itself is at least $s_0 - O(s_0e^{-\delta t})$.  Taking $\sigma$ to be distance minimising and $t$ sufficiently large relative to $\theta$ and $s_0$, we get a contradiction. 

Now we know that the $\tilde{\mathfrak{a}}_\infty$-orbits are the lines in $\tX = \R^k \times X'$. Fixing linear coordinates $y$ on $\R^k$ and writing $x$ for points in $X'$ for simplicity, (\ref{e:approx_equiv_lifted}) then implies that
\begin{equation}\label{e:splitting_phi}
\tilde{\Phi}(t, y, x) = (\tilde{\Phi}(t,0,x)_{\R^k} + \tilde{\phi}(y), \tilde{\Phi}(t,0,x)_{M'}) + O(|y|e^{-\delta t}).
\end{equation}
Here we have decomposed the target $\tM = \R^k \times M'$. Notice that (\ref{e:approx_equiv_lifted}) provides $O(|y|e^{-\delta t})$ control on the errors only in a distance sense; we will take it for granted that if $|y| \ll 1$ and $t \gg 1$ then this can be upgraded to $C^\infty$ control in local charts (alternatively we could arrange for $\tilde{\Phi}$ to be \emph{precisely} equivariant but this requires similar technical work to make precise). It then follows from (\ref{e:splitting_phi}) and the almost isometry property $\tilde{\Phi}^*[dy^2 + g_{M'}] = [dt^2 + dy^2 + g_{X'}] + O(e^{-\delta t})$ that
\begin{equation}
\tilde{\Phi}(t,0,x)_{\R^k}  = const + O(e^{-\delta t}), \;\, ({\Phi}')^*[g_{M'}] = [dt^2 + g_{X'}] + O(e^{-\delta t}),
\end{equation}
where we have defined $\Phi'(t,x) \equiv \tilde{\Phi}(t,0,x)_{M'}$.

To conclude that $M'$ is an ACyl manifold in the sense of Definition
\ref{d:ACyl}, it remains to prove that $M'\setminus E'$ is bounded. If not, then $M'$ would be a cylinder by the splitting theorem, \ie there exists a function $t': M' \to \R$ with $\nabla^2t' = 0$ which is exponentially asymptotic to $t: E' \to [0,\infty)$ on $E'$. Notice that the trivial extension of $t'$ to $\tM = \R^k \times M'$ is deck group invariant because $\tilde{E}$ and $t$ are. But then $t'$ pushes down to an unbounded Lipschitz function on the bounded region $U \subset M$. (This whole argument crucially exploits that $\tilde{E}$ is connected by our initial reductions.)\end{proof}

With the proof of the main theorem of this section out of the way, we now
explain some related but more elementary observations that are needed in
\cite[\S 2]{chnp1} and \cite[\S 3]{chnp2}.

\begin{prop}
\label{prop:acyl_cy_top}
Let $M$ be {\rm ACyl} Calabi-Yau and let $n = \dim_{\C}{M}$.
\begin{enumerate}
\item \label{it:ACyl_pi1} 
If $\pi_1(M)$ is finite then $M$ has
a single end and
$\pi_1(M_\cyl) \to \pi_1(M)$ is surjective.
\item \label{it:dim3}
If $\pi_1(M)$ is finite and $n = 3$ then $M$ has holonomy $\sunitary3$.
\item \label{it:pi1_end}
If $M_\cyl = \bbr \times \Sph^1 \times D$ with $\pi_1(D)$ finite then
either $\pi_1(M)$ is finite or $M = M_\cyl / \Z_2$.
\end{enumerate}
\end{prop}

\begin{proof}
\ref{it:ACyl_pi1}
This follows from Lemma \ref{lem:cyl_cover} if we can show that every cover $\tM \to M$ has a single end. 
But otherwise $\tM$ would be a Calabi-Yau cylinder $\bbr \times \tX$ by the splitting theorem, and $b^1(\tX) = 0$ since $\pi_1(\tM)$ is finite, 
whereas $Jdt$ is a nontrivial harmonic 1-form on~$\tX$.

\ref{it:dim3}
Let $\tM$ be the universal cover of $M$. By \ref{it:ACyl_pi1}, this
is ACyl with a single end. 
If $\Hol(\tM)$ were a proper subgroup of $\sunitary3$ then by the de Rham
theorem $\tM$ would be a product
of simply-connected {lower-}dimensional submanifolds with even smaller holonomies, so one of these factors would be 
$\bbc$, contradicting that $\tM$ is ACyl. Now $\Hol(\tM) = \sunitary3$ implies $\Hol(M) = \sunitary3$ by \cite[4.1.10]{jnthesis}.

\ref{it:pi1_end}
If $\pi_1(M)$ is infinite then Corollary \ref{t:acyl_top} shows that $M$ has a finite cover $\tM = \Tor^k \times M'$ with $k \geq 1$ and
$M'$ simply-connected ACyl Calabi-Yau. Let $X'$ denote the cross-section of $M'$; this may not be connected. Then $\Tor^k \times X'$ covers $\Sph^1 \times D$, so $\pi_1(D)$ finite implies $k = 1$.
Since $\tM$ is K\"ahler, the space of parallel $1$-forms on $\tM$ inherits a
complex structure and therefore has even dimension. Hence $M'$ has a parallel
1-form.
Since $b^1(M') = 0$, $M'$ must have more than one end by Proposition \ref{p:acyl:hodge}\ref{it:harm:dr:1end}, hence splits as a cylinder, and so Lemma \ref{lem:cyl_cover} tells us that $M = M_\infty/\Z_2$.
\end{proof}

The simplest example of an ACyl Calabi-Yau manifold $M = M_\infty/\Z_2$ as in Proposition \ref{prop:acyl_cy_top}\ref{it:pi1_end}
is $M = (\bbr \times \Sph^1 \times D)/(-1,-1, \tau)$ with $D$ a $K3$ surface and $\tau$ a free anti-symplectic
involution of $D$; see Remark \ref{r:endrestrictions}. There is exactly one deformation family of such pairs $(D,\tau)$ (``Enriques surfaces''), so this is essentially the unique $M$ of this kind with $n \leq 3$.

\subsection{Holonomy considerations}\label{s:acyl_holonomy}
The main content of this section
is the proof of Theorem \ref{T:acyl:holo} but first we need to recall some background material.

The first ingredient is the well-known relation between special holonomy and
parallel spinors~\cite{wang89}. If $Z$ is a Riemannian spin manifold, then we
write $s(Z)$ for the number of parallel spinors on~$Z$.
A K\"ahler manifold $Z$ with trivial canonical bundle is spin and its spinor bundle is naturally identified with the total bundle of $(0,p)$-forms \cite[1.156]{besse87}, so that parallel spinors correspond to parallel \mbox{$(0,p)$-forms} and we always have $s(Z) \geq 1$ from $p = 0$. Let $d = \dim_\C Z$. If $\Hol(Z) \subseteq {\rm SU}(d)$ then $s(Z) \geq 2$ from the conjugate holomorphic volume form except if $Z$ is a point. If $Z$ is even hyper-K\"ahler, 
\ie $\Hol(Z) \subseteq {\rm Sp}(\frac{d}{2})$, then $s(Z) \geq \frac{d}{2} + 1$ from the powers of the conjugate holomorphic symplectic form.
If $\Hol(Z)$ is \emph{equal} to ${\rm SU}(d)$ or $ {\rm Sp}(\frac{d}{2})$,
then $s(Z) = 2$ if $d > 0$ and $s(Z) = \frac{d}{2} + 1$, respectively \cite{wang89}; 
this is a purely representation-theoretic fact. (The converse is false---in Remark \ref{r:endrestrictions}\ref{it:end_restr_2} we mentioned a K\"ahler $4$-fold with holonomy $({\rm SU}(2) \times {\rm SU}(2)) \rtimes \Z_2$ and $s = 2$.) Finally, it is helpful to keep in mind that \emph{all} holomorphic forms on a compact K\"ahler manifold with ${\rm Ric} \geq 0$ are parallel by the Bochner method; this still holds for all \emph{bounded} holomorphic forms in the ACyl case.

The second ingredient is the following structure theorem for compact Ricci-flat manifolds.

\begin{prop}[Calabi, Fischer-Wolf]\label{p:cfw}
Let $X$ be compact connected Ricci-flat and set $k = b^1(X)$. There exists a flat torus $\Tor^k$ and a finite normal Riemannian covering $\Tor^k \times X' \to X$ such that:
\begin{enumerate}
\item The deck group can be written as $\{(h(\psi),\psi): \psi \in \Psi\}$, where $\Psi$ is a finite group of isometries of $X'$ and $h$ is an injective homomorphism of $\Psi$ into the translation group of $\Tor^k$. 
\item $X'$ is compact connected Ricci-flat and carries no $\Psi$-invariant parallel vector fields.
\end{enumerate}
\end{prop}

This could be deduced from Remark \ref{R:acyl:rf}\ref{it:toruscover}
(\ie \cite[Thm 4.5]{fischer}) but is also proved directly in
\mbox{\cite[Thm 4.1]{fischer}}
without relying on the splitting theorem of \cite{cheeger71}. The proposition
generalises an earlier result for compact flat manifolds due to Calabi;
according to \cite{fischer}, Calabi was independently aware of this extension
to the compact Ricci-flat case, but had only published the result for $X$
K\"ahler.

\begin{proof}[Proof of Theorem \ref{T:acyl:holo}]
Since $M$ is simply-connected irreducible, either $\Hol(M) = \sunitaryn$, or $n$ is even and $\Hol(M) = {\rm Sp}(\frac{n}{2})$. The proof proceeds by analysing these two cases separately but in parallel, 
based on the facts reviewed above and on the following consequence of  Proposition \ref{prop:spcount}:
\begin{equation}\label{e:spinorcount}
s(M) = \frac{1}{2}s(M_\infty).
\end{equation}
The main aim is to rule out the Sp$(\frac{n}{2})$ case and show that in the SU$(n)$ case, $b^1(X)$ (which is always at least $1$ because of the parallel $1$-form $Jdt$) has to be exactly $1$. The latter then already implies a large part of the statement of Theorem \ref{T:acyl:holo}\ref{cylstruc} by applying Proposition \ref{p:cfw} for $k = 1$.

The analysis in fact relies on the conclusion of Proposition \ref{p:cfw}, \ie that we have a finite normal Riemannian covering $\Tor^k \times X' \to X$ whose deck group $\Psi$ is a finite group of isometries of $X'$ acting effectively on $\Tor^k$ by translations, and that $\Psi$ does not preserve any parallel vector fields on $X'$. We will use this to construct parallel spinors on $M_\infty$---almost always more than (\ref{e:spinorcount}) allows. \medskip\

\noindent {\bf Case 1: Holonomy} {\rm SU}$(n)$. Then $M$ has exactly two parallel holomorphic forms, so (\ref{e:spinorcount}) tells us that $s(M_\infty) = 4$. Now since $M_\infty$ is K\"ahler with respect to $J_\infty$, the parallel vector fields on $M_\infty$ are  closed under $J_\infty$, and so both $\bbr \times \Tor^k$ and $X'$ are $\Psi$-invariantly K\"ahler. Thus, $k = 2\ell + 1$ for some
$\ell \in \N_0$ and $\R \times \Tor^k$ is $\Psi$-invariantly Calabi-Yau. But this implies that $X'$ is not just Ricci-flat and $\Psi$-invariantly K\"ahler, but $\Psi$-invariantly Calabi-Yau---by contracting the holomorphic $n$-form pulled back from $M_\infty$ with the holomorphic $(\ell + 1)$-form on $\R \times \Tor^k$. We see that $\R \times \Tor^k$ has $2^{\ell + 1}$ parallel holomorphic $\Psi$-invariant forms, and $X'$ has at least $2$ unless $X'$ is a point, when there is only one. Thus, $s(M_\infty) \geq 2^{\ell + 2}$ if $X'$ is not a point, and $s(M_\infty) \geq 2^{\ell + 1}$ if $X'$ is a point. But $s(M_\infty) = 4$, and hence $\ell = 0$, $k = 1$, unless $\ell = 1$, $k = 3$, $n = 2$; we explicitly excluded the latter case.

If $k = 1$, then $\Psi$ is a finite subgroup of U$(1)$, so $\Psi = \langle \iota \rangle$ for some finite order isometry $\iota$ of $X'$. Moreover, we already know that $\iota$ preserves the complex structure and holomorphic volume form. Now $X'$ can have more parallel $(p,0)$-forms with $p > 0$ (\eg parallel vector fields), but if any of those were $\Psi$-invariant, this would immediately contradict the above counting inequalities. \medskip\

\noindent {\bf Case 2: Holonomy} Sp$(\frac{n}{2})$. In this case, $s(M_\infty) = n + 2$. Since $M_\infty$ is hyper-K\"ahler, the parallel vector fields on $M_\infty$ are closed under $I_\infty, J_\infty, K_\infty$, so $\bbr \times \Tor^k$ and $X'$ are  themselves $\Psi$-invariantly 
hyper-K\"ahler.
In particular, $k = 4\ell + 3$ for some $\ell \in \N_0$, and there are now even more $\Psi$-invariant
parallel holomorphic forms than before (but also more on $M_\infty$ to begin with):
$2^{2\ell + 2}$ on the $\R \times \Tor^k$ factor and at least
$\frac{n}{2} - \ell$ on the $X'$ factor (which equals $1$ if $X'$ is a point).
As before we deduce that $n + 2 \geq 2^{2\ell + 2}(\frac{n}{2}-\ell)$. We 
now argue that this leaves no possibility except for
$\ell = 0$, $k = 3$, $n = 2$; but this is the excluded case.
If the inequality fails for some $\ell$ and
$n$ then it also fails for the same $\ell$ and all larger $n$. But
$n \geq 2\ell + 2$, and the inequality does fail for $n = 2\ell + 2$ unless
$\ell = 0$. If $\ell = 0$ then $k = 3$, and the inequality clearly holds for
$n = 2$ but fails for all larger $n$.
\end{proof}

\begin{remark}
A similar argument of counting parallel spinors was used in
\cite[Thm 4.1.19]{jnthesis} to give a criterion for an ACyl 8-manifold
to have holonomy ${\rm Spin}(7)$.
\end{remark}

\section{Complex analytic compactifications}
\label{sec:cptfy}

\subsection{Proof of Theorem \ref{t:cptfy} modulo technical results}\label{s:cptfy_prelim}

 Let $M$ be simply-connected irreducible ACyl Calabi-Yau of complex dimension $n > 2$. By Theorem \ref{T:acyl:holo}\ref{nohk}, $M$ has holonomy SU$(n)$; hence there exists
precisely one parallel complex structure $J$ on $M$ up to sign.
Theorem \ref{T:acyl:holo}\ref{cylstruc} tells us that the cylindrical end
$M_\infty$ has a finite cover $\tM_\infty$ biholomorphic to
$\C^* \times D$ for some compact Ricci-flat K\"ahler manifold
$D$. Thus, $\tM_\infty$ can be compactified as $\C \times D$.
One would then expect that $M$ itself has a holomorphic
compactification $\oM$. This is true, but not obvious; it is also not obvious that $\oM$ is K\"ahler. 
However, once we know this, Theorem \ref{t:cptfy} follows reasonably quickly.

We begin by stating the technical compactification results. This requires some terminology. Let $\Delta$
denote the unit disc in $\C$ and put $\Delta^* = \Delta \setminus \{0\}$. Let $D$ be a compact complex manifold and $g_D$ an arbitrary Hermitian metric on $D$. Let $M_\infty^+ = \bbr^+ \times \Sph^1 \times D$ with product complex
structure $J_\infty$ and Hermitian metric $g_\cyl = dt^2 + d\anglen^2 + g_D$, where $\theta \in \Sph^1 = \R/2\pi\Z$ and $J_\cyl(\partial_t) = \partial_\theta$.

\begin{theorem}
\label{thm:cptfy}
Let $J$ be an integrable complex structure on $M_\cyl^+$ such that $J - J_\infty = O(e^{-\delta  t})$ with respect to $g_\cyl$ as $t \to +\infty$, including all covariant derivatives, for some $\delta > 0$.
Then there exists a diffeomorphism
$\Psi: M_\cyl^+ \to \Delta^* \times D$ such that $\Psi_*J$ extends as an integrable complex structure on $\Delta \times D$. Moreover, the submanifold $\{0\} \times D$ is complex and biholomorphic to $D$ with respect to this extension, and its normal
bundle is trivial as a holomorphic line bundle on $D$.
\end{theorem}

\begin{theorem}\label{t:kfy}
In the setting of Theorem \ref{thm:cptfy}, assume in addition that there exists
a $J$-K\"ahler form $\omega$ on $M_\infty^+$ such that
$\omega - \omega_\infty = O(e^{-\delta t})$ as $t \to +\infty$. Then
$\Delta \times D$ admits a $\Psi_*J$-K\"ahler form which coincides with
$\Psi_*\omega$ on $\{\frac{1}{2} < |w| < 1\}\times D$, where $w$ denotes the usual
complex coordinate on $\Delta$. 
\end{theorem}

Let us first see how the full statement of Theorem \ref{t:cptfy} now follows.

\begin{proof}[Proof of Theorem \ref{t:cptfy}]
We are given an $m$-sheeted covering $\tM_\cyl$ of $M_\cyl$ such
that $\tM_\cyl = \R \times \Sph^1 \times D$ for some compact
Ricci-flat Kähler manifold $D$. We can assume that the circle factor
has length $2\pi$. Pulling back $J$ from $M$ to $M_\infty^+$ by the ACyl
diffeomorphism and further pulling back by the covering map
$\tM_\cyl^+ \to M_\cyl^+$, we obtain a complex structure $\tilde{J}$ on
$\tM_\infty^+$. Theorem \ref{thm:cptfy} applies and produces a
$\tilde{J}$-holomorphic compactification
$\tilde{\Psi}: \tM_\infty^+ \hookrightarrow \Delta \times D$.
The action of the deck group of the covering $\tM_\cyl \to M_\cyl$
extends and preserves the divisor $D$ at infinity, so that $M$ itself
can be compactified as an orbifold $\oM$ by adding a suborbifold
$\oD = D/\langle\iota\rangle$. Averaging the K\"ahler form on $\Delta \times D$
provided by Theorem \ref{t:kfy} under the given holomorphic $\Z_m$-action, passing to the
quotient, and joining it to the ACyl K\"ahler form on $M$, we obtain an
orbifold Kähler form on $\oM$. 

Following \cite[Prop 2.2]{kovalev-lee08}, we can now easily see that $\oM$ must even be projective. As in the smooth case, it suffices to prove that $\oM$ does not admit any holomorphic $(2,0)$-forms.
But any holomorphic $(p,0)$-form on $\oM$ restricts to an asymptotically
translation-invariant holomorphic $(p,0)$-form on $M$, and since 
$\Hol(M) = \sunitary{n}$ by Theorem \ref{T:acyl:holo}\ref{nohk},
a standard Bochner argument then shows that there are no such forms if $0 < p < n$ 
(up to a complex multiple, the only nonzero bounded holomorphic form on $M$ is the parallel holomorphic volume form, which has a first order pole along $\oD$).

The fact that the plurigenera $h^0(\oM, \ell K_{\oM})$ vanish for all $\ell > 0$ is even easier. Indeed, $-K_{\oM}$ is an effective line bundle, so that $-\ell K_{\oM}$ has a nonzero holomorphic section for all $\ell > 0$. Thus, if $\ell K_{\oM}$ had a nonzero holomorphic section as well, then pairing these two sections would yield a nonzero holomorphic function on $\oM$, proving that $-\ell K_{\oM}$ is trivial, which is clearly not the case. See Yau \cite[p.~247]{yau:ICM78} for a more abstract argument that works in much greater generality.

As for the fibration of $\oM$ by $|m\oD|$, observe that we have a short exact sequence 
\begin{align}\label{e:iffyshortexactsequence}
0 \to \calo_{\oM} \to \calo_{\oM}(m\oD) \to \calo_{m\oD}(m\oD) \to 0.
\end{align}
The cokernel sheaf $\calo_{m\oD}(m\oD)$ is the
sheaf of sections of the restriction of the line bundle $m\oD$ to the scheme $m\oD$, \ie an infinitesimal ``thickening'' of $\oD$.
This yields a long exact sequence 
\[ 0 \to H^0(\calo_{\oM}) \to H^0(\calo_{\oM}(m\oD)) \to H^0(\calo_{m\oD}(m\oD))
\to H^1(\calo_{\oM}) . \]
Notice that $H^{0,1}(\oM) = 0$. Thus, if we knew that $\calo_{m\oD}(m\oD)$ had a section, then we would find that 
$h^0(\calo_{\oM}(m\oD)) = 2$, so $|m\oD|$
is a pencil. Now the line bundle $\ell \oD$ is trivial on $\oD$ for all $\ell \in m\Z$, but this does not imply that it is trivial on $m\oD$ except if $m = 1$ (on the other hand, if $m = 1$, it is then also clear that $|\oD|$ has no base locus). However, we have a general ``lifting'' sequence
\begin{align}\label{e:lift}
0 \to \mathcal{O}_{k\oD}(\ell \oD) \to \mathcal{O}_{(k+1)\oD}((\ell + 1)\oD) \to \mathcal{O}_{\oD}((\ell + 1)\oD) \to 0
\end{align}
for every $k \in \N_0$ and $\ell \in \Z$. Setting $k = \ell = m-1$ and taking cohomology yields
\begin{align}\label{e:lift2}
H^0(\calo_{m\oD}(m\oD)) \to H^0(\calo_{\oD}(m\oD)) \to H^1(\calo_{(m-1)\oD}((m-1)\oD)).
\end{align}
Thus, if the $H^1$ term vanishes (\eg if $m = 1$), then our trivialising section extends from $\oD$ to $m\oD$. We can get a handle on this $H^1$ by taking cohomology in the \emph{upstairs counterpart} to (\ref{e:lift}):
$$
H^1(\mathcal{O}_{kD}(\ell D)) \to H^1(\mathcal{O}_{(k+1)D}((\ell + 1)D)) \to H^1(\mathcal{O}_{D}((\ell + 1)D)).
$$
Now suppose that $b^1(D) = 0$ (which in fact follows from $m = 1$ in our
setting). Since $\ell D$ is trivial on $D$ for all $\ell \in \Z$, the third
term vanishes, and so induction on $k \in \N_0$ yields
$H^1(\mathcal{O}_{kD}(\ell D)) = 0$ for all $k \in \N_0$ and $\ell \in \Z$. 
In particular, setting $k = \ell = m-1$ and taking $\Z_m$-invariants, we find
that the obstruction space in (\ref{e:lift2}) vanishes and the trivialising
section of $\calo_{\oD}(m\oD)$ does extend.
\end{proof}

\begin{remark}\label{r:fibering_obstructions}
In Example \ref{ex:end}, we have $m = 2$, so the formal obstruction space in
(\ref{e:lift2}) coincides with the $\Z_2$-invariants in $H^1(\calo_{D}(D))$. To
compute these, it is helpful to identify this $H^1$ with the space of constant
$(0,1)$-forms on $D$ taking values in the normal bundle. The two standard
generators are then $d\bar{x} \otimes \frac{\partial}{\partial w}$ and
$d\bar{y} \otimes \frac{\partial}{\partial w}$, with $w = re^{-i\theta}$, as in
Example \ref{ex:end}. But these are obviously $\Z_2$-invariant and so the
formal obstruction space to fibering $\oM$ by $|2\oD|$ is $2$-dimensional.
\end{remark}

It remains to prove Theorems \ref{thm:cptfy}--\ref{t:kfy}.
This will be done in the following two subsections.

\subsection{Holomorphic compactification}\label{s:hol_cptf} We begin with a discussion of the main difficulties and an outline of the argument. For $(t, \anglen) \in \bbrp \times \Sph^1$ let $w = e^{-t-i\anglen}$. Then the
diffeomorphism
\begin{equation}
\label{eq:expmap}
\Psi_\infty: M_\infty^+ \to \Delta^* \times D, \;\,
(t, \anglen, x) \mapsto (w,x),
\end{equation}
pushes $J_\infty$ forward to the product complex structure $J_0$ on $\Delta^* \times D$, which is clearly compactifiable. However, $(\Psi_\infty)_*J$ may not even be \emph{uniformly bounded} with respect to $g_0 = |dw|^2 + g_D$ as $w \to 0$.
Specifically, for any section $s$ of
$(T^*\Delta)^a \otimes (T^*D)^b \otimes (T\Delta)^c \otimes  (TD)^d$ over $\Delta^* \times D$ we have that 
\begin{equation}\label{e:compare_norms}
|\Psi_\infty^*s|_{g_\cyl} = O(e^{-\delta t}) \Longleftrightarrow |s|_{g_0} = O(|w|^{\delta+c-a}).
\end{equation}
Thus, in terms of the decomposition $T\Delta \oplus TD$, the off-diagonal $T^*\Delta \otimes TD$ components of $(\Psi_\infty)_*J$ can be expected to blow up like $|w|^{-1 + \delta}$ as $|w|\to 0$; all the remaining components of $(\Psi_\infty)_*J$ are at least $C^{0,\delta}$ H\"older continuous along $\{0\} \times D$, but not---a priori---smooth.

The key point in resolving this problem is that the integrability of $J$ is equivalent to a nonlinear first-order differential equation: the vanishing of the Nijenhuis torsion. This equation is not elliptic, but the lack of ellipticity can be traced back to diffeomorphism invariance. In other words, there is
hope that a suitable improvement of $\Psi_\infty$ \emph{will} map $J$ to a smooth complex structure on $\Delta \times D$.

The proof of Theorem \ref{thm:cptfy} now follows in three steps. Step 1 shows how to construct a gauge in which $J$ coincides with $J_\infty$ in directions tangent to $\R^+ \times \Sph^1 \times \{x\}$ ($x \in D$). This already fixes the discontinuity of $(\Psi_\infty)_*J$ at infinity. Based on this, Step 2 then uses an elliptic regularity argument along these cylinders to show that the pushforward of $J$ is actually smooth at infinity; this involves the $C^{1,\alpha}$
Newlander-Nirenberg theorem of \cite{nij-woolf}. Step 3 deals with the normal bundle.

\subsection*{Step 1: Gauge fixing}

The pushforward $(\Psi_\infty)_*J$ fails to be continuous at $\{0\} \times D$ if and only if the $J_\infty$-holomorphic cylinders $\R^+ \times \Sph^1 \times \{x\}$ are not $J$-holomorphic. This suggests replacing $\Psi_\infty$ by $\Psi_\infty \circ F^{-1}$, where $F \in {\rm Diff}(M_\infty^+)$ maps each $\R^+ \times \Sph^1 \times \{x\}$ onto a $J$-holomorphic curve exponentially asymptotic to it. For this, it suffices to find $(J_\infty, J)$-holomorphic maps $F_x: \R^+ \times \Sph^1 \times \{x\} \to M_\infty^+$ that are exponentially asymptotic to the identity and depend smoothly on $x \in D$. 

To solve this problem, it is helpful to invoke some of the usual formalism for
the construction of holomorphic curves. Given $x \in D$ and the tautological
map $f_{0,x}: \R^+ \times \Sph^1 \to
\R^+ \times \Sph^1 \times \{x\} \subset M_\infty^+$, let $\mathcal{E}_x$ denote
the space of all smooth embeddings $f: \R^+ \times \Sph^1 \to M_\infty^+$
exponentially asymptotic to~$f_{0,x}$, and let
$\mathcal{V}_x \to \mathcal{E}_x$ denote the natural vector bundle whose
fibre at $f \in \mathcal{E}_x$ is the vector space of all exponentially
decaying vector fields along $f$. With a very slight abuse of notation, we then
have a section $\bar{\partial} \in \Gamma(\mathcal{E}_x, \mathcal{V}_x)$ whose
value at $f$ is given by $\bar{\partial}f
\equiv \frac{\partial f}{\partial t} + J\frac{\partial f}{\partial \theta}$.
Restricting to the region $t \gg 1$, we can assume that
$\|\bar{\partial}f_{0,x}\| \ll 1$ uniformly in $x$, and our goal is to
construct a genuine zero $f_x \in \mathcal{E}_x$ of the section
$\bar{\partial}$ which, as an embedding of $\R^+ \times \Sph^1$ into
$M_\infty^+$, depends smoothly on $x$.

We begin by choosing a chart for $\mathcal{E}_x$ near $f_{0,x}$ (modelled on a definite neighbourhood of the origin in $T_{f_{0,x}}\mathcal{E}_x$), as well as a trivialisation for $\mathcal{V}_x$ over it. There are no canonical choices for either, but a natural and useful way is to apply the exponential map and parallel transport with respect to $g_\infty$. This now allows us to view
$\bar{\partial}$ $\in$ $\Gamma(\mathcal{E}_x, \mathcal{V}_x)$ as a nonlinear first-order differential operator $\bar{\partial}_x$ acting on some definite open neighbourhood of the origin in $C^{k,\alpha}_\delta(\R^+ \times \Sph^1, f_{0,x}^*TM_\infty^+)$. We have $\|\bar{\partial}_x(0)\| \ll 1$, and the linearisation $\mathcal{L}_x$ of $\bar{\partial}_x$ at $0$ satisfies $\mathcal{L}_x = \mathcal{L} + \mathcal{U}_x$, where
$$
\mathcal{L} V \equiv \frac{\partial V}{\partial t} + J_\infty\biggl(\frac{\partial V}{\partial\theta}\biggr), \;\,  \|\mathcal{U}_x\|_{\rm op} \ll 1.
$$
Also, $\mathcal{U}_x$ 
varies smoothly with $x$ if we use parallel transport with respect to the Chern connection 
of $(M_\infty^+, g_\infty)$ in order to identify $C^{k,\alpha}_\delta(\R^+ \times \Sph^1, f_{0,x}^*TM_\infty^+)$ with $C^{k,\alpha}_\delta(\R^+ \times \Sph^1, f_{0,y}^*TM_\infty^+)$ for different nearby points $x,y \in D$. Notice that these identifications 
do not affect the operator $\mathcal{L} \equiv \bar{\partial}_{J_\infty}$.

Since the $\bar{\partial}$-equation in one complex variable with values in a complex vector space is elliptic, we can apply Remark \ref{r:right_inverse} to construct some bounded right inverse $\mathcal{R}_x$ to $\mathcal{L}$ at any given point $x \in D$. Since $\mathcal{R}_x$ is not unique, some care is needed to ensure that $\mathcal{R}_x$ depends smoothly on $x$. For this we choose a finite cover of $D$ by open sets $U_1, \ldots, U_N$ with basepoints $x_i \in U_i$ such that $f_{0,x_i}(t,\theta)$ can be joined to $f_{0,x}(t,\theta)$ by a unique Chern geodesic with respect to $g_\infty$ for all $x \in U_i$ and for all $\theta \in \Sph^1$ and $t \gg 1$. Moreover, let  $\chi_1, \ldots, \chi_N$ be a partition of unity subordinate to this cover. Choosing $\mathcal{R}_{x_i}$ as above for all $i \in \{1,\ldots, N\}$, we can then put $\mathcal{R}_x \equiv \sum \chi_i(x)(\mathcal{P}_{x_ix} \circ \mathcal{R}_{x_i} \circ \mathcal{P}_{xx_i})$ for all $x \in D$, with $\mathcal{P}_{xy}: C^{k-1,\alpha}_\delta(\R^+ \times \Sph^1, f_{0,x}^*TM_\infty^+) \to C^{k-1,\alpha}_\delta(\R^+ \times \Sph^1, f_{0,y}^*TM_\infty^+)$ denoting Chern parallel transport.

The desired holomorphic maps $f_x$ are then obtained by an elementary fixed point argument---by
iterating the contraction mappings $\mathcal{R}_x \circ (\mathcal{L}-\bar{\partial}_x)$ on some neighbourhood of the origin.

\subsection*{Step 2: Elliptic regularity.}

If we define $\Psi \equiv \Psi_\infty \circ F^{-1}$ with $F \in {\rm Diff}(M_\infty^+)$ as in Step 1, then we know that $\Psi_*J$ is equal to the standard complex structure $J_0$ on the horizontal 
subbundle $T\Delta$ of $T(\Delta \times D)$. In particular, by (\ref{e:compare_norms}), $\Psi_*J$ extends $C^{0,\delta}$ across $\{0\} \times D$. We will now first explain
how the vanishing of the Nijenhuis torsion of $J$ implies that $\Psi_*J$ automatically extends $C^{1,\alpha}$.

Since $\tilde{J} \equiv F^*J$ satisfies $\tilde{J}\partial_t = \partial_\theta$, the vanishing of the torsion of $J$ (or $\tilde{J}$) implies that
\begin{align}\label{e:zerotorsion}
\pd{\tilde{J}}{t} + \tilde{J} \circ \pd{\tilde{J}}{\anglen} = 0.
\end{align}
Thus, the endomorphism field $K \equiv \tilde{J}-J_\infty$, which is exponentially decaying, satisfies the following quadratic perturbation of the $\mathcal{L}$- or $\bar{\partial}_{J_\infty}$-equation:
\begin{equation}\label{e:quadratic_K}
\mathcal{L}K + K \circ \frac{\partial K}{\partial\theta} = 0.
\end{equation}
Using the right inverses $\mathcal{R}_x$ to $\mathcal{L}$ constructed above, we can therefore write 
\begin{equation}\label{eq:iterate}
K = \tilde{K} - \mathcal{R}_x(K \circ \partial_\theta K), \;\,\tilde{K} \in {\rm ker}(\mathcal{L}).
\end{equation}
Observe that ${\rm ker}(\mathcal{L})$ consists of Laurent series in $w$ whose coefficients are constant sections of the vector bundle ${\rm End}_{\R}(f_{0,x}^*TM_\infty^+)$ over $\R^+ \times \Sph^1$. It is clear from (\ref{eq:iterate}) that $\tilde{K}$ depends smoothly on $x$. Thus, by the Cauchy integral formula, each of its Laurent coefficients depends smoothly  on $x$. 

Since $K$ already decays exponentially and $\mathcal{R}$ preserves the decay rate, (\ref{eq:iterate}) yields that
\begin{equation}
\label{jexpansioneq}
K = \tilde{J} - J_\infty = w \tilde{K}_1 + O(|w|^{1+\alpha})
\end{equation}
for all $\alpha \in (0,1)$, by iteration. Here $\tilde{K}_1 = \tilde{K}_{1}(x)$ denotes a constant section of ${\rm End}_{\R}(f_{0,x}^*TM_\infty^+)$ that depends smoothly on $x$, and the product with $w$ is again understood in the sense that $iA \equiv J_\infty \circ A$ for any endomorphism $A$. Moreover, denoting $L \equiv K - w \tilde{K}_1$, we have that $\partial_w^a\partial_x^b L = O(|w|^{1+\alpha-a})$ for all $a,b \in \N_0$. We now claim that (\ref{jexpansioneq}) implies that $\Psi_*J$ extends $C^{1,\alpha}$ to $\Delta \times D$ as desired. Indeed, since $(\Psi_\infty)_*K$ vanishes on the horizontal subbundle $T\Delta$, the same is true for the slicewise constant section $(\Psi_\infty)_*\tilde{K}_1$, which therefore extends $C^\infty$ to $\Delta \times D$. Thus, it remains to consider $(\Psi_\infty)_*L$; but, using (\ref{e:compare_norms}) and the above derivative properties of $L$, it is clear that $|\nabla_{g_0}(\Psi_\infty)_*L|_{g_0} = O(|w|^\alpha)$. 

The version of the Newlander-Nirenberg theorem of \cite[Thm II]{nij-woolf} now tells us that there exists a complex analytic atlas on $\Delta \times D$ whose coordinate functions are $\Psi_*J$-holomorphic and $C^{1,\frac{\alpha}{n}}$ with respect to $g_0$. (Thus, in our main application---Theorem \ref{t:cptfy}---we would now already know that $M$ is holomorphically compactifiable by adding a divisor.)
However, we are claiming more: $\Psi_*J$ in fact extends smoothly \emph{as a tensor field}, not just modulo $C^{1,\frac{\alpha}{n}}$ local diffeomorphisms.

To prove this, note that \cite[Thm II]{nij-woolf} in particular tells us that there exist sufficiently many local $\Psi_*J$-holomorphic functions so that $\Psi_*J$ can be recovered from their differentials as a tensor field.
It therefore suffices to check that $\Psi_*J$-holomorphic functions are smooth. Let $z$ be $\Psi_*J$-holomorphic on a neighbourhood of a point in $\{0\} \times D$. Since $\Psi_*J$ coincides with $J_0$ on $T\Delta$, we immediately find that $z$ is $J_0$-holomorphic on each horizontal slice. In other words, we have
\begin{equation}\label{e:powerseries}
z = z_0 + wz_1 + w^2z_2 + \cdots,
\end{equation}
and the Cauchy integral formula expresses the coefficients $z_i = z_i(x)$ in terms of $z(w,x)$ with $w \neq 0$. But we already know that $z$ is smooth for $w \neq 0$ because $\Psi_*J$ is. 

\begin{remark}
It is conceivable that a similar (if more difficult) argument could work for the tensor $K$ itself, by refining the partial expansion \eqref{jexpansioneq} to a complete one based on \eqref{e:quadratic_K}.
\end{remark}

\subsection*{Step 3: Normal bundle to the compactifying divisor.} We identify $J$ and $\Psi_*J$ for convenience.
It is clear that $\{0\} \times D$ is a $J$-complex submanifold of $\Delta \times D$,
biholomorphic to $D$. It remains only to prove that the normal bundle $N_D$ is
holomorphically trivial with respect to $J$.
Since every slice $\Delta \times \{x\}$ is a $J$-complex
submanifold by construction, the complex tangent vector field $\contra{w}$ is of type $(1,0)$ with respect to $J$. We show
that the section of $N_D$ that it induces is $J$-holomorphic; recall here that elements of $N_D$ are by definition cosets modulo the complex tangent space of $D$. 

For every $x \in D$ there is a $J$-holomorphic function $z$ on a neighbourhood
$U$ of $(0,x)$ in $\Delta \times D$ which vanishes to order $1$ along $D$.
Let $U' \equiv U \cap (\{0\} \times D)$. Then $dz$ is a trivialising holomorphic
section of $N^*_D$ over $U'$, so $\contra{w}$ will map to a
holomorphic section of $N_D$ if and only if $dz(\contra{w}) = \frac{\partial z}{\partial w}$ is
a holomorphic function on $U'$. Now if we expand $z$ as a power series in $w$ as in (\ref{e:powerseries}),
\begin{equation}\label{e:powerseries2}
z = wz_1 + w^2z_2 + \cdots,
\end{equation}
then $\frac{\partial z}{\partial w} = z_1$ on $U'$. On the other hand, applying the $\bar{\partial}$-operator of $J$ to (\ref{e:powerseries2}) yields
\begin{equation}\label{e:powerseries3}
0 = \db_Jz =  w\db_{J} z_1 + (z_1 + 2wz_2) \db_{J} w +O(|w|^2).
\end{equation}
In order to conclude from this that $\db_J z_1 = 0$ along $U'$, we need to know that $\db_J w = o(|w|)$ in terms of $g_0$. But $w$ is $J_0$-holomorphic, so $\db_J w = \frac{i}{2}dw \circ (J - J_0)$. Now the only components of $J - J_0$ not annihilated by $dw$ are the 
$T^*D \otimes T\Delta$ ones, whose $g_0$-length is $|w|$ times their $g_\infty$-length, and the $g_\infty$-length of $J-J_0$ certainly goes to zero; in fact, by \eqref{jexpansioneq}, it is even $O(|w|)$.
\hfill $\Box$

\subsection{Kähler compactification}
We have found two different proofs of Theorem \ref{t:kfy}, both of which will be explained in this section. We will assume the conclusion of Theorem \ref{thm:cptfy} but usually ignore the diffeomorphism $\Psi$. Both proofs begin by writing the ACyl K\"ahler form on $M_\infty^+$ as
\begin{equation}\label{e:acyl_kf}
\omega = i\partial\bar{\partial} t^2 + \omega_D + O(e^{-\delta t}).
\end{equation}
Here $i\partial\bar{\partial}$ is with respect to $J$, and $\omega_D$ is pulled back from the $D$ factor in $M_\infty^+ = \R^+ \times \Sph^1 \times D$; in 
particular, $\omega_D$ is closed, but not necessarily $(1,1)$ with respect to $J$.
The most intuitive approach to ``compactifying'' $\omega$ may be to replace $t^2$ by the K\"ahler potential of a half-cylinder with a spherical cap attached, but there are two (related) problems with this: (1) The $O(e^{-\delta t})$ terms have no reason to extend smoothly to the complex compactification. (2) The capped-off potential will be $O(e^{-2t})$, so the $O(e^{-\delta t})$ errors may dominate and the modified form may not be positive. 

Our first proof uses ideas from Section \ref{s:hol_cptf} to fix (1) and, by consequence, (2). Specifically, recall that the cylinders $\R^+ \times \Sph^1 \times \{x\}$ are $J$-holomorphic by the construction of $\Psi$. Solving $\bar{\partial}$-equations along these cylinders, we will be able to construct $u = O(e^{-\delta t})$ supported far out in $M_\cyl^+$ such that the exponential errors of the \emph{corrected} K\"ahler form $\omega + i\partial\bar{\partial}u$ do extend smoothly. It then follows immediately from this that we can cap off the $i\partial\bar{\partial}t^2$ part without losing positivity.

The second proof will emphasise positivity over smoothness. We back up one step and cap off the 
infinite end of the cylinder metric on $\R^+ \times \Sph^1$ by a \emph{cone} of angle $2\pi \epsilon$ ($\epsilon \ll \delta$) rather than a disk or hemisphere. This amounts to replacing 
$t^2$ in (\ref{e:acyl_kf}) by $e^{-2\epsilon t}$ rather than $e^{-2t}$ at infinity. 
Then (2) is not a problem to begin with, but (1) now looks worse.
However, geometrically, we have created
an \emph{edge singular} K\"ahler metric on the compactified space, and we will see that this ``edge metric'' has continuous local K\"ahler potentials. It can therefore be regularised using
the method of \cite{V}.  \\

\noindent \emph{First proof of Theorem \ref{t:kfy}.} By translating $t$, we can assume without loss that (\ref{e:acyl_kf}) holds on all of $M_\infty^+ = \R^+ \times \Sph^1 \times D$ and that the exponential errors are bounded by $\epsilon e^{-\delta t}$, where $\epsilon$ is as small as we like. The moral point of the proof is to correct $\omega$ by $i\partial\bar{\partial}u$, with $u$ exponentially decaying and small (obtained by solving $\bar{\partial}$-equations on each horizontal slice), in order to arrange that the exponential errors of $\omega + i\partial\bar{\partial}u$ have a power series expansion in $w$, or are at least smooth at infinity.

Let $\psi$ denote the $O(e^{-\delta t})$ error terms 
in (\ref{e:acyl_kf}). We begin by noting that $\psi = d(\eta + \bar{\eta})$
for some $(0,1)$-form $\eta = O(e^{-\delta t})$. Indeed, we can write $\psi = dt \wedge \psi_1 + \psi_2$, where $\psi_i = O(e^{-\delta t})$ is a $1$-parameter family of $i$-forms on $X$; the closedness of $\psi$ implies that $\xi(t,x) \equiv - \int_t^\infty \psi_1(s,x) \, ds$ is a primitive for $\psi$ and we let $\eta$ be the $(0,1)$-part of $\xi$.
Next, we solve $\bar{\partial}f_x = \eta|_{C_x}$ along $C_x = \bbrp \times \Sph^1 \times \{x\} \subset M_\cyl^+$ for each $x \in D$ in such a way that the $f_x$ depend smoothly on $x$ with $|f_x| \leq C\epsilon e^{-\delta t}$. In particular, we obtain a smooth complex-valued function $f$ on $M_\infty^+$, and we now put $u \equiv -2\, {\rm Im} f$.

It is immediate that 
\begin{equation}\label{e:omega_modified}
\omega + i\partial\bar{\partial}u = i\partial\bar{\partial}t^2 + \omega_D + d(\kappa  + \bar{\kappa}) > 0,\;\, \kappa \equiv \eta - \bar{\partial} f = O(e^{-\delta t}),
\end{equation}
and the restriction of $\kappa$ to each of the usual $J$-holomorphic cylinders $C_x$ vanishes by construction. Thus, for all $(t,\theta,x)$, we can view $\kappa|_{(t,\theta,x)}$ as an element of $V_x \equiv T^*_xD \otimes \C$, which we in turn view as a real vector space (with an obvious complex structure, but this will not be relevant). Now $V_x$ has a natural family of complex structures
 $\mathcal{J}_x(t,\theta)$ defined by the pullback action of $-J$, which leaves 
$T^*D \subset T^*M_\infty^+$ invariant because the action of $J$ on vectors preserves $T\Delta \subset TM_\infty^+$. Given any fixed~$x$, we then view $\kappa$ as a function on $\R^+ \times \Sph^1$ taking values in $V_x$, and we claim that 
\begin{equation}\label{e:dbar_kappa}
\frac{\partial\kappa}{\partial t} + \mathcal{J}_x \frac{\partial \kappa}{\partial \theta} = 0.
\end{equation}
To see this, first note that 
$\partial_t \kappa + \mathcal{J}_x\partial_\theta\kappa= (\partial_t + i\partial_\theta) \,\llcorner\,\bar{\partial}\kappa$, where $\bar{\partial}\kappa$ means the $\bar{\partial}$-derivative of $\kappa$ as a $(0,1)$-form on $M_\infty^+$; this is proved using that $\bar{\partial}\kappa = \frac{1}{2}(d\kappa - J^*d\kappa)$, that $\kappa$ is vertical, and that $T\Delta$ is $J$-invariant. On the other hand,  $\bar{\partial}\kappa$ is equal to the $(0,2)$-part of $-\omega_D$ by (\ref{e:omega_modified}), and
$$\omega_D^{0,2}(X,Y) = \frac{1}{4}(\omega_D(X,Y) - \omega_D(JX, JY) + i(\omega_D(JX, Y) + \omega_D(X, JY ))),$$
so if $X$ is horizontal then this vanishes for every $Y$ since $JX$ is horizontal as well. 

We now exploit the $\bar{\partial}$-type equation (\ref{e:dbar_kappa}), together with the smoothness at infinity of $\mathcal{J}_x$ from Section \ref{s:hol_cptf}, to deduce that $\kappa$ is itself smooth at infinity. For this we pass to the disk picture, writing $w = u + iv \in \Delta$ with $u = e^{-t}\cos\theta$ and $v = -e^{-t}\sin\theta$. Then (\ref{e:dbar_kappa}) yields
$\partial_u \kappa + \mathcal{J}_x \partial_v \kappa = 0$ on $\Delta^*$,
where the function $\kappa: \Delta \to V_x$ is $C^{0,\delta}$ H\"older continuous, smooth away from the origin, and zero at the origin itself, and the function $\mathcal{J}_x: \Delta \to {\rm End}_{\R}(V_x)$ is smooth with $\mathcal{J}_x^2 = -{\rm id}_{V_x}$. Smoothness of $\kappa$ at $w = 0$ now follows from elementary elliptic regularity; for example, by applying $\partial_u - \mathcal{J}_x\partial_v$ we 
can deduce that $\Delta \kappa + \mathcal{K}_x(\partial_v\kappa) = 0$, where $\mathcal{K}_x \equiv \partial_u \mathcal{J}_x - \mathcal{J}_x \partial_v\mathcal{J}_x$ is smooth, and using $\kappa = O(|w|^\delta)$ and $d\kappa =O(|w|^{\delta-1})$ one checks that $\kappa \in W^{1,2}$ solves this
 equation in the weak sense at $w = 0$. Smooth dependence of $\kappa = \kappa_x(u,v)$ on the parameter $x$ is then standard.

To conclude the proof, we will now verify that the closed $(1,1)$-form
\begin{equation}\label{e:capped_off_form}
\omega_D + d(\kappa + \bar{\kappa}) + i\partial\bar{\partial}((1-\chi)t^2 + \chi\phi)
\end{equation}
on $M_\infty^+$ is positive and extends to a smooth K\"ahler form on $\Delta \times D$, where $\chi(t)$ is a cut-off function with $\chi \equiv 0$ on $\{t < 1\}$ and $\chi \equiv 1$ on $\{t > 2\}$, and $\phi(t)$ is a convex function with
\begin{align*}
\phi(t) = \begin{cases} 
t^2 + C_1 t + C_2 &{\rm for}\;\,t \in (0, 3),\\
C_3 e^{-2t}
& {\rm for}\;\,t \in (5, \infty),
\end{cases}
\end{align*}
the absolute constants $C_1, C_2, C_3$ being chosen so that the two branches of the definition match up at $t = 4$ including first and second derivatives. This is understood in the sense that we have already shifted $t$ so that $|J - J_\infty| + |\kappa| \leq \epsilon e^{-\delta t}$ on the whole of $M_\infty^+$, with $\epsilon$ as small as necessary.

Since we already know that $J, \kappa$ extend smoothly, and since $e^{-2t} = |w|^2$ is smooth on $\Delta \times D$, it is clear that the form in (\ref{e:capped_off_form}) extends smoothly. Positivity for $t \in (0,3)$ is also clear, given that we can assume that $|i\partial\bar{\partial}t| \leq \epsilon$. For $t \in (3,\infty)$, we would be stuck if all we knew was that $\kappa = O(e^{-\delta t})$
for {some} $\delta > 0$ (even $\delta = 1$) because such terms can easily swamp $i\partial\bar{\partial}\phi$. But $d(\kappa + \bar{\kappa}) + \omega_D^{2,0} + \omega_D^{0,2}$
extends smoothly and vanishes along $D$, while $i\partial\bar{\partial}\phi + \omega_D^{1,1}$ is smooth and positive near $D$.
\hfill $\Box$

\begin{remark}
Unlike $\kappa$ of (\ref{e:omega_modified}), the $(0,1)$-form $\eta$ describing the exponential errors in  (\ref{e:acyl_kf}) has no reason to be smooth at infinity {even though} $(\partial_t + i\partial_\theta) \,\llcorner\, \bar{\partial}\eta = 0$. Of course we expect that $\kappa$ really is more regular than $\eta$, but there is a subtle point here: formally, (\ref{e:dbar_kappa}), which gives regularity for $\kappa$, is derived from $(\partial_t + i\partial_\theta) \,\llcorner\, \bar{\partial}\kappa = 0$, which also holds for $\eta$, using \emph{only} that $\kappa$ is vertical.
\end{remark}

\begin{remark}\label{r:alter_reg}
We also mention an alternative approach to regularity for $\kappa$. In the disk picture, pick a $\C$-basis $\{\kappa_i\}$ for $(V_x, \mathcal{J}_x(0))$, so that $\{\kappa_i\}$ still is a $\C$-basis for  $(V_x, \mathcal{J}_x(w))$ if $|w|$ is small. Each $\kappa_i$ 
trivially solves (\ref{e:dbar_kappa}), and using (\ref{e:zerotorsion}) one can compute that $\mathcal{J}_x\kappa_i$ solves (\ref{e:dbar_kappa}) too. We now expand $\kappa = \sum f_i\kappa_i$ with $f_i: \Delta^* \to \C$, again in the sense that $i \in \C$ acts on $V_x$ by $\mathcal{J}_x$.
Then $\kappa$ solves (\ref{e:dbar_kappa}) if and only if all the $f_i$ are holomorphic, so we can apply the removable singularities theorem.

We can interpret this argument as follows. By (\ref{e:zerotorsion}), the $(0,1)$-part of the trivial connection $\nabla$ on the complex vector
bundle $(V_x, \mathcal{J}_x)$ is a $(0,1)$-connection, \ie
$\nabla^{0,1}(f\kappa) = \bar{\partial}f \otimes \kappa + f \nabla^{0,1}\kappa$.
We could have worked in any local frame $\{\kappa_i\}$ with $\nabla^{0,1}\kappa_i = 0$. Such frames exist for every $(0,1)$-connection over the disk (\ie the 
$(0,1)$-connection is integrable, defining a holomorphic structure).
\end{remark}

\noindent \emph{Second proof of Theorem \ref{t:kfy}.}
We again assume that all $O(e^{-\delta t})$ error terms are uniformly as small as necessary on the whole cylinder $M_\infty^+$, and we write our ACyl K\"ahler form as 
$\omega = i\partial\bar{\partial}t^2 + \omega_D + \psi$ with $\psi = O(e^{-\delta t})$. We then construct the following closed $(1,1)$ modification $\tilde{\omega}$ of $\omega$:
\begin{equation}\label{e:pencil_form}
\tilde{\omega} = i\partial\bar{\partial}((1-\chi)t^2 + \chi\phi)  + \omega_D + \psi,
\end{equation}
where $\chi(t)$ is a cut-off with $\chi \equiv 0$ on $\{t < 1\}$ and $\chi \equiv 1$ on $\{t > 2\}$, and $\phi(t)$ is convex with
\begin{align*}
\phi(t) = \begin{cases} 
t^2 + C_1 t + C_2 &{\rm for}\;\,t \in (0, 3),\\
C_3 e^{-2\epsilon t}
& {\rm for}\;\,t \in (5, \infty).
\end{cases}
\end{align*}
Here $\epsilon > 0$ is fixed but \emph{strictly smaller} than $\frac{\delta}{2}$, and $C_1,C_2,C_3$ are determined by $\epsilon$ so that the two branches match up at $t = 4$ including first and second derivatives. This construction is similar to  (\ref{e:capped_off_form}), except that now the reason why (\ref{e:pencil_form}) defines a \emph{positive} form on $M_\infty^+$ is that the
good term $i\partial\bar{\partial}\phi + \omega_{D}^{1,1} > 0$ swallows the error $\psi + \omega_D^{2,0} + \omega_D^{0,2}$ by Cauchy-Schwarz because $\epsilon$ is small.

Now $\tilde{\omega}$ does not extend smoothly, but the Riemannian metric associated with $\tilde{\omega}$ only has a fairly mild (conical with cone angle $2\pi\epsilon$) singularity along the compactifying divisor $\{0\} \times D$. We pursue this idea by
proving that $\tilde{\omega}$ has local potentials that remain continuous at the divisor. 

For this, 
we first cover a neighbourhood of $\{0\} \times D$ by holomorphic charts isomorphic to $\Delta \times \mathbb{B}$, where $\mathbb{B}$ denotes the unit ball in $\C^{n-1}$, such that $(\{0\} \times D) \cap (\Delta \times \mathbb{B}) = \{0\} \times \mathbb{B}$. It 
is then clear that Proposition \ref{p:ddbar} applies to $\eta \equiv \tilde{\omega} - p^*\omega_D$, where $p$ denotes projection onto the $\mathbb{B}$ factor. This produces a smooth potential $\phi$ for $\tilde{\omega}$ on $\Delta^* \times \mathbb{B}$ such that $\phi$ extends as a $C^{0,2\epsilon}$ H\"older function to the full domain
$\Delta \times \mathbb{B}$ and satisfies $d\phi = O(|z_1|^{2\epsilon-1})$. 

We now apply the (elementary but clever) Varouchas method \cite{V} for smoothing singular K\"ahler forms with continuous local potentials; the presentation in Perutz \cite{P} is particularly convenient. In order to do so, we first need to check that $\phi$ is {strictly} plurisubharmonic in the sense of currents on the whole of $\Delta \times \mathbb{B}$. By definition, we must prove that $\phi' \equiv \phi - \lambda|z|^2$ is \emph{weakly} plurisubharmonic in
the sense of currents for some $\lambda > 0$. Now if $\lambda$ is small enough, then surely $\tilde{\omega}' \equiv \tilde{\omega} - i\partial\bar{\partial}(\lambda|z|^2) \geq 0$ on 
$\Delta^* \times \mathbb{B}$. We then pick any test form $\zeta \in C^\infty_0(\wedge^{n-1,n-1}(\Delta \times \mathbb{B}))$ with $\zeta \geq 0$ and compute
$$\int_{|z_1| > \delta} \phi' dd^c\zeta =\int_{|z_1| > \delta} \tilde{\omega}' \wedge \zeta + \int_{|z_1| = \delta} (\phi'  d^c\zeta - d^c\phi' \wedge \zeta);$$
the first term is nonnegative, and the second
term goes to zero as $\delta \to 0$ because $d\phi' = O(|z_1|^{2\epsilon-1})$.
We are now in a position to apply \cite[Lemma 7.5]{P} to the K\"ahler cocycle $(U_i, \phi_i)_{i \in I}$ thus obtained, where $X = \Delta \times D$, $X_1 = \Delta^* \times D$, and $X_2$ is the union of all our $\Delta \times \mathbb{B}$ coordinate charts. \hfill $\Box$\medskip\

It remains to prove the $i\partial\bar\partial$-lemma with estimates that was crucially used in the above. The result is perhaps most conveniently 
stated by identifying $\Delta^* \times \mathbb{B}$ with the cylinder $\R^+ \times \Sph^1 \times \mathbb{B}$ and using weighted H\"older spaces $C^{k,\alpha}_\epsilon$ on this cylinder. We will write $z_1, \dots, z_n$ for the standard holomorphic coordinates on $\Delta \times \mathbb{B}$, and we will use indices with respect to those.

\begin{prop}\label{p:ddbar}
Fix $\epsilon > 0$ small enough. Let $\eta \in C^\infty_\epsilon$ be a closed real $(1,1)$-form on $\Delta^* \times \mathbb{B}$. Then $\eta = i\partial\bar{\partial}\xi$ for some real-valued function $\xi \in C^{\infty}_\epsilon$. In particular, $\xi = O(|z_1|^\epsilon)$ extends as a $C^{0,\epsilon}$ H\"older function to the full domain $\Delta \times \mathbb{B}$ and $d\xi = O(|z_1|^{\epsilon-1})$.
\end{prop}

\begin{proof}
The proof consists of a reduction to known analytic results on the two factors. We make no pretense of optimality in the analysis. Let us begin by stating the results that we need.
\begin{enumerate}
\item The operators $\partial, \partial\bar{\partial}$ acting on weighted H\"older spaces $C^{k,\alpha}_\epsilon$ on $\Delta^* = \R^+ \times \Sph^1$ admit bounded right inverses $\mathcal{R}_{\partial}^{h}, \mathcal{R}_{\partial\bar{\partial}}^{h}$ (here the $h$ means ``horizontal'') that are compatible with the obvious inclusions of H\"older spaces. See Remark \ref{r:right_inverse} for this.
\item The operators $\bar{\partial}, \partial\bar{\partial}$ acting on smooth functions on $\mathbb{B}$ have right inverses $\mathcal{R}_{\bar{\partial}}^{v}, \mathcal{R}_{\partial\bar{\partial}}^{v}$ defined on the spaces of smooth 
$\bar{\partial}$-closed $(0,1)$-forms and smooth $d$-closed $(1,1)$-forms, respectively,
that extend to bounded operators $C^k \to C^k$.
For $\bar{\partial}$ this is proved in \cite{siu}. For $\partial\bar{\partial}$ let $\mathcal{P}$ denote the usual Poincar{\'e} operator on a star-shaped domain \cite[\S 11.5]{janich}, so that $d\mathcal{P}\eta = \eta$ for all closed forms $\eta$. Then $\mathcal{R}_{\partial\bar{\partial}}^{v}\eta \equiv 2i{\rm Im}\,\mathcal{R}_{\bar{\partial}}^{v}((\mathcal{P}\eta)^{0,1})$ works because $\mathcal{P}$ is clearly bounded $C^k \to C^k$. 

\item Since these right inverses $\mathcal{R}$ are all linear and bounded with respect to $C^{k}$ type norms, they commute with partial differentiation of $C^\infty$ forms with respect to $C^\infty$ parameters.
\end{enumerate}

We now define $\xi \equiv {\rm Re}(\xi^{(1)} + \xi^{(2)} + \xi^{(3)})$, where the $\xi^{(i)}$ are constructed as follows. First, 
$$\xi^{(1)} \equiv \mathcal{R}^h_{\partial\bar{\partial}}(\eta_{1\bar{1}})$$ on each horizontal slice. Next, we construct a vertical $(0,1)$-form $\zeta$ by setting
$$\zeta_{\bar{k}} \equiv \mathcal{R}_{\partial}^h(\eta_{1\bar{k}} - \xi^{(1)}_{,1\bar{k}}) \;\,(k > 1).$$
Then (iii) above and the closedness of $\eta$ imply that $\zeta$ is $\bar{\partial}$-closed on each vertical fibre; hence we can set $\xi^{(2)} \equiv \mathcal{R}^v_{\bar{\partial}}(\zeta)$ fibrewise. Again using (iii) and the closedness of $\eta$, one checks that
$$
\xi^{(2)}_{,1\bar{1}} = 0, \;\, \xi^{(2)}_{,1\bar{k}} = \eta_{1\bar{k}} - \xi^{(1)}_{,1\bar{k}} \;\,(k > 1).
$$
With $\xi^{(3)} \equiv \mathcal{R}^v_{\partial\bar{\partial}}(\eta_{j\bar{k}} - \xi^{(1)}_{,j\bar{k}} - \xi^{(2)}_{,j\bar{k}})$, where again $j,k > 1$, a similar computation shows that $\xi^{(3)}_{,1} = 0$.
The proposition now follows easily from the stated identities.
\end{proof}

\section{Existence and uniqueness}
\label{s:existence}

\subsection{Discussion and overview} The main purpose of this section is to prove Theorem \ref{t:geom-exi}, which generalises and refines the Tian-Yau existence result for complete Ricci-flat K{\"a}hler metrics of linear volume growth \cite[Cor 5.1]{tianyau90}. At the end we quickly explain the proof of Theorem \ref{t:uniqueness}.

We will deduce Theorem \ref{t:geom-exi} from the following analytic existence theorem.

\begin{theorem}[ACyl version of the Calabi conjecture]\label{t:ana-exi}
Let $(M, g, J)$ be an {\rm ACyl} K{\"a}hler manifold of complex dimension $n$ with K{\"a}hler form $\omega$. If $0 < \epsilon \ll 1$ and if $f \in C^{\infty}_\epsilon(M)$ satisfies 
\begin{equation}
\label{eq:f:vanish}
\int_M (e^f-1)\omega^n = 0,
\end{equation}
then
there exists a unique $u \in C^{\infty}_\epsilon(M)$ such that $\omega + i\partial\bar{\partial}u > 0$ and  $(\omega + i\partial\bar{\partial}u)^n = e^f\omega^n$.
\end{theorem}

\begin{remark}
Integration by parts shows that \eqref{eq:f:vanish} is indeed necessary in order for $u$ to exist. This is a nonlinear version of the mean-value-zero assumption of Proposition \ref{p:invert_laplacian}.
As in the linear case,  if $f \in C^{\infty}_\epsilon(M)$ but \eqref{eq:f:vanish} is \emph{not} satisfied, then there may still exist solutions that grow at infinity since the Green's function on $M$ is asymptotically pluriharmonic (in fact, asymptotically linear).
\end{remark}

Theorem \ref{t:ana-exi} could be proved (although this proof is not written down anywhere) 
by combining the proof of \cite[Thm 1.1]{tianyau90} with a new idea concerning asymptotics 
of solutions to complex Monge-Amp{\`e}re equations from \cite{hein}. However, the ingredients from \cite{tianyau90} that would be required for such an approach are in fact very general and technically quite formidable. Here we will instead give an easy direct proof specifically tailored to the ACyl case. We achieve this by using weighted function spaces and by retooling the decay argument from \cite[Prop 2.9(i)]{hein} as an a priori estimate.

Joyce already employed weighted spaces to treat certain examples of \emph{maximal} volume growth---ALE and QALE K{\"a}hler manifolds, see \cite[\S 8.5,\,\S 9.6]{joyce00}---but his weighted nonlinear estimates break
down in our \emph{minimal} 
volume growth situation. This issue is related to an error in the construction of ACyl Calabi-Yau manifolds with exponential asymptotics in \cite{kovalev03}, where the analysis is based \mbox{\cite[p.~132]{kovalev03}}  on an estimate for the maximal volume growth case \cite[p.~52]{tianyau91}. This is incorrect because  the estimate from \cite{tianyau91} crucially relies on a Euclidean type Sobolev inequality that definitely \emph{fails} for any volume growth rate less than the maximal one. See Proposition \ref{p:wsob} below for comparison.

We will prove Theorem \ref{t:ana-exi} in Section \ref{s:proof_ana_exi}, after having deduced 
Theorem \ref{t:geom-exi} from it in Section \ref{s:reduction}. The proof of Theorem \ref{t:uniqueness} is essentially independent of this and will be given in Section \ref{s:uniqueness}. It may be worth advertising that 
our proof of Theorem \ref{t:ana-exi} will be \emph{self-contained} with only two exceptions:
(1)~We use Proposition \ref{p:invert_laplacian} without proof, but no other facts from linear analysis on ACyl manifolds.
(2)~We assume that the reader is familiar with Yau's proof \cite{yau78} of the Calabi conjecture on 
compact K\"ahler manifolds; see B{\l}ocki \cite{blocki} for a detailed and readable exposition.

\subsection{The analytic existence theorem implies the geometric one}\label{s:reduction}
In order to prove Theorem~\ref{t:geom-exi} we need to construct an ACyl Kähler
metric $\tomega$ on $M = \oM \setminus \oD$ such that Theorem \ref{t:ana-exi}
applies to the pair $(M,\tomega)$ and the smooth function $f$ defined by 
\begin{equation}\label{e:define_f}
e^f \tomega^n = i^{n^2}\Omega \wedge \bar{\Omega}.
\end{equation}
Applying Theorem \ref{t:ana-exi}, the 
desired Calabi-Yau metric $\omega$ is then given by $\omega = \tomega + i\partial\bar{\partial}u$.

We will explain the construction of $\tomega$ in two stages. In Part 1, we assume that $\oM$ is smooth and fibred by the linear system $|\oD|$. This is the setting originally considered by Tian-Yau in \cite{tianyau90} though our presentation will be closer in spirit to \cite[\S3.4]{hein}. We discuss this special case separately because it allows for a particularly transparent construction.  In Part 2, we then explain the modifications needed to treat the general case. The orbifold singularities of $\oM$ pose no particular difficulty but the absence of a fibration introduces many unpleasant error terms.

\subsection*{Remark about notation and constants} $A \lesssim B$ means $A \leq CB$ for some large generic constant $C$ (so that $A \sim B$ if and only if $A \lesssim B$ and $B \lesssim A$), and $A \ll B$ means $CA \leq B$. We will eventually encounter parameters $r,s, \dots$ to be fixed only at the very end such that---for instance---$s \ll r \ll 1$; 
it is important to make sure that no generic constant $C$ depends on these parameters.

\subsection*{Part 1: Construction of $\tomega$ if $\oM$ is smooth and fibred by $|\oD|$} Fix any K{\"a}hler form $\omega_0$
in the chosen Kähler class $\mathfrak{k}$ on $\oM$. The first step is to find a
K{\"a}hler form $\tomega_0$ on $\oM$ that is cohomologous to $\omega_0$ when
restricted to~$M$ and Ricci-flat when restricted to $\oD$.

For this, we first of all observe that $K_{\oD}$ is trivial by adjunction. Thus, by the Calabi-Yau theorem,
there exists $u_0 \in C^\infty(\oD)$ such that
$\omega_0|_{\oD} + i\partial\bar{\partial}u_0$ is Ricci-flat. 
Fix a $C^\infty$ trivialisation of the given fibration $|\oD|$ near $\oD$, thus identifying a tubular neighbourhood of $\oD$ with $\Delta \times \oD$, where $\Delta$ denotes the unit disk $\{|w| < 1\}$. Extend $u_0$ to be constant along the $\Delta$ factor and multiply this extension by a cut-off function pulled back from $\Delta$ to further extend $u_0$ to the whole of $\oM$. If the initial tubular 
neighbourhood was small enough, then the restriction of $\omega_0 + i\partial\bar{\partial}u_0$ to any fibre will be positive. All negative components 
of $\omega_0 + i\partial\bar{\partial}u_0$ on the total space $\oM$ can be compensated by
adding the pullback of a sufficiently positive ``bump $2$-form'' on $\Delta$ supported in an annulus containing the cut-off region; such a pullback is automatically closed $(1,1)$ on $\oM$ and exact on $M$. This creates $\tomega_0$.

We now modify $\tomega_0$ to become asymptotically cylindrical with the correct volume form at infinity. Notation: Define $\Delta(r) = \{|w| < r\}$, fix parameters $s \ll r \ll 1$ to be chosen later, and pick a cut-off function $\chi: \Delta \to \R$ with $\chi = 1$ on $\Delta(r-s)$, $\chi = 0$ away from $\Delta(r+s)$, and $s|\chi_w| + s^2|\chi_{w\bar{w}}| \leq C$.
Fix a bump $2$-form $\beta \geq 0$ on $\Delta$ with support contained in $\Delta(r + 2s) \setminus \Delta(r-2s)$ such that $\beta = \frac{i}{2}dw\wedge d\bar{w}$ on $\Delta(r+s)\setminus\Delta(r - s)$, and identify $\beta$ with its pullback to $\oM$ under the given fibration.

The K{\"a}hler potentials of the cylinder metric
$\frac{i}{2}|w|^{-2}dw \wedge d\bar{w}$ are given by
$u(w) =(\log |w|)^2 + h(w)$ with $h$ any harmonic function. We use these
potentials to define closed $(1,1)$-forms on $M$:
$$\tomega_{t} \equiv \tomega_0 + \lambda i\partial\bar{\partial}( \chi u) + t\beta.$$ 
Being compactly supported, the $t\beta$ term does not change the asymptotics of the metric at infinity, but the extra degree of freedom $t > 0$ is needed to deal with the integral condition (\ref{eq:f:vanish}). Also, $\lambda > 0$ is a fixed real number determined by the condition that
\begin{equation}\label{e:define_lambda}
(\tomega_0|_{\oD})^{n-1} = \frac{2}{n\lambda} i^{(n-1)^2}R \wedge \bar{R},
\end{equation}
where $R = {\rm Res}_{\oD}\Omega$ is the holomorphic volume form on $\oD$ specified by $\Omega = \frac{dw}{w} \wedge R + O(1)$ as $w \to 0$.
The forms $\tomega_{t}$ are then positive definite except possibly over $\Delta(r+s)\setminus \Delta(r - s)$. Moreover, if $\tomega_t$ is in fact positive definite globally, then the associated Riemannian metric on $M$ is ACyl and the volume form $\tomega_t^n$ is 
exponentially asymptotic to $i^{n^2} \Omega\wedge\bar{\Omega}$. (To show that $M$ is ACyl, fix a local trivialisation $\Psi: \Delta \times \oD \hookrightarrow \oM$ of the fibration such that $\Psi(0,x) = x$ for all $x \in \oD$ and $d\Psi$ is $\C$-linear along $\oD$; \cf \ref{obs:complex_trivial}. Then we obtain an ACyl map $\Phi$ by substituting $w = e^{-t-i\theta}$ in $\Psi$ as usual.)

We complete the construction by choosing $h(w) = (\log r)^2 - (2 \log r)\log |w|$. This implies that
\begin{equation}\label{e:normpot}
|u| + s|u_w| \leq C\frac{|{\log r}|}{r^2}s^2
\end{equation}
in the gluing region $\Delta(r+s)\setminus\Delta(r-s)$, by Taylor expansion around $|w| = r$.\medskip\

\noindent {\bf Claim.} Given any fixed choice of $r \ll 1$ and $s \ll r$, there exists a unique value of $t > 0$ such that we have $\tomega_{t} > 0$ globally and $\int_M (\tomega_{t}^n - i^{n^2}\Omega \wedge \bar{\Omega}) = 0$.\medskip\

Thus for any choice of $s \ll r \ll 1$ we obtain an ACyl K\"ahler metric $\tomega = \tomega_t$ such that the function $f \in C^\infty_\epsilon(M)$ associated with $\tomega$ by (\ref{e:define_f}) satisfies (\ref{eq:f:vanish}) with respect to $(M, \tomega)$. Then Theorem \ref{t:ana-exi} can be applied. (The resulting Calabi-Yau metric $\omega$ is independent of $r,s$, by Theorem \ref{t:uniqueness}.) \medskip\

\noindent \emph{Proof of the claim.} Using (\ref{e:normpot}), positivity quickly reduces to $t \gg \frac{1}{r^2}|{\log r}|$. The integral condition is equivalent to the following linear equation for $t$: 
\begin{equation}\label{e:linear_volume_form}
\int_M (\tomega_0^n +  n\lambda i\partial\bar{\partial}({\chi}u) \wedge \tomega_0^{n-1} - i^{n^2}\Omega\wedge\bar{\Omega}) + nt \int_M \beta \wedge \tomega_0^{n-1} = 0.
\end{equation}
The $t$-coefficient is positive and $\sim rs$. The constant term can be split as a sum of three contributions: $O(r)$ from $\Delta(r-s)$ since the integrand is $O(|w|^{-1}\tomega_0^n)$ there due to our choice of $\lambda$; $O(|{\log r}|\frac{s}{r})$ from
the gluing region, using (\ref{e:normpot}) again; and a negative part $\sim \log r$ from the rest of $M$. We see that the solution $t \sim \frac{1}{rs}|{\log r}|$ if $s \ll r \ll 1$, which is well within the positivity constraint. \hfill $\Box$

\subsection*{Part 2: Modifications needed to construct $\tomega$ in general} The key simplification in Part 1 was the existence of a holomorphic fibration. This was used in three related ways:
\begin{enumerate}
\item[(1)]
We can write down our ACyl K\"ahler form $\tomega_t$ without first specifying an ACyl map $\Phi$.
\item[(2)]
The pullback of a $2$-form on $\Delta$ is $(1,1)$ upstairs. (This was used twice:~in the initial process of cutting off $u_0$, and then later when working with the bump $2$-form $\beta$.)
\item[(3)]
The volume form of $\tomega_t$ depends linearly on $t$ because the square of a $2$-form on $\Delta$ is zero.
\end{enumerate}
Absent a holomorphic fibration we will need to make the following changes; since we will frequently refer to results from Appendix \ref{app:trivial_nb}, the reader may find it 
helpful to review this appendix first.

\begin{enumerate}
\item[($1'$)]
We begin by constructing $\Phi$ as in \ref{obs:holo_trivial}. In particular this provides a global defining function $w$
for the divisor such that $\bar{\partial}w = O(|w|^2)$. One consequence of this property is that the $\wedge^2T^*D$-components of $i\partial\bar{\partial}(\log |w|)^2$ are indeed negligible at infinity; \cf the end of Appendix \ref{app:trivial_nb}. 
\item[($2'$)]
We only use bump $2$-forms $\beta$ on $\Delta$ that are radially symmetric. Then $\beta = i\partial\bar{\partial}B$ for a unique function $B$ that vanishes identically near $\partial\Delta$; in return, $B$ blows up like $\log |w|$ at the origin. Instead of pulling back $\beta$ under $w$, we pull back $B$ and compute $i\partial\bar{\partial}$ upstairs.
\item[($3'$)] Since the fibres of $w$ are no longer complex, checking positivity and the integral condition now
involves many new terms. These all turn out to be of lower order because $\bar{\partial}w = O(|w|^2)$.
\end{enumerate}
We will now explain the construction of $\tomega$ in more detail, following the basic outline of Part 1 but taking into account these changes as well as the (rather harmless) orbifold singularities of $\oM$.

\subsection*{Step $\mathbf{1'}$}
By assumption, the holomorphic normal bundle to $\oD$ is isomorphic to
$(\C \times D)/\langle \iota \rangle$, where $D$ is smooth and $\iota \in {\rm Aut}(D)$ acts on the product via $\iota(w,x) = (\exp(\frac{2\pi i}{m}) w, \iota(x))$ with $m = {\rm ord}(\iota)$.

Even if $N_{\oD}$ was isomorphic to $(\C \times D)/\langle \iota\rangle $ only as a \emph{smooth complex} orbifold line bundle, there would already exist
a smooth orbifold embedding $\Psi: (\Delta \times D)/\langle \iota\rangle \hookrightarrow \oM$ such that $\Psi(0,x) = x$ for all $x \in \oD = D/\langle \iota \rangle$ and 
$d\Psi$ is $\C$-linear along $\oD$; compare \ref{obs:complex_trivial}. In particular, if $J$ denotes the complex structure on $\oM$ pulled back to $\Delta \times D$, then $J - J_0 = O(|w|)$ and $\bar{\partial}w = O(|w|)$ with respect to $J$. As in \ref{obs:holo_disks} we can assume that the disks $\Delta \times \{x\}$ are $J$-holomorphic.
Now since $N_{\oD}$ is isomorphic to $(\C \times D)/\langle \iota \rangle$ even as a \emph{holomorphic} orbifold line bundle, \ref{obs:holo_trivial} implies that $\bar{\partial}w = O(|w|^2)$ on $\Delta \times D$. We then define our ACyl diffeomorphism $\Phi$ by substituting $w = e^{-t-i\theta}$ in $\Psi$ as usual.

Let us repeat very explicitly that the $T^*\Delta \otimes (T\Delta \oplus TD)$ component of the endomorphism $J - J_0$ vanishes identically, and its $T^*D \otimes T\Delta$ component, $K$, vanishes to second order at the divisor. 

\subsection*{Step $\mathbf{2'}$} In analogy with Part 1 we now construct the following closed $(1,1)$-forms on $\oM$: 
\begin{align}
\label{e:newform1}\tomega_0 &= \omega_0 + i\partial\bar{\partial}(\chi_0 u_0) + t_0 i\partial\bar{\partial}B_0,\\
\label{e:newform2}\tomega_{t} &= \tomega_0 + \lambda i\partial\bar{\partial}( \chi u) + ti\partial\bar{\partial}B.
\end{align}
Here $\omega_0$ is an orbifold K\"ahler form on $\oM$ representing the given K\"ahler class $\mathfrak{k}$, $\omega_0|_{\oD} + i\partial\bar{\partial}u_0$ is the unique Ricci-flat orbifold K\"ahler form representing $\mathfrak{k}|_{\oD}$, $\lambda$ is as in (\ref{e:define_lambda}), $u$ is a cylinder potential on $\Delta^*$ normalised as in (\ref{e:normpot}), and $t_0,t$ will be chosen later. To explain the remaining pieces we pass to the smooth $\Delta \times D$ cover and work $\iota$-invariantly, as follows.

First we extend $u_0$ to be constant along the $\Delta$-factor. Then we choose radial cut-off functions $\chi_0, \chi$ on $\Delta$ with $\nabla \chi_0, \nabla\chi$ supported in $\Delta(2r_0)\setminus \Delta(r_0)$ and $\Delta(r + s)\setminus \Delta(r-s)$, where $s \ll r \ll r_0 \ll 1$. Finally, we choose radial bump forms $\beta_0,\beta$ supported in $\Delta(3r_0)$ and $\Delta(r+2s)\setminus\Delta(r-2s)$ such that $\beta_0 = \frac{i}{2}dw \wedge d\bar{w}$ on $\Delta(2r_0)$ and $\beta = \frac{i}{2}dw \wedge d\bar{w}$ on $\Delta(r+s) \setminus \Delta(r-s)$, and we use the following lemma 
to construct suitable functions $B_0, B$ on $\Delta^*$ such that $i\partial\bar\partial B_0 = \beta_0$ and $i\partial\bar\partial B = \beta$ on $\Delta^*$. 

\begin{lemma}\label{l:easy_potential}
Let $\gamma$ be a radial $2$-form with compact support on $\Delta$.
\begin{enumerate}
\item\label{it:easy_potential_1} There exists a unique radial function $G$ on $\Delta^*$ such that $G \equiv 0$ near $\partial\Delta$ and $i\partial\bar{\partial}G = \gamma$. Also, if ${\rm supp}(\gamma) 
\subset \Delta(\rho)$ for some $\rho < 1$ then ${\rm supp}(G) \subset \Delta(\rho)$ as well.
\item\label{it:easy_potential_2} We have $G(w) = -\frac{1}{\pi}(\int \gamma)\log |w| + \widehat G(w)$, where $\widehat G$ is radial and smooth at $w = 0$.
\item\label{it:easy_potential_3} We have derivative estimates $|\nabla \widehat{G}(w)| \leq \psi(|w|) \frac{1}{|w|}(|w|^2 - \rho_0^2) $ and $|\nabla^2 \widehat{G}| \leq \sqrt{10}\psi(|w|)$, where
 $\psi(\rho) \equiv \max_{|v| \leq \rho} |\gamma(v)|$ and $\rho_0 \equiv \max\{0,\max\{\rho \geq 0: \psi(\rho) = 0\}\}$.
\end{enumerate}
\end{lemma}

Before proving this lemma, let us record its main consequences for Step $3'$. Recall that $K$ denotes the $T^*D \otimes T\Delta$ component of $J - J_0$, introduced at the end of Step $1'$ and discussed in Appendix \ref{app:trivial_nb}, and that we have $K = O(|w|^2)$ because the normal bundle of $D$ is holomorphically trivial.

\begin{corollary}\label{cor:error_bounds}
Let $p: \Delta^* \times D \to \Delta^*$ denote projection onto the first factor.
Keeping the notation of \ref{l:easy_potential}, the form $i\partial\bar{\partial}(G \circ p)$ upstairs has support contained in $\Delta(\rho) \times D$ if $\gamma$ has support contained in $\Delta(\rho)$. Moreover, it can be decomposed
as $i\partial\bar{\partial}(G \circ p) = p^*\gamma -\frac{1}{\pi}(\int \gamma) \eta + \widehat\gamma$, where 
\begin{align}
\label{e:error_bounds_1}\eta = i\partial\bar{\partial}\log |w| = -\frac{1}{2}d({\rm Re}(d\log w) \circ K)  = \begin{cases}0 &{\rm \textit{horizontally}},\\
O(1) & {\rm \textit{mixed\;directions}},\\
O(|w|) & {\rm \textit{vertically}};\end{cases}\\
\label{e:error_bounds_2}\widehat{\gamma} = -\frac{1}{2}d(d\widehat G \circ K) = \begin{cases}0 &{\rm \textit{horizontally}},\\
O(\psi(|w|)|w|^2) & {\rm \textit{mixed\;directions}},\\
O(\psi(|w|)(|w|^2-\rho_0^2)|w|) & {\rm \textit{vertically}}.\end{cases}
\end{align}
The implied constants here are independent of $\gamma$ and in fact only depend on $K$.
\end{corollary}

The stated decomposition of $i\partial\bar\partial(G \circ p)$ follows quickly by observing that $p^*\gamma = i\partial_0\bar{\partial}_0(\widehat{G} \circ p)$
and that $i\partial\bar{\partial}\phi = -\frac{1}{2}d(d\phi \circ J) = i\partial_0\bar\partial_0\phi - \frac{1}{2}d(d\phi \circ K)$ whenever $\phi$ is pulled back from the base disk, $\Delta$.
Similar estimates are discussed informally at the end of Appendix \ref{app:trivial_nb}.

\begin{proof}[Proof of Lemma \ref{l:easy_potential}]
We write $\gamma = g \frac{i}{2}dw \wedge d\bar{w}$, so that $\frac{1}{2}\Delta_{\R^2}G = g$. 
Since the radial component of $\Delta_{\R^2}$ is given by $\frac{1}{\rho}\partial_\rho(\rho\partial_\rho)$, we obtain the following representation for $G$, proving \ref{it:easy_potential_1}:
\begin{equation}\label{e:easy_potential_1}
G(w) = \int_1^{|w|} \frac{2}{\rho}\int_1^\rho g(\sigma) \sigma\, d\sigma \, d\rho.
\end{equation}
Then we decompose the $d\sigma$ integral in \eqref{e:easy_potential_1} as $\int_1^\rho = \int_1^0 + \int_0^\rho $, which proves \ref{it:easy_potential_2}  with
\begin{equation}\label{e:easy_potential_2}
\widehat{G}(w) = \int_1^{|w|}\frac{2}{\rho}\int_0^\rho g(\sigma)\sigma \, d\sigma \, d\rho.
\end{equation}
For \ref{it:easy_potential_3} we first observe that  $|\nabla \widehat{G}| = |\widehat{G}_\rho|$ and $|\nabla^2 \widehat{G}|^2 = \widehat{G}_{\rho\rho}^2 + \frac{1}{\rho^2}\widehat{G}_\rho^2$. Now
\eqref{e:easy_potential_2} yields
\begin{equation}
\widehat{G}_\rho(w) = \frac{2}{|w|}\int_{0}^{|w|} g(\sigma)\sigma \, d\sigma, \;\, \widehat{G}_{\rho\rho}(w) = -\frac{1}{|w|}\widehat{G}_\rho(w)  + 2g(w),
\end{equation}
and hence the claim by applying the triangle inequality.
\end{proof}

\subsection*{Step $\mathbf{3'}$} If $\tomega_t$ is positive definite, then the associated Riemannian metric will indeed be ACyl with respect to the diffeomorphism $\Phi$ from Step $1'$ since $\bar{\partial}w = O(|w|^2)$; see again the end of Appendix \ref{app:trivial_nb}. Hence all that remains to be done is to prove the counterpart of the Claim in Part 1.

First we show that $\tomega_0$ of (\ref{e:newform1}) is positive for $r_0 \ll 1$ and $t_0 \sim r_0^{-2}$. The first issue is that the good term $i\partial\bar{\partial}B_0$ no longer has only horizontal components. However, Corollary \ref{cor:error_bounds} with $\gamma = \beta_0$ shows that the mixed and vertical components of $i\partial\bar{\partial}B_0$ are controlled by $(\int\beta_0)\eta$ and $\widehat{\beta}_0$; more precisely, the
mixed parts are $O(r_0^2)$ and the vertical parts are $O(r_0^2|w|)$. Thus, $\omega_0 + t_0i\partial\bar{\partial}B_0$ is bounded below by a smooth K\"ahler form on $\oM$ if $r_0 \ll 1$ and $t_0 = o(r_0^{-3})$, and has a positive horizontal component $\sim t_0$ on $\Delta(2r_0) \times D$ if $t_0 \gg 1$. We must now prove that choosing $t_0 \sim r_0^{-2}$ compensates all negative components of
$i\partial\bar{\partial}(\chi_0 u_0)$ over the annulus $(\Delta(2r_0)\setminus \Delta(r_0))\times D$. This is clear horizontally, and the mixed or vertical components are negligible.
\Eg the worst term, $u_0 i\partial\bar{\partial}\chi_0$, contributes $u_0 d(d\chi_0 \circ K)$ to these errors; the mixed components of this are 
$O(1)$ and the vertical ones are $O(r_0)$.

Positivity of $\tomega_t$ in (\ref{e:newform2}) is similar. First, Corollary \ref{cor:error_bounds} applied with $\gamma = \beta$ tells us that $i\partial\bar{\partial}B$ has
$O(rs + \chi_{\rm ann}r^2)$ mixed and $O(rs|w|)$ vertical components; here $\chi_{\rm ann}$ is the smooth function defined by $\beta = \chi_{\rm ann} \frac{i}{2}dw \wedge d\bar{w}$, which is essentially equal to the indicator function of the gluing annulus. On the other hand, 
the horizontal component of $i\partial\bar\partial B$ is always nonnegative and $\sim 1$ over the annulus.
Thus, $\tomega_0 + \lambda \chi i \partial\bar{\partial}u + ti\partial\bar{\partial}B$ is again bounded below by some smooth K\"ahler form on $\oM$ as long as 
$t = o(\frac{1}{r^2s})$, and has a horizontal component $\sim t$ in the gluing region if $t \gg 1$. 
Now we need to add on the error terms involving derivatives of $\chi$, and we claim that---exactly as in the fibred case---taking $t \gg \frac{1}{r^2}|{\log r}|$ restores positivity.
This is obvious
horizontally, and the mixed or vertical components are again negligible. \Eg the worst term $ui\partial\bar{\partial}\chi$ contributes $ud(d\chi \circ K)$,
which has $O(|{\log r}|)$ mixed and $O(s|{\log r}|)$ vertical pieces; Cauchy-Schwarz allows us to bound the mixed ones from below by a 
horizontal term which is $O(\frac{1}{r}|{\log r}|) = o(t)$ and a vertical term which is $O(r|{\log r}|)$.

It remains to see that the integral condition is still satisfied for some $t \sim \frac{1}{rs}|{\log r}|$. This condition is now a degree $n$ equation in $t$ whose constant and linear coefficients are small perturbations of the ones in (\ref{e:linear_volume_form}), and whose $t^2, \dots, t^n$ coefficients are small. More precisely, we want to solve
\begin{equation}\label{e:nonlinear_volume_form}
(c_0 + \sum_{p=2}^{n} \epsilon_{0,p}) + (c_1 + \sum_{p=1}^{n-1} \epsilon_{1,p})t + \sum_{\ell = 2}^n (\sum_{p=0}^{n-\ell} \epsilon_{\ell,p})t^\ell = 0,
\end{equation}
where $c_0$ and $c_1$ are defined exactly like the constant and linear terms in (\ref{e:linear_volume_form}), and 
\begin{equation}\label{define_epsilons}
\epsilon_{\ell,p} \sim \int_M (i\partial\bar{\partial}B)^\ell \wedge (i\partial\bar{\partial}(\chi u))^p \wedge \tilde{\omega}_0^{n-\ell-p}\;\,{\rm for}\;\,\ell + p \in \{2,\dots,n\}.
\end{equation}
These integrals are small because they involve wedge products of almost horizontal $2$-forms.

The main tool needed to carry out the actual estimates is the following table:
\begin{equation}\label{aux_epsilons}
\begin{split}
i\partial\bar{\partial}B = \begin{cases}\chi_{\rm ann} &{\rm horizontally},\\
O(rs + \chi_{\rm ann}r^2) & {\rm mixed\;directions},\\
O(rs|w|) & {\rm vertically},\end{cases} \\
 i\partial\bar{\partial}(\chi u) = \begin{cases} \frac{1}{|w|^2} + O(\chi_{\rm ann}\frac{|{\log r}|}{r^2}) &{\rm horizontally},\\
O(|{\log |w|}|) & {\rm mixed\;directions},\\
O(|w||{\log |w|}|) & {\rm vertically},\end{cases}
\end{split}
\end{equation}
on $\Delta(r + 2s) \times D$. Here $\chi_{\rm ann}$ is again defined by $\beta = \chi_{\rm ann} \frac{i}{2}dw \wedge d\bar{w}$ and the bounds for $i\partial\bar\partial B$ follow from Corollary \ref{cor:error_bounds}, whereas the ones for $i\partial\bar{\partial}(\chi u)$ follow from a direct computation (compare again the end of Appendix \ref{app:trivial_nb}). Given this information and the fact that $({\rm\textit{horizontal}}\,)^{\wedge a} \wedge ({\rm\textit{mixed}}\,)^{\wedge b} = 0$ if
$a \geq 2$ or $a = 1, b \geq 1$ or  $b \geq 3$, a lengthy computation (see Appendix \ref{s:nonfibred_errors}) yields that
\begin{equation}\label{result_epsilons}
c_0 + \sum_{p=2}^{n} \epsilon_{0,p}\sim -|{\log r}|, \;\, c_1 + \sum_{p=1}^{n-1} \epsilon_{1,p} \sim rs, \;\, \sum_{p = 0}^{n-\ell} \epsilon_{\ell, p} = O((r^2s)^{\ell-1}(rs + r^3)) \;\,{\rm for} \;\,\ell \geq 2.
\end{equation}
Estimating the $\epsilon_{\ell,p}$ with $\ell \geq 2$ is the most difficult step; the main contribution arises by integrating $({\rm \textit{vertical}}\,)^{\ell-1}({\rm \textit{horizontal}}\,)$ and $({\rm \textit{vertical}}\,)^{\ell-2}({\rm \textit{mixed}}\,)^2$ type terms over the annulus for $p = 0$.

We now concentrate on the interval $t \sim \frac{1}{rs}|{\log r}|$, which contains the unique zero of the linear part of (\ref{e:nonlinear_volume_form}). At the two boundary points, the linear part of  (\ref{e:nonlinear_volume_form}) is comparable to $\pm |{\log r}|$, while the nonlinear terms of  (\ref{e:nonlinear_volume_form}) are at worst $O(r|{\log r}|^2(1 + \frac{1}{s}r^2))$ on the whole interval.
Thus it suffices to choose $1 \gg r_0 \gg r \gg s \gtrsim r^2$ (unlike in Part 1, we are not free to make $s$ arbitrarily small).

\subsection{Proof of the analytic existence theorem}\label{s:proof_ana_exi} The proof of Theorem \ref{t:ana-exi} requires a nontrivial technical preliminary:~the proof of a global Sobolev inequality on $M$. Such inequalities are sensitive to the volume growth at infinity, and need to take rather different shapes depending on whether the growth rate is slower
or faster than quadratic. Our proof follows the strategy expounded in \cite{gsc}; see also \cite{hein1, min} for closely related results and applications.

\begin{prop}\label{p:wsob}
Let $(M^n, g)$ be an {\rm ACyl} manifold as in Definition \ref{d:ACyl}.
Then for all $\mu > 0$ there exists a piecewise constant positive function
$\psi_\mu = O(e^{-2\mu t})$ with $\int_M \psi_\mu\,d{\rm vol} = 1$ such that
\begin{equation}\label{e:wsob}
\|e^{-\mu t}(u-\bar{u}_\mu)\|_{2\sigma} \leq C_{M,\mu,\sigma} \|\nabla u\|_2
\end{equation} 
holds for all $\sigma \in [1,\frac{n}{n-2}]$ and all $u \in C^\infty_0(M)$, where $\bar{u}_\mu \equiv \int_M u\psi_\mu\,d{\rm vol}$.
\end{prop}

The subtraction of an average on the left-hand side of (\ref{e:wsob}) is inevitable because $M$ has less than quadratic volume growth. In \cite{tianyau90}, the relation (\ref{eq:f:vanish}) is directly applied to compensate this.

\begin{proof}[Proof of Proposition \ref{p:wsob}]
We have $M = \bigcup {\rm clos}(A_i)$, where $A_0 = U$ and $A_i = (i-1, i) \times X$ for $i \in \N$, and we begin by discretising the left-hand side of (\ref{e:wsob}) accordingly:
\begin{equation}\label{e:aux1}
\|e^{-\mu t}(u-\bar{u}_\mu)\|_{2\sigma}^2 \leq C\sum \|\chi_i(u - \bar{u}_i)\|_{2\sigma}^2 + C \sum e^{-2\mu i}|\bar{u}_i - \bar{u}_\mu|^2,
\end{equation}
where $\chi_i$ is the characteristic function of $A_i$ and $\bar{u}_i$ is the
average of $u$ over $A_i$. Since the $A_i$ have uniformly bounded geometry,
$\|\chi_i(u - \bar{u}_i)\|_{2\sigma} \leq C\|\chi_i\nabla u\|_{2}$ by the usual
Sobolev inequality. Thus, it suffices to estimate the second sum in
(\ref{e:aux1}). This involves defining the weight function $\psi_\mu$. In order for our argument to go through, we require that $\sum e^{-2\mu i}(\bar{u}_i - \bar{u}_\mu) = 0$ for all test functions $u$, and so we define $\psi_\mu \equiv \phi_\mu/\int_M \phi_\mu\,d{\rm vol}$, where $\phi_\mu$ is constant equal to $e^{-2\mu i}/|A_i|$ on $A_i$. Then
\begin{align*}
\sum e^{-2\mu i}|\bar{u}_i - \bar{u}_\mu|^2 \leq C \sum_{i < j} e^{-2\mu(i + j)} |\bar{u}_i - \bar{u}_j|^2 
\leq C \sum_{i < j}  e^{-2\mu(i + j)} |i - j| \sum_{k = i}^{j-1} |\bar{u}_k - \bar{u}_{k+1}|^2. \end{align*}
Next, we define $B_k \equiv {\rm int}({\rm clos}(A_k \cup A_{k+1}))$ and observe that
\begin{align*}
|\bar{u}_k - \bar{u}_{k+1}|^2 \leq \frac{1}{|A_k||A_{k+1}|}\int_{A_k \times A_{k+1}} |u(x) - u(y)|^2\,dx\,dy
\leq \frac{2|B_k|}{|A_k||A_{k+1}|}\int_{B_k} |u - \bar{u}_{B_k}|^2,\end{align*}
where $\bar{u}_{B_k}$ denotes the average of $u$ over $B_k$. Since $B_k$ is connected, we can now apply the standard Poincar{\'e} inequality on $B_k$, which completes the proof.
\end{proof}

\begin{proof}[Proof of Theorem \ref{t:ana-exi}]
The uniqueness claim is proved independently in Section \ref{s:uniqueness} and really only requires that $u \in C^2_\epsilon(M)$. Thus, it suffices to prove the existence of a solution $u \in C^{k+2,\alpha}_\epsilon(M)$ for any given $k \in \N_0$ and $\alpha \in (0,1)$. For this we take 
$\epsilon \in (0,\delta]$ to be smaller than the square root of the first eigenvalue of the Laplacian on the cross-section $X$, and set up a continuity method. Let
\begin{align*}
\mathcal{X} = \{u \in C^{k+2,{\alpha}}_\epsilon(M): \omega_u = \omega + i\partial\bar{\partial}u > 0\}, \;\, \mathcal{Y} = \{f \in C^{k,{\alpha}}_\epsilon(M): \int_M (e^f - 1)\omega^n = 0\}.
\end{align*}
Then $\mathcal{X}$ is an open set, $\mathcal{Y}$ is a hypersurface, and the complex Monge-Amp{\`e}re operator $\mathcal{F}$ given by $(\omega + i\partial\bar{\partial}u)^n = e^{\mathcal{F}(u)}\omega^n$ induces a map $\mathcal{F}: \mathcal{X} \to \mathcal{Y}$. For $u \in \mathcal{X}$, the metric $g_u$ associated with $\omega_u$ is again asymptotically cylindrical (though only of regularity $C^{k,{\alpha}}_\epsilon$) with respect to $\Phi$ and $X$.

Given $f$ as in the statement of the theorem, we wish to solve the family of equations $\mathcal{F}(u_\tau) = f_\tau$ for $u_\tau \in \mathcal{X}$, with $f_\tau \equiv 
\log(1 + \tau(e^f-1)) \in \mathcal{Y}$ for $\tau \in [0,1]$. We have a trivial solution $u_0 = 0$. Next, we need to show that the set of all $\tau$ for which a solution $u_\tau \in \mathcal{X}$ exists is open. For $u \in \mathcal{X}$, 
%\log(1 + \tau(e^f-1)) \in \mathcal{Y}$ for $\tau \in [0,1]$. We have a trivial solution $u_0 = 0$. Next, we need to show that the set of all $\tau$ for which a solution $u_\tau \in \mathcal{X}$ exists is open. For $u \in \mathcal{X}$, 
$$
T_u\mathcal{F} = \frac{1}{2}\Delta_{g_u}: T_u \mathcal{X} = C^{k+2,{\alpha}}_\epsilon(M) \to T_{\mathcal{F}(u)}\mathcal{Y} = C^{k,{\alpha}}_\epsilon(M)_{0,g_u},
$$
the subscripts $0, g_u$ indicating mean value zero with respect to $g_u$, and we must show that this is an isomorphism if $u = u_\tau$. But if $u = u_\tau$, then $\mathcal{F}(u_\tau) = f_\tau$, which implies $u_\tau \in C^\infty_\epsilon(M)$ by a standard bootstrapping argument, and so $g_u$ is smooth enough to apply Proposition \ref{p:invert_laplacian} as written.

It remains to prove a \emph{quantitative} a priori bound on the $C^{k+2,\alpha}_\epsilon$-norm of $u_\tau$, using the \emph{qualitative} information that $u_\tau \in C^\infty_\epsilon(M)$. We proceed in a sequence of four partial a priori estimates. We will abbreviate $u = u_\tau$ and $f = f_\tau$, but all constants are understood to be independent of $\tau$. 

\subsection*{Step 1: $C^0$ from Moser iteration} We apply Moser iteration as in \cite[\S 3.1]{hein1} or \cite[Lemma 3.5]{tianyau90} to derive an a priori bound on the sup norm of $u$. First let us recall the basic underlying computation. To this end, fix $T > 0$ and define an auxiliary form $\eta \equiv \sum_{k = 0}^{n-1} \omega^k \wedge \omega_u^{n-1-k}$. Then we have
\begin{equation}\label{e:ibp}
\int_{t < T} |\nabla|u|^{\frac{p}{2}}|^2 \omega^n \leq -\frac{np^2}{2(p-1)}\left[\int_{t < T} u|u|^{p-2}(e^f-1)\omega^n - \frac{1}{2}\int_{t = T} u|u|^{p-2} d^c u \wedge \eta\right]
\end{equation}
for all $p > 1$. See \cite[p.~212]{blocki} for this, although in \cite{blocki} there are of course no boundary terms. Notice that (\ref{e:ibp}) still holds with $u$ replaced by $u - \lambda$ for any constant $\lambda \in \R$, 
and also that the boundary term goes to zero as $T \to \infty$ (no matter what $\lambda$ we subtract) because $d^c(u - \lambda) = O(e^{-\epsilon t})$.

We begin the iteration process by setting $p = 2$ and $\lambda = \bar{u}_{\mu}$ (as in Proposition \ref{p:wsob}), with $\mu$ to be determined as we go along. If $\mu < \epsilon$, then (\ref{e:wsob}) and (\ref{e:ibp}) imply that
$$\|e^{-\mu t}(u - \bar{u}_\mu)\|_{2\sigma}^2 \leq C\|\nabla u\|_2^2 \leq C\|e^{-\epsilon t}(u - \bar{u}_\mu)\|_1 \leq C\|e^{-\mu t}(u - \bar{u}_\mu)\|_{2\sigma}.$$
To continue the iteration, we will prove that, for all $\sigma \in (1,2)$ with $2\mu\sigma < \epsilon$ and for all $k \in \N_0$,
\begin{equation}\label{e:ind}
\left\|e^{-\mu t}|u - \bar{u}_\mu|^{\sigma^{k+1}}\right\|_{2\sigma}^2 \leq C\sigma^k\max\left\{1,\left\|e^{-\mu t}|u-\bar{u}_\mu|^{\sigma^k}\right\|_{2\sigma}^{2\sigma} \right\}
\end{equation}
Given this, a standard argument \cite[p.~212]{blocki} then shows that the $L^{2\sigma^k}$-norm of $u - \bar{u}_\mu$ with respect to the measure $e^{-2\mu\sigma t}d{\rm vol}$ is bounded uniformly in $k$, so that $\|u-\bar{u}_\mu\|_\infty \leq C$. Since $u = O(e^{-\epsilon t})$, we deduce that $|\bar{u}_\mu| \leq C$, hence $\|u\|_\infty \leq C$ as desired.

In order to prove (\ref{e:ind}), we first apply (\ref{e:ibp}) with $p = 2\sigma^{k+1}$ and with $u$ replaced by $u - \bar{u}_\mu$, and then (\ref{e:wsob}). Abbreviating $u_k \equiv |u - \bar{u}_\mu|^{\sigma^k}$, this yields the following inequalities:
$$
\|e^{-\mu t}(u_{k+1} - \overline{u_{k+1}}{\hskip0.1mm}_{,\hskip0.2mm\mu})\|_{2\sigma}^2 \leq C\|\nabla u_{k+1}\|_2^2 \leq C\sigma^k \|e^{-\epsilon t}|u - \bar{u}_\mu|^{2\sigma^k-1}\|_1.
$$
Proceeding on the right-hand side, H{\"o}lder's inequality tells us that
$$
\|e^{-\epsilon t}|u-\bar{u}_\mu|^{2\sigma^k-1}\|_1 \leq C\|e^{(2\mu\sigma-\epsilon)t}\|_{2\sigma^{k+1}}\max\{1,\|e^{-\mu t}u_k\|_{2\sigma}^{2\sigma}\},
$$
and if $2\mu\sigma < \epsilon$ then the prefactor converges to $1$ as $k \to \infty$. On the other hand, 
$$\|e^{-\mu t} \overline{u_{k+1}}{\hskip0.1mm}_{,\hskip0.2mm\mu}\|_{2\sigma}^2 = \|e^{-\mu t}\|_{2\sigma}^2 \|\psi_\mu u_{k+1}\|_1^2 \leq C\|e^{(\sigma - 2)\mu  t}\|_2^2 \|e^{-\mu t}u_k\|_{2\sigma}^{2\sigma},$$
which is finite if $\sigma < 2$, and of the required form. All in all, this proves (\ref{e:ind}).

\subsection*{Step 2: $C^0$ implies $C^{\infty}$} We do not need to say very much here. Given that functions in the space $\mathcal{X}$ attain their extrema on $M$ and that $M$ has uniformly bounded geometry at infinity, the classical arguments proving Step 2 in the compact case \cite[\S 5.5, \S5.6]{blocki} go through verbatim.

\subsection*{Step 3: $C^\infty$ implies $C^\infty_{\epsilon'}$ for some uniform $\epsilon' \in (0,\epsilon]$} This is a special case of an energy decay argument from \cite[Prop 2.9(i)]{hein}, which we use as an a priori estimate here. We begin by writing out the counterpart of the $p = 2$ case of (\ref{e:ibp}) for the outer domain $\{t > T\}$: 
\begin{equation}\label{e:ibp2}
\int_{t > T} |\nabla u|^2 \omega^n \leq -2n\left[\int_{t > T} u(e^f-1)\omega^n + \frac{1}{2}\int_{t = T} u\,d^c u \wedge \eta\right].
\end{equation}
This is proved by repeating the standard computation on $\{T < t < T'\}$ and sending $T' \to \infty$. Also, (\ref{e:ibp2}) again holds with $u$ replaced by $u - \lambda$ for any constant $\lambda \in \R$; we take $\lambda$ to be the average of $u$ over $\{t = T\}$. 
Defining $Q_T$ to be the quantity on the left-hand side of (\ref{e:ibp2}), this yields
$$
Q_T \leq Ce^{-\epsilon T} + C \int_{t = T} |u - \lambda| |\nabla u| \leq C e^{-\epsilon T} + C \int_{t = T} |\nabla u|^2 \leq Ce^{-\epsilon T} - C \frac{dQ_T}{dT},
$$
where we have used our $C^2$ a priori estimate from Steps 1 and 2, Cauchy-Schwarz, and the Poincar{\'e} inequality. It is elementary to deduce from this that $Q_T \leq C e^{-\epsilon' T}$ for some uniform $\epsilon' \in (0,\epsilon]$.

Now define $A_T \equiv \{T < t < T+1\}$ and let $u_T$ denote the average of $u$ on $A_T$. Then our estimate for $Q_T$ and the Poincar{\'e} inequality imply that $\|u - u_T\|_{L^2(A_T)} \leq Ce^{-\epsilon'T}$. On the other hand, simply by rewriting the Monge-Amp{\`e}re equation, we have 
\[
\mathcal{L}(u - u_T) = e^f - 1 = O(e^{-\epsilon T}) \ \text{on\ }  A_T, 
\]
where the linear operator $\mathcal{L}$ is defined by
\begin{equation}\label{e:frozen_cma}
(\mathcal{L}v)\omega^n = i\partial\bar{\partial}v \wedge (\omega^{n-1} + \omega^{n-2} \wedge \omega_u + \cdots + \omega_u^{n-1})
\end{equation}
as in \cite[p.~137]{kovalev03}. Since $\mathcal{L}$ is uniformly elliptic with respect to $g$ by Step 2, Moser iteration now tells us that $|u - u_T| \leq Ce^{-\epsilon'T}$ on a slightly smaller domain; see \cite[Thm 4.1]{hanlin} for this type of estimate. 
Then Schauder theory gives $|\nabla^k u| \leq C_k e^{-\epsilon' t}$ for all $k > 0$. Thus, eventually, $|u| \leq C e^{-\epsilon' t}$ for some uniform constant $C$, by integrating the exponentially decaying bound on $\nabla u$ along rays.

\subsection*{Step 4: $C^\infty_{\epsilon'}$ implies $C^{\infty}_\epsilon$} We are assuming that $u \in C^\infty_\epsilon(M)$ with ineffective bounds,
and Step 3 yields $u \in C^\infty_{\epsilon'}(M)$ with effective bounds
for some uniform $\epsilon' \in (0,\epsilon]$.
To upgrade from $\epsilon'$ to $\epsilon$ in the effective bounds,
we first rewrite the complex Monge-Amp{\`e}re equation as
\begin{equation}\label{e:bs}
\frac{1}{2}\Delta_g u = (e^f - 1) - \mathcal{Q}(u), \;\, \mathcal{Q}(u)\omega^n = {n \choose 2}(i\partial\bar{\partial}u)^2 \wedge \omega^{n-2} + \cdots + (i\partial\bar{\partial}u)^n.  
\end{equation}
If $u \in C^\infty_\delta(M)$ with $\delta \in (0,\epsilon]$, then the right-hand side of the PDE in (\ref{e:bs}) lies in $C^\infty_{\delta'}(M)_{0,g}$, $\delta' = \min\{2\delta,\epsilon\}$, so that Proposition \ref{p:invert_laplacian} yields $u \in C^\infty_{\delta'}(M)$, effective estimates understood throughout. We then  put $\delta = \epsilon'$ and iterate a bounded number of times to obtain the desired conclusion. \end{proof}

\begin{remark}
Let us quickly review how we used that $\int_M (e^f - 1)\omega^n = 0$. Unlike in \cite[Lemma 3.4]{tianyau90}, this played no direct role in the nonlinear estimates. However, we needed to drop boundary terms at infinity in (\ref{e:ibp}) and (\ref{e:ibp2}). This was possible because we were working in a space of functions with exponential decay, which the linear analysis allowed us to do because $\int_M (e^f-1)\omega^n = 0$. 
\end{remark}

\subsection{Uniqueness}\label{s:uniqueness}
Finally, let us explain why the Ricci-flat ACyl metric produced by 
Theorem \ref{t:geom-exi} is unique among metrics that are ACyl with respect to the same diffeomorphism 
$\Phi$. This follows from Hodge theory arguments as in Section \ref{s:acyl_analysis}.

\begin{proof}[Proof of Theorem \ref{t:uniqueness}]
First we deduce an \acyl $i\partial\db$-lemma, showing that the exact
decaying $(1,1)$-form $\omega = \omega_2 - \omega_1$ can be written as
$i\partial\db u$ for some function $u$ of linear growth.

Since $\omega$ is exact and decaying, it can according to
\cite[Thm 2.3.27]{jnthesis} be written as $\omega = d\alpha$, where 
$\alpha$ is asymptotic to a translation-invariant harmonic 1-form on $M_\infty$. In
particular, $\bar{\partial}^*\alpha^{0,1}$ is a decaying function and can
therefore be written as $\bar{\partial}^*\bar{\partial}\gamma$ for a function $\gamma$ of linear
growth. The form $\db \gamma - \alpha^{0,1}$ is bounded harmonic, hence closed.
Thus, if we set $u = 2 \im\gamma$, then $i\partial\db u = \partial\alpha^{0,1} + \bar{\partial}\alpha^{1,0} = \omega$.

Now $\omega_1^n = \omega_2^n$ implies that $\mathcal{L}u = 0$, where
$\mathcal{L}v = i\partial\db v \wedge \eta$ with
\[ \eta = \omega_1^{n-1} + \omega_1^{n-2} \wedge \omega_2
+ \cdots + \omega_2^{n-1} \]
as in (\ref{e:frozen_cma}).
The $(n{-}1, n{-}1)$-form $\eta$ is positive in the sense that
$\eta \wedge i\alpha \wedge \bar\alpha > 0$ for every nonzero $(1,0)$-form $\alpha$. It follows that there is a Hermitian metric $\omega$ such that
$\omega^{n-1} = \eta$. This is not typically Kähler, but the
``balanced'' condition that $d\omega^{n-1} = 0$ implies that $\mathcal{L}$ is exactly the
Laplacian with respect to the Riemannian metric associated with $\omega$.
Since any subexponentially growing harmonic function $h$ defines a direction in the cokernel of the Laplacian on exponentially decaying functions (because $\int (\Delta v)h = 0$ if $v$ is decaying), and since this cokernel is $1$-dimensional by
Proposition \ref{p:invert_laplacian}, the only
subexponential harmonic functions are the constants. Hence $u$ is a constant. 
\end{proof} 

\appendix

\section{Divisors with trivial normal bundle}\label{app:trivial_nb}

Let $D$ be a smooth compact divisor in some complex manifold and $U$ a tubular neighbourhood of $D$ that we are free to shrink as needed. We wish to discuss various ``product-like'' conditions for $U$. Let $N$ denote the normal bundle to $D$ in $U$, $\Delta$ the unit disk in $\C$ with standard coordinate $w$, $J$ the complex structure on $U$, and $J_0$ the product complex structure on $\Delta \times D$.

\begin{obs}\label{obs:complex_trivial}
$N$ is trivial as a complex line bundle if and only if there exists a diffeomorphism $\Psi: \Delta \times D \to U$ with $\Psi(0,x) = x$ for all $x \in D$ such that $\Psi^*J - J_0 = 0$ along $\{0\} \times D$. In particular, viewing $w$ as a defining function for $D$ in $U$, we have that $\bar{\partial}w = O(|w|)$.
\end{obs}

Indeed, given $\Psi$, the restriction of $\Psi_*\partial_w$ to $D$ defines a section of $T^{1,0}U|_D$ complementing $T^{1,0}D$, hence a trivialisation of $N$ as a smooth complex line bundle. There is significant freedom in choosing such diffeomorphisms $\Psi$, and the next observation provides a very useful normalisation.

\begin{obs}\label{obs:holo_disks}
In {\ref{obs:complex_trivial}} we can arrange that $\Psi^*J - J_0 = 0$ on the horizontal subbundle $T\Delta$ of the tangent bundle $T(\Delta \times D)$ without changing the vector field $\Psi_*\partial_w|_D$.
\end{obs}

In particular, the disks $\Psi(\Delta \times \{x\})$ will be holomorphic. This is proved as in Section \ref{s:hol_cptf}, Step 1. With a more careful choice of a right inverse to the $\bar\partial$-operator, one could in fact not only prescribe the tangent vectors of these holomorphic disks at $w = 0$ but their full Taylor expansions.

We require the following application of \ref{obs:holo_disks} in Section \ref{s:reduction}, Part 2.

\begin{obs}\label{obs:holo_trivial}
$N$ is trivial as a holomorphic line bundle if and only if there exists $\Psi$ as in \ref{obs:holo_disks} such that the $T^*D \otimes T\Delta$ component of $\Psi^*J - J_0$ is $O(|w|^{2})$. In particular, denoting this component by $K$, we have that $\bar{\partial}w = \frac{i}{2} dw \circ \Psi_*K = O(|w|^{2})$.
\end{obs}

\begin{proof}
As in Section \ref{s:hol_cptf},  Step 3, it suffices to show that if we have $\Psi$ as in \ref{obs:holo_disks}, then $\Psi_*\partial_w|_D$ induces a \emph{holomorphic} trivialising section of $N$ if and only if $\bar{\partial}w = O(|w|^2)$. Now the former is equivalent to $\frac{\partial z}{\partial w}$ being holomorphic on $D$ for every local holomorphic defining function $z$ of $D$. 
Restricting $z$ to the holomorphic disks $\Psi(\Delta \times \{x\})$ we obtain a power series expansion  $z= \sum_{j=1}^\infty z_j w^j$, where the $z_j$ are
smooth locally defined functions on $D$ and $z_1$ never vanishes. Applying $\bar{\partial}$ to this identity
quickly shows 
that $\bar{\partial}w = O(|w|^2)$ if and only if $z_1$ is holomorphic on $D$, as desired.
\end{proof}

Let $\mathcal{J}_D$ denote the ideal sheaf of $D$ in $\mathcal{O}_U$. Given $m \in \N$, the $(m-1)$st infinitesimal neighbourhood $mD$ of $D$ in $U$ is defined as the analytic space $(D, \mathcal{O}_U/\mathcal{J}_D^m)$. The following partial extension of \ref{obs:holo_trivial} to higher orders may be useful to keep in mind in Section \ref{s:cptfy_prelim}.

\begin{obs}\label{obs:holo_trivial_m}
If $\mathcal{O}_{mD}(D)$ is trivial as a holomorphic line bundle, then there exists a smooth
defining function $w: U \to \Delta$ for $D$ such that $\bar{\partial}w = O(|w|^{m+1})$.
\end{obs}

\begin{proof}
The exact sequence $0 \to \mathcal{J}_D^{m-1} \to \mathcal{O}_U(D) \to \mathcal{O}_{mD}(D) \to 0$ tells us that $\mathcal{O}_{mD}(D)$ is trivial if and only if there exists a finite cover of $U$ by open sets $U_j$ together with meromorphic functions $z_j$ such that ${\rm div}(z_j) = -(D \cap U_j) $ and $z_j - z_k \in \mathcal{J}_D^{m-1}(U_j \cap U_k)$ for all $j,k$.
Fix a partition of unity $\chi_j$ subordinate to this open cover and define $w \equiv \sum \chi_j w_j$, where each $w_j \equiv \frac{1}{z_j}$ is a local holomorphic defining function for $D$ in $U_j$. We need to check that $w$ does not vanish in $U$ except on $D$, and that $\bar\partial w = O(|w|^{m+1})$; both properties follow easily from the fact that $w_j - w_k \in \mathcal{J}_D^{m+1}(U_j \cap U_k)$.
\end{proof}

The limiting case of \ref{obs:holo_trivial_m} as $m \to \infty$ amounts to

\begin{obs}\label{obs:fibred}
$\mathcal{O}_U(D)$ is holomorphically trivial if and only if there is a holomorphic defining function $w: U \to \Delta$ for $D$. This is the case if and only if $U$ is fibred by the linear system $|D|$. 
\end{obs}

\begin{remark}
By standard results in deformation theory, the linear system $|D|$ will certainly define a fibration of $U$ whenever
$N = \mathcal{O}_D(D)$ is holomorphically trivial and $h^{0,1}(D) = 0$.
\end{remark}

\begin{remark}
One sometimes encounters a slightly weaker flatness condition than \ref{obs:fibred}: that the real hypersurface $\partial U$ is \emph{Levi-flat}, \ie foliated by complex hypersurfaces of the ambient space.
\end{remark}

To conclude this appendix, we wish to explain on an intuitive level why the existence of an
ACyl Hermitian metric on $U \setminus D$ is equivalent to $N = \mathcal{O}_D(D)$ being trivial as a holomorphic line bundle. More precise results along these lines are proved in Sections \ref{s:hol_cptf} and \ref{s:reduction}.

$\bullet$ Suppose we are given an ACyl Hermitian metric on $U \setminus D$. We assume that the cylindrical end is $\R^+ \times \Sph^1 \times D$ with an ACyl diffeomorphism of the form $(t,\theta,x) \mapsto \Psi(e^{-t-i\theta},x)$ with $\Psi$ as in \ref{obs:holo_disks}. Using the ACyl metric, we can see that the purely vertical ($\wedge^2T^*D$) components of  $i\partial\bar{\partial} \log |w|$ must 
vanish as $w \to 0$. On the other hand, writing $K$ as in \ref{obs:holo_trivial}, we have $i\partial\bar{\partial}\log |w| = -\frac{1}{2}d({\rm Re}\,\frac{dw}{w} \circ \Psi_*K)$;
since $K$ is a smooth section of $T^*D \otimes T\Delta$, this equation tells us that
$i\partial\bar\partial\log|w|$ has zero horizontal, $O(|w|^{-2}|K| + |w|^{-1}|\partial_h K|  )$ mixed, and $O(|w|^{-1}|\partial_vK|)$ vertical components,
where $\partial_h$ and $\partial_v$ denote horizontal and vertical partials. It is therefore essentially forced on us that $K = O(|w|^2)$.

$\bullet$ Conversely, given $D \subset U$ and a defining function $w$, it is natural to try and construct an ACyl Hermitian metric on $U \setminus D$ by making an ansatz of the form $i\partial\bar{\partial}(\log |w|)^2 + \omega_0$ for some Hermitian metric $\omega_0$ on $U$. With a diffeomorphism $\Psi$ as in \ref{obs:holo_disks}, computations as above show that $K = O(|w|^2)$ then suffices in order for this ansatz to be ACyl with ACyl diffeomorphism $(t,\theta,x) \mapsto \Psi(e^{-t-i\theta},x)$; for instance, the purely vertical components of $i\partial\bar{\partial}(\log |w|)^2$ are $O(|w|^{-1}|{\log |w|}| |\partial_v K|)$.

\section{Error estimates for the nonfibred case of Theorem \ref{t:geom-exi}}\label{s:nonfibred_errors}

In this section we prove the estimates (\ref{result_epsilons}) for the integrals defined in  (\ref{define_epsilons}), using the auxiliary estimates (\ref{aux_epsilons}). We write the domain of integration as a union of two regions that will be treated separately:~the \emph{annulus} $(\Delta(r + 2s) \setminus \Delta(r - 2s)) \times D$ and the \emph{tube} $\Delta(r-2s) \times D$.
In each case, the integrand is a wedge product of $2$-forms with $n$ factors. We decompose each of these $2$-form factors into its horizontal ($\wedge^2T^*\Delta$), mixed ($T^*\Delta \otimes T^*D$), and vertical ($\wedge^2 T^*D$) components, estimates for which can be found in (\ref{aux_epsilons}). In addition to the absolute value bounds of (\ref{aux_epsilons}), we will also make use of the fact that $({\rm\textit{horizontal}}\,)^{\wedge a} \wedge ({\rm\textit{mixed}}\,)^{\wedge b} = 0$ if
$a \geq 2$ or $a = 1, b \geq 1$ or  $b \geq 3$.

Before estimating the errors $\epsilon_{\ell,p}$, let us quickly note the following bounds for the constants $c_0, c_1$ of (\ref{e:nonlinear_volume_form}) and (\ref{result_epsilons}), whose proofs are similar but much less complicated (see also (\ref{e:linear_volume_form})):
\begin{align}
c_0 = \int_M (\tilde\omega_0^n + n\lambda i\partial\bar\partial(\chi u) \wedge \tilde\omega_0^{n-1} - i^{n^2}\Omega\wedge\bar\Omega)\sim -|{\log r}|, \\
c_1 = \int_M ni\partial\bar\partial B \wedge \tilde\omega_0^{n-1} \sim rs.
\end{align}

We subdivide the remaining estimates into three cases. We abbreviate horizontal/mixed/vertical $2$-forms by $h/m/v$ respectively, and $v^p$ refers to a wedge product of $p$ vertical $2$-forms etc.

\subsection{Estimating $\epsilon_{0,p}$ for $p \in \{2,\ldots,n\}$} This is the easiest case because there are no $i\partial\bar\partial B$ factors. We have the following contributions to $\epsilon_{0,p}$, the crosses indicating the dominant ones.\vskip3mm

\begin{center}
\begin{tabular}{l|l|l|l}
\hline\hline
annulus & $v^p$ &  $rs(r|{\log r}|)^{p}$ & \\\hline
& $v^{p-1}m$ & $rs(r |{\log r}|)^{p-1}|{\log r}|$ & \\\hline
&$v^{p-1}h$ & $rs (r |{\log r}|)^{p-1}\frac{|{\log r}|}{r^2}$ & $\times$ \\\hline
& $v^{p-2}m^2$ & $rs(r |{\log r}|)^{p-2}|{\log r}|^2$ & \\\hline\hline
tube & $v^p$ & $\int_0^r \rho(\rho |{\log \rho}|)^{p}  \, d\rho $ & \\\hline
& $v^{p-1}m$ & $\int_0^r \rho(\rho |{\log \rho}|)^{p-1} |{\log \rho}| \, d\rho$ & \\\hline
& $v^{p-1}h$ & $\int_0^r \rho(\rho |{\log \rho}|)^{p-1} \frac{1}{\rho^2}\, d\rho$ & $\times$ \\\hline
& $v^{p-2}m^2$ & $\int_0^r \rho(\rho |{\log \rho}|)^{p-2} |{\log \rho}|^2 \, d\rho$ & \\\hline\hline\end{tabular}
\end{center}
\vskip3mm
It follows immediately that
\begin{equation}
\sum_{p=2}^n \epsilon_{0,p}  = O((r + s|{\log r}|)|{\log r}|).
\end{equation}

\subsection{Estimating $\epsilon_{1,p}$ for  $p \in \{1, \ldots , n-1\}$} The only nonzero contributions to the integrand arise by multiplying a component from the left column of the following table with a component from the right column labelled with the same colour.

\vskip3mm

\begin{center}
\begin{tabular}{l|l}
$i\partial\bar\partial B$ & $(i\partial\bar\partial(\chi u))^p$ \\\hline
$v$ {\color{blue}$\bullet$}  & $v^p$ {\color{blue}$\bullet$} {\color{red}$\bullet$}  {\color{green}$\bullet$}\\
$m$ {\color{red}$\bullet$} & $v^{p-1}m$ {\color{blue}$\bullet$} {\color{red}$\bullet$} \\
$h$ {\color{green}$\bullet$}& $v^{p-1}h$ {\color{blue}$\bullet$}    \\
$\;$ & $v^{p-2}m^2$ $($if $p \geq 2$) {\color{blue}$\bullet$}\\\hline
\end{tabular}
\end{center}

\vskip3mm

\noindent Then $\epsilon_{1,p}$ consists of the following contributions, the cross again indicating the largest one.

\vskip3mm

\begin{center}
\begin{tabular}{l|l|l|l}
\hline\hline
annulus & {\color{blue}$\bullet$} & $rs (r^2s) (r |{\log r}|)^p$ & \\
& & $rs (r^2s) (r |{\log r}|)^{p-1} |{\log r}| $ & \\
& & $rs (r^2s) (r |{\log r}|)^{p-1} \frac{|{\log r}|}{r^2}$ & \\
& & $rs (r^2s) (r |{\log r}|)^{p-2}|{\log r}|^2$ (if $p \geq 2$) & \\\hline
& {\color{red}$\bullet$}  & $rs (r^2)(r |{\log r}|)^p$ &  \\
& & $rs (r^2)(r |{\log r}|)^{p-1} |{\log r}|$ & \\\hline
& {\color{green}$\bullet$}  & $rs (r |{\log r}|)^p$ & $\times$ \\\hline\hline
tube & {\color{blue}$\bullet$}  & $\int_0^r \rho(rs\rho) (\rho|{\log \rho}|)^p \, d\rho$ & \\
& & $\int_0^r \rho(rs\rho) (\rho |{\log \rho}|)^{p-1} |{\log \rho}|\, d\rho$ & \\
& & $\int_0^r \rho(rs\rho)(\rho |{\log \rho}|)^{p-1}\frac{1}{\rho^2} \, d\rho$ & \\
& & $\int_0^r \rho(rs\rho)(\rho |{\log \rho}|)^{p-2}|{\log \rho}|^2\,d\rho$ (if $p \geq 2$) & \\\hline
& {\color{red}$\bullet$}  & $\int_0^r \rho(rs)(\rho |{\log \rho}|)^p\, d\rho$ &\\
& & $\int_0^r \rho(rs)(\rho |{\log \rho}|)^{p-1}|{\log \rho}|\,d\rho$ & \\\hline
& {\color{green}$\bullet$}  & $0$ & \\\hline\hline
\end{tabular}
\end{center}

\vskip3mm

\noindent As an immediate consequence,
\begin{equation}
\sum_{p=1}^{n-1} \epsilon_{1,p} = O( r|{\log r}| rs).
\end{equation}

\subsection{Estimating $\epsilon_{\ell,p}$ for $\ell \in \{2, \ldots, n\}$ and $p \in\{0, \ldots, n-\ell\}$} This step is entirely similar to the previous one, if slightly more complicated, so we only give the tables and the final result.

\vskip3mm

\begin{center}
\begin{tabular}{l|l}
$(i\partial\bar\partial B)^\ell$ & $(i\partial\bar\partial(\chi u))^p$\\\hline

$v^\ell$ {\color{blue}$\bullet$} & $v^p$ {\color{blue}$\bullet$} {\color{red}$\bullet$} {\color{green}$\bullet$} {\color{black}$\bullet$}\\
$v^{\ell-1}m$ {\color{red}$\bullet$} & $v^{p-1}m$ (if $p \geq 1$) {\color{blue}$\bullet$} {\color{red}$\bullet$}\\
$v^{\ell-1}h$ {\color{green}$\bullet$} & $v^{p-1}h$ (if $p \geq 1$) {\color{blue}$\bullet$}\\
$v^{\ell-2}m^2$ $\bullet$ & $v^{p-2}m^2$ (if $p \geq 2$) {\color{blue}$\bullet$}\\\hline
\end{tabular}
\end{center}

\vskip3mm

\begin{center}
\begin{tabular}{l|l|l|l}
\hline\hline
annulus &  {\color{blue}$\bullet$} & $rs (r^2s)^\ell(r|{\log r}|)^p$ &\\
& & $rs(r^2s)^{\ell}(r |{\log r}|)^{p-1}|{\log r}|$ (if $p \geq 1$)&\\
& & $rs(r^2s)^\ell(r |{\log r}|)^{p-1}\frac{|{\log r}|}{r^2}$ (if $p \geq 1$)&\\
& & $rs(r^2s)^\ell(r |{\log r}|)^{p-2}|{\log r}|^2$ (if $p \geq 2$)&\\\hline
&  {\color{red}$\bullet$} & $rs(r^2s)^{\ell-1}r^2(r |{\log r}|)^p$&\\
& & $rs(r^2s)^{\ell-1}r^2(r |{\log r}|)^{p-1}|{\log r}|$ (if $p \geq 1$)&\\\hline
&  {\color{green}$\bullet$} & $rs(r^2s)^{\ell-1} (r |{\log r}|)^p$ & $\times$ \\ \hline
&  {\color{black}$\bullet$} & $rs(r^2s)^{\ell-2} r^4 (r |{\log r}|)^p$ & $\times$ \\\hline\hline
tube &  {\color{blue}$\bullet$} & $\int_0^r \rho(rs\rho)^\ell (\rho |{\log \rho}|)^p \, d\rho$&\\
& & $\int_0^r \rho(rs\rho)^\ell (\rho |{\log \rho}|)^{p-1}|{\log \rho}|\, d\rho$ (if $p \geq 1$)&\\
& & $\int_0^r \rho(rs\rho)^\ell (\rho |{\log \rho}|)^{p-1} \frac{1}{\rho^2}\, d\rho$ (if $p \geq 1$)&\\
& & $\int_0^r \rho(rs\rho)^\ell (\rho |{\log \rho}|)^{p-2} |{\log \rho}|^2\, d\rho$ (if $p \geq 2$)&\\\hline
&  {\color{red}$\bullet$} & $\int_0^r \rho(rs\rho)^{\ell-1} rs (\rho |{\log \rho}|)^p \, d\rho$ &\\
& & $\int_0^r \rho(rs\rho)^{\ell-1} rs (\rho |{\log \rho}|)^{p-1} |{\log \rho}|\, d\rho$ (if $p \geq 1$)&\\\hline
& {\color{green}$\bullet$} & $0$&\\\hline
&  {\color{black}$\bullet$} & $\int_0^r \rho(rs\rho)^{\ell-2}(rs)^2(\rho |{\log \rho}|)^p \, d\rho$& \\\hline\hline
\end{tabular}
\end{center}

\vskip3mm

\begin{equation}
\sum_{p = 0}^{n-\ell} \epsilon_{\ell,p} = O((r^2s)^{\ell-1}(rs + r^3))
\end{equation}

\bibliographystyle{amsplain}
\bibliography{bib}

\providecommand{\bysame}{\leavevmode\hbox to3em{\hrulefill}\thinspace}
\providecommand{\MR}{\relax\ifhmode\unskip\space\fi MR }
\providecommand{\MRhref}[2]{%
  \href{http://www.ams.org/mathscinet-getitem?mr=#1}{#2}
}
\providecommand{\href}[2]{#2}
\begin{thebibliography}{10}

\bibitem{atiyah75}
M.F. Atiyah, V.K. Patodi, and I.M. Singer, \emph{Spectral asymmetry and
  {R}iemannian geometry, {I}}, Math. Proc. Camb. Phil. Soc. \textbf{77} (1975),
  97--118.

\bibitem{besse87}
A.L. Besse, \emph{Einstein manifolds}, Springer-Verlag, New York, 1987.

\bibitem{BM}
O.~Biquard and V.~Minerbe, \emph{A {K}ummer construction for gravitational
  instantons}, Comm. Math. Phys. \textbf{308} (2011), 773--794.

\bibitem{blocki}
Z.~B{\l}ocki, \emph{The {C}alabi-{Y}au theorem}, Complex Monge-Amp{\`e}re
  equations and geodesics in the space of K{\"a}hler metrics, Lecture Notes in
  Math., vol. 2038, Springer, 2012, pp.~201--227.

\bibitem{cheeger71}
J.~Cheeger and D.~Gromoll, \emph{The splitting theorem for manifolds of
  non-negative {R}icci curvature}, J. Diff. Geom. \textbf{6} (1971), 119--128.

\bibitem{cg1}
\bysame, \emph{On the structure of complete manifolds of nonnegative
  curvature}, Ann. of Math. (2) \textbf{96} (1972), 413--443.

\bibitem{cheeger:tian}
J.~Cheeger and G.~Tian, \emph{On the cone structure at infinity of {R}icci-flat
  manifolds with {E}uclidean volume growth and quadratic curvature decay},
  Invent. Math. \textbf{118} (1994), 493--571.

\bibitem{AC3}
R.J. Conlon and H.-J. Hein, \emph{Asymptotically conical {C}alabi-{Y}au
  manifolds, {III}}, arXiv:1405.7140.

\bibitem{CMR}
R.J. Conlon, R.~Mazzeo, and F.~Rochon, \emph{The moduli space of asymptotically
  cylindrical {C}alabi-{Y}au manifolds}, arXiv:1408.6562.

\bibitem{chnp1}
A.~Corti, M.~Haskins, J.~Nordström, and T.~Pacini, \emph{Asymptotically cylindrical {C}alabi-{Y}au $3$-folds from weak
  {F}ano $3$-folds}, Geom. Topol. \textbf{17} (2013), 1955--2059.

\bibitem{chnp2}
\bysame, \emph{${G}_2$-manifolds and
  associative submanifolds via semi-{F}ano $3$-folds}, arXiv:1207.4470, to
  appear in Duke Math. J.

\bibitem{fischer}
A.E. Fischer and J.A. Wolf, \emph{The structure of compact {R}icci-flat
  {R}iemannian manifolds}, J. Diff. Geom. \textbf{10} (1975), 277--288.

\bibitem{frantzenthesis}
K.~Frantzen, \emph{${K}3$ surfaces with special symmetry}, Ph.D. thesis,
  Ruhr-Universität Bochum, 2008.

\bibitem{fujiki}
A.~Fujiki, \emph{Finite automorphism groups of complex tori of dimension two},
  Publ. Res. Inst. Math. Sci. \textbf{24} (1988), 1--97.

\bibitem{greb13}
D.~{Greb} and C.~{Lehn}, \emph{Base manifolds for {L}agrangian fibrations on
  hyperk\"ahler manifolds}, IMRN (2014), 5483--5487.

\bibitem{gsc}
A.~Grigoryan and L.~Saloff-Coste, \emph{Stability results for {H}arnack
  inequalities}, Ann. Inst. Fourier \textbf{55} (2005), 825--890.

\bibitem{hanlin}
Q.~Han and F.~Lin, \emph{Elliptic partial differential equations}, Courant
  Lecture Notes in Mathematics, vol.~1, New York University Courant Institute
  of Mathematical Sciences, New York, 1997.

\bibitem{hein1}
H.-J. Hein, \emph{Weighted {S}obolev inequalities under lower {R}icci curvature
  bounds}, Proc. Amer. Math. Soc. \textbf{139} (2011), 2943--2955.

\bibitem{hein}
\bysame, \emph{Gravitational instantons from rational elliptic surfaces}, J.
  Amer. Math. Soc. \textbf{25} (2012), 355--393.

\bibitem{hwang}
J.-M. Hwang, \emph{Base manifolds for fibrations of projective irreducible
  symplectic manifolds}, Invent. Math. \textbf{174} (2008), 625--644.

\bibitem{janich}
K.~J\"anich, \emph{Vector analysis}, Undergraduate Texts in Mathematics,
  Springer-Verlag, New York, 2001.

\bibitem{joyce00}
D.D. Joyce, \emph{Compact manifolds with special holonomy}, OUP Mathematical
  Monographs series, Oxford University Press, 2000.

\bibitem{koz05}
J.~Keum, K.~Oguiso, and D.-Q. Zhang, \emph{The alternating group of degree $6$
  in the geometry of the {L}eech lattice and ${K}3$ surfaces}, Proc. London
  Math. Soc. \textbf{90} (2005), no.~3, 371--394.

\bibitem{kovalev03}
A.G. Kovalev, \emph{Twisted connected sums and special {R}iemannian holonomy},
  J. Reine Angew. Math. \textbf{565} (2003), 125--160.

\bibitem{kovalev05}
\bysame, \emph{Ricci-flat deformations of asymptotically cylindrical
  {C}alabi-{Y}au manifolds}, Proceedings of {G}ökova Geometry-Topology
  Conference 2005, International Press, 2006, pp.~140--156.

\bibitem{kovalev-lee08}
A.G. Kovalev and N.-H. Lee, \emph{{$K3$} surfaces with non-symplectic
  involution and compact irreducible {$G\sb 2$}-manifolds}, Math. Proc. Camb.
  Phil. Soc. \textbf{151} (2011), 193--218.

\bibitem{ChiLi}
C.~Li, \emph{Deformations of complex cones and neighborhoods of ample
  divisors}, arXiv:1405.2433.

\bibitem{lockhart85}
R.B. Lockhart and R.C. McOwen, \emph{Elliptic differential operators on
  noncompact manifolds}, Ann. Scuola Norm. Sup. Pisa Cl. Sci. \textbf{12}
  (1985), 409--447.

\bibitem{matsushita}
D.~Matsushita, \emph{On fibre space structures of a projective irreducible
  symplectic manifold}, Topology \textbf{38} (1999), 79--83.

\bibitem{mazya78}
V.G. Maz'ya and B.A. Plamenevski\u{\i}, \emph{Estimates in ${L}_{p}$ and
  {H}ölder classes and the {M}iranda-{A}gmon maximum principle for solutions
  of elliptic boundary value problems in domains with singular points on the
  boundary}, Math. Nachr. \textbf{81} (1978), 25--82, English translation: {\it
  Amer. Math. Soc. Transl. Ser. 2}, 123:1--56, 1984.

\bibitem{melrose94}
R.~Melrose, \emph{The {A}tiyah-{P}atodi-{S}inger index theorem}, AK Peters,
  Wellesley, MA, 1994.

\bibitem{min}
V.~Minerbe, \emph{Weighted {S}obolev inequalities and {R}icci-flat manifolds},
  Geom. Funct. Anal. \textbf{18} (2009), 1696--1749.

\bibitem{mok}
N.~Mok, \emph{An embedding theorem of complete {K}\"ahler manifolds of positive
  {R}icci curvature onto quasi-projective varieties}, Math. Ann. \textbf{286}
  (1990), 373--408.

\bibitem{mok-zhong}
N.~Mok and J.Q. Zhong, \emph{Compactifying complete {K}\"ahler-{E}instein
  manifolds of finite topological type and bounded curvature}, Ann. of Math.
  (2) \textbf{129} (1989), 427--470.

\bibitem{mukai:finite:k3}
S.~Mukai, \emph{Finite groups of automorphisms of {$K3$} surfaces and the
  {M}athieu group}, Invent. Math. \textbf{94} (1988), 183--221.

\bibitem{nij-woolf}
A.~Nijenhuis and W.~Woolf, \emph{Some integration problems in almost-complex
  and complex manifolds}, Ann. of Math. \textbf{77} (1963), 424--489.

\bibitem{nikulin80}
V.V. Nikulin, \emph{Finite groups of automorphisms of {K}ählerian surfaces of
  type ${K}3$ ({E}nglish translation)}, Moscow Math. Soc. \textbf{38} (1980),
  71--137.

\bibitem{jnthesis}
J.~Nordstr\"om, \emph{Deformations and gluing of asymptotically cylindrical
  manifolds with exceptional holonomy}, Ph.D. thesis, University of Cambridge,
  2008.

\bibitem{jn1}
\bysame, \emph{Deformations of asymptotically cylindrical ${G}_{2}$-manifolds},
  Math. Proc. Camb. Phil. Soc. \textbf{145} (2008), 311--348.

\bibitem{P}
T.~Perutz, \emph{Hamiltonian handleslides for {H}eegaard {F}loer homology},
  Proceedings of G\"okova Geometry-Topology Conference 2007, G\"okova, 2008,
  pp.~15--35.

\bibitem{siu}
Y.-T. Siu, \emph{The $\bar{\partial}$-problem with uniform bounds on
  derivatives}, Math. Ann. \textbf{207} (1974), 163--176.

\bibitem{sor}
C.~Sormani, \emph{Busemann functions on manifolds with lower bounds on {R}icci
  curvature and minimal volume growth}, J. Diff. Geom. \textbf{48} (1998),
  557--585.

\bibitem{tianyau90}
G.~Tian and S.-T. Yau, \emph{Complete {K}\"ahler manifolds with zero {R}icci
  curvature, {I}}, J. Amer. Math. Soc. \textbf{3} (1990), 579--610.

\bibitem{tianyau91}
\bysame, \emph{Complete {K}\"ahler manifolds with zero {R}icci curvature,
  {II}}, Invent. Math. \textbf{106} (1991), 27--60.

\bibitem{V}
J.~Varouchas, \emph{Stabilit{\'e} de la classe des vari\'et\'es
  k\"ahl\'eriennes par certains morphismes propres}, Invent. Math. \textbf{77}
  (1984), 117--127.

\bibitem{wang89}
M.Y. Wang, \emph{Parallel spinors and parallel forms}, Ann. Global Anal. Geom.
  \textbf{7} (1989), 59--68.

\bibitem{yau78}
S.-T. Yau, \emph{On the {R}icci curvature of a compact {K}ähler manifold and
  the complex {M}onge-{A}mp\`{e}re equation, {I}}, Comm. Pure Appl. Math
  \textbf{31} (1978), 339--411.

\bibitem{yau:ICM78}
\bysame, \emph{The role of partial differential equations in differential
  geometry}, Proceedings of the International Congress of Mathematicians
  (Helsinki, 1978), Acad. Sci. Fennica, 1980, pp.~237--250.

\end{thebibliography}

\end{document}